\documentclass[]{siamart}

\usepackage{braket,amsfonts}

\usepackage{array}

\usepackage[caption=false]{subfig}

\usepackage{pgfplots}

\usepackage{fancyvrb,cprotect}
\usepackage{mathtools}

\newsiamthm{claim}{Claim}
\newsiamremark{rem}{Remark}
\newsiamremark{expl}{Example}
\newsiamremark{hypothesis}{Hypothesis}
\crefname{hypothesis}{Hypothesis}{Hypotheses}

\usepackage{graphicx,epstopdf}

\Crefname{ALC@unique}{Line}{Lines}

\usepackage{amsopn,amssymb}

\SetCommentSty{mycommfont}
\usepackage{graphicx,booktabs}
\usepackage{subfig}

\usepackage{xspace}
\usepackage{bold-extra}

\colorlet{texcscolor}{blue!50!black}
\colorlet{texemcolor}{red!70!black}
\colorlet{texpreamble}{red!70!black}
\colorlet{codebackground}{black!25!white!25}

\title{FAST HIERARCHICAL SOLVERS FOR SPARSE MATRICES USING EXTENDED SPARSIFICATION AND LOW-RANK APPROXIMATION%
  \thanks{Funding from the ``Army High Performance Computing Research Center'' (AHPCRC), sponsored by the U.S. Army Research Laboratory under contract No. W911NF-07-2-0027, at Stanford, supported in part this research. The second author is a post-doctoral fellow of the Research Foundation Flanders (FWO) and a Francqui Foundation fellow of the Belgian American Educational Foundation (BAEF). The financial support is gratefully acknowledged.}}

\author{HADI POURANSARI\footnotemark[2]\and PIETER COULIER\footnotemark[2] \footnotemark[3] \and ERIC DARVE\footnotemark[2] \footnotemark[4]}

\headers{FAST HIERARCHICAL SPARSE SOLVERS}
{HADI POURANSARI, PIETER COULIER\and ERIC DARVE}

\begin{document}
\maketitle

\renewcommand{\thefootnote}{\fnsymbol{footnote}}

\footnotetext[2]{Stanford University, Department of Mechanical Engineering, Stanford, CA 94305, USA (\email{[hadip,pcoulier,darve]@stanford.edu}).}
\footnotetext[3]{KU Leuven, Department of Civil Engineering, Kasteelpark Arenberg 40, 3001 Leuven, Belgium.}
\footnotetext[4]{Stanford University, Institute for Computational and Mathematical Engineering, Stanford, CA 94305, USA.}

\renewcommand{\thefootnote}{\arabic{footnote}}

\newcommand{\red}[1]{\textcolor{black}{#1}}
\newcommand{\blue}[1]{\textcolor{black}{#1}}
\newcommand{\orange}[1]{\textcolor{black}{#1}}

\newcommand{\x}{\texttt{Var}}
\newcommand{\rhs}{\texttt{RHS}}
\newcommand{\mat}{\texttt{Mat}}

\begin{abstract}
 Inversion of sparse matrices with standard direct solve schemes \red{ is robust, but computationally expensive}. Iterative solvers, on the other hand, demonstrate better scalability, \red{but need to be used with an appropriate preconditioner (e.g.,  ILU, AMG, Gauss-Seidel, etc.) for proper convergence.} The choice of an effective preconditioner is highly problem dependent. We propose a novel fully algebraic sparse matrix solve algorithm, which has linear complexity with the problem size. Our scheme is based on the Gauss elimination. For a given matrix, we approximate the LU factorization with a tunable accuracy determined a priori. This method can be used as a stand-alone direct solver with linear complexity and tunable accuracy, or it can be used as a black-box preconditioner in conjunction with iterative methods such as GMRES.
The proposed solver is based on the low-rank approximation of fill-ins generated during the elimination. Similar to $\mathcal{H}$-matrices, fill-ins corresponding to blocks that are well-separated in the adjacency graph are represented via a hierarchical structure. \red{The linear complexity of the algorithm is guaranteed if the blocks corresponding to well-separated clusters of variables are numerically low-rank.}
\end{abstract}


\slugger{sisc}{xxxx}{xx}{x}{x--x}



\section {Introduction}
In the realm of scientific computing, solving a sparse linear system,
\begin{equation}
A {\boldsymbol{x}} = {\boldsymbol{b}},
\label{eqn:linsys}
\end{equation}
is known to be one of the challenging parts of \red{many} calculations, and is \red{often} the main bottleneck. Such a system of equations may be the result of the discretization of some partial differential equation (PDE), or more generally, can represent the local interactions of units in a network.

Solving a system of equations of size $n$ using a naive implementation of Gauss elimination has $\mathcal{O}(n^{3})$ time complexity. The best proved time complexity to solve a general linear system is $\mathcal{O}(n^{\omega})$, where $\omega<2.376$ \cite{bunch1974triangular,coppersmith1987matrix,ibarra1982generalization}. In the case of sparse matrices, the time and memory complexity can be reduced when a proper elimination order is employed. Finding the optimal ordering (that results in the minimum number of new non-zeros in the LU factorization) is known to be an NP-complete problem \cite{yannakakis1981computing}. For matrices resulting from the discretization of some PDE in physical space, nested dissection \cite{george1973nested, lipton1979generalized} is known as an efficient elimination strategy. \cite{alon2010solving} discusses the complexity of nested dissection based on the sparsity pattern of the matrix. For a three-dimensional problem, the time and memory complexities are expected to be $\mathcal{O}(n^2)$ and $\mathcal{O}(n^{4/3})$, respectively, when nested dissection is employed. As the size of the problem grows, such complexities make direct solvers prohibitive.

Iterative methods, such as conjugate gradient \cite{hestenes1952methods}, minimum residual \cite{paige1975solution}, and general minimum residual \cite{saad1986gmres}, are generally more time and memory efficient. In addition, iterative solvers such as those based on Krylov subspace can be accelerated using fast linear algebra techniques. The fast multipole method (FMM) \cite{darve2000fast,fong2009black,greengard1987fast, nishimura2002fast,pouransari2015optimizing,ying2004kernel}, for example, can accelerate matrix-vector multiplication ---from quadratic complexity to linear---, which is the bulk calculation in iterative solvers based on Krylov subspace. However, in practice, iterative methods need to be used in conjunction with preconditioners to limit the number of iterations. The choice of an efficient preconditioner is highly problem dependent. There are many ongoing efforts to develop preconditioners that are optimized for particular applications. Hence, there is a need for general purpose preconditioners. Hierarchical matrices enable us to develop such preconditioners.

FMM matrices are a subclass of a larger category of matrices called hierarchical matrices ($\mathcal{H}$-matrices) \cite{bebendorf2008hierarchical,borm2003introduction,hackbusch2002h2}. $\mathcal{H}$-matrices have a hierarchical low-rank structure. For instance, in a hierarchically off-diagonal low-rank (HODLR) matrix \red{\cite{aminfar2016fast}}, off-diagonal blocks can be represented through a hierarchy of low-rank interactions. If the bases used in the hierarchy are nested (i.e., the low-rank basis at each level is constructed using the low-rank basis of the child level) the method is called hierarchically semi-separable (HSS) \cite{ambikasaran2013mathcal,chandrasekaran2006fast, xia2010fast}. In a more general case of hierarchical matrices, more complex low-rank structures can be considered. A full\orange{-rank} dense matrix with many low-rank structures is in fact data-sparse \cite{hackbusch1999sparse,hackbusch2000sparse}. A data-sparse matrix can be represented  via an extended sparse matrix, \orange{which has extra $\mathcal{O}(n)$ rows/columns, but with only few non-zero entries} \cite{ambikasaran2013fast}. The hierarchical structure of such matrices can be used for efficient calculation and storage. 

\red{Recently, hierarchical interpolative factorization \cite{ho2015hierarchicalDiff, ho2015hierarchicalInt} was proposed, which can be used to directly solve systems obtained from differential and integral equations based on elliptic operators. The fast factorization is obtained by skeletonization of fronts in the multifrontal scheme. Using low-rank structure of the off-diagonal blocks to develop fast direct solvers for linear systems arising from integral equations has been widely studied \cite{corona2015n,gillman2012direct,greengard2009fast,kong2011adaptive}. \cite{gillman2014direct} proposed a direct solver for \orange{elliptic} PDEs with variable coefficients on two-dimensional domains by exploiting internal low-rank structures in the matrices. \cite{napov2016algebraic} used hierarchical low-rank structures of the off-diagonal blocks to introduce a preconditioner for sparse matrices based on a multifrontal variant of sparse LU factorization. \cite{oseledets2012solution} introduced a black-box linear solver using tensor-train format. \cite{li2013divide} used a recursive low-rank approximation algorithm based on the Sherman-Morrison formula to obtain a preconditioner for symmetric sparse matrices.}

Sparse matrices can be considered as a very special case of hierarchical matrices, where instead of low-rank blocks they initially have zero blocks. However, during the elimination process in a direct solve scheme, many of the zero blocks get filled. For a large category of matrices, including those obtained from the discretization of PDEs, most of the new fill-ins are \red{numerically} low-rank. This is justified when the Green's function associated to the PDE is smooth \blue{(non-oscillatory)}. In this paper, we will use the $\mathcal{H}$-matrix structure to compress the fill-ins. A similar process can be applied in the elimination of an extended sparse matrix resulting from an originally dense matrix \cite{ambikasaran2014inverse, coulier2015inverse}. This reduces the complexity of the direct solver to linear. \red{The linear complexity of the method is guaranteed if the blocks corresponding to the interaction of well-separated nodes are numerically low-rank. We define the well-separated condition in \cref{sec:HTree}.}

\red{The proposed algorithm can be considered as} an extension to the block incomplete LU (ILU) \cite{saad1994ilut} preconditioners. In a block ILU factorization, most of the new fill-ins (i.e., blocks that are created during the elimination process which are originally zero) are ignored, and therefore, the block sparsity of the matrix is preserved, while the accuracy is not. In the proposed algorithm, instead, we use low-rank approximations to compress new fill-ins. Using a tree structure, new fill-ins at the fine level are compressed and pushed to the parent (coarse) level. The elimination \red{and} compression \red{processes are} done in a bottom-to-top traversal. 

In addition, the proposed algorithm has formal similarities with algebraic multi-grid (AMG) methods \cite{brandt1986algebraic, brandt1985algebraic,ruge1987algebraic, stuben2001review}. However, the two methods differ in the way they build the coarse system, and use restriction and prolongation operators. In AMG, the original system is solved at different levels (from fine to coarse). Here, the compressed fill-ins ---corresponding to the Schur complements--- of each level are solved at the coarser level above. Note that the proposed algorithm is purely algebraic, similar to AMG. If the matrix comes from discretization of a PDE on a physical grid, the grid information can be exploited to improve the performance of the solver, similar to geometric multi-grid.

The algorithm presented in this paper computes \blue{a hierarchical representation of the LU factorization of a sparse matrix using low-rank approximations}. 
\red{We introduce intermediate operations to compress new fill-ins. The compressed fill-ins are represented using a set of extra variables. This technique is known as extended sparsification \cite{chandrasekaran2006fast2}.}
The accuracy of the factorization phase \red{(i.e., Gauss elimination and compression)}, $\epsilon$, can be determined a priori. The time and memory complexity of the factorization are $\mathcal{O} \left(n\log^2 {1}/{\epsilon} \right)$ and $\mathcal{O} \left( n \log {1}/{\epsilon}\right)$, respectively, as will be clarified in \cref{sec:standalone}.

\red{The method presented in this paper is similar to the fast hierarchical methods developed by Hackbusch et al. \cite{borm2003introduction,hackbusch1999sparse,hackbusch2002h2,hackbusch2000sparse} in a sense that both methods use a tree decomposition to identify and represent low-rank blocks. The key difference, however, is that in the Hackbusch's algorithm the LU factorization is computed using a depth-first tree traversal order, whereas here we use a breadth-first (level by level from leaf to root) traversal. The connection and differences of the proposed algorithm and Hackbusch's fast $\mathcal{H}$-algebra is discussed in our companion paper \cite{coulier2015inverse}, in which a similar method is used for dense matrix factorization.}

Our solver can be used as a stand-alone direct solver with tunable accuracy. The factorization part is completely separate from the solve part and is generally more expensive. This makes the algorithm appealing \red{when multiple right hand sides are available (e.g., using the proposed solver as a black-box preconditioner in an iterative method)}. We have implemented the algorithm in C++ (the code can be downloaded from \url{bitbucket.org/hadip/lorasp}), and benchmarked it as both a stand-alone solver (see \cref{sec:standalone}), and a preconditioner in conjunction with the generalized minimum residual (GMRES) iterative solver \cite{saad1986gmres} (see \cref{sec:precond}).

Furthermore, the proposed algorithm has interesting parallelization properties. On one hand, all calculations are block matrix computations which can be highly accelerated using BLAS3 operations \cite{anderson1999lapack}. On the other hand, since the sparsity pattern at every level is preserved, the data dependency is very local, which is an interesting property to reduce the amount of communications. In addition, the amount of calculation scales with the third power of the size of blocks, while the communications scales with the second power of block sizes. This helps with the concurrency of the parallel implementation. Moreover, the order of elimination does not change the complexity of the presented algorithm. This is in particular an appealing property for parallel implementation. \red{The parallel implementation of the proposed method is not further discussed in this paper.}

The remainder of this paper is organized as follows. In \cref{sec:SLS} we briefly introduce a graph representation of sparse matrices, and an interpretation of the Gauss elimination using the adjacency graph. In \cref{sec:HTree} some concepts related to the hierarchical representation of matrices are defined. The algorithm is explained in \cref{sec:alg} in detail, and the linear complexity analysis is provided in \cref{sec:lin}. We present numerical results obtained from various benchmarks in \cref{sec:numerical}. There are many avenues for optimization and extension of the algorithm. We discuss some of these opportunities in \cref{sec:conclusion}.

\section {Sparse linear systems} \label{sec:SLS}
In this section we briefly introduce the graphical framework that is required in the rest of the paper. We assume a sparse linear system \red{of size $n$} as in \cref{eqn:linsys} is given.

\subsection{Adjacency graph} \label{sec:adjgraph}
In many algorithms, including the method proposed in this paper, it is necessary (or more efficient) to operate on sub-blocks of the matrix rather than single elements.
\red{The blocks of the matrix can be identified using a partitioning as defined below.}
\red{\begin{definition} (partitioning) 
A partitioning $\mathcal{P}$ is defined as a \blue{surjective} map $\{1,  \ldots, n\} \to \{1, \ldots, \blue{n_{\mathcal{P}}}\}$. 
\blue{$\mathcal{P}$ groups rows/columns of $A$ into $n_{\mathcal{P}}$ clusters, $C_j^{\mathcal{P}} := \{k \in \{1, \ldots, n \} | \mathcal{P}(k) = j\}$ for $1 \le j \le n_{\mathcal{P}}$.}
\end{definition}}

\blue{We \orange{denote}  an entry of matrix $A$ located in row $k$ and column $t$ by $A_{[k,t]}$. For $1 \le i,j \le n_{\mathcal{P}}$, we use $A_{i,j}$ to represent a sub-matrix formed by concatenating all entries $A_{[k,t]}$ such that $k \in C_i^{\mathcal{P}}$ and $t \in C_j^{\mathcal{P}}$. Additionally, for a vector $\boldsymbol{x}$ of size $n$, we use $\boldsymbol{x}_i$ to represent a sub-vector formed by concatenating all $\boldsymbol{x}_{[k]}$ entries such that $k \in C_i^{\mathcal{P}}$.}

It is often fruitful to represent sparse \blue{matrices} using graphs. \red{An adjacency graph, as defined below, represents a sparse matrix with partitioning.}
\red{\begin{definition} (adjacency graph)  \label{def:adjGraph}
A sparse matrix $A$ with a partitioning $\mathcal{P}$ can be represented by \blue{its adjacency} graph $G(V,E)$, where $V = \{v_1, \ldots, \blue{v_{n_{\mathcal{P}}}} \}$. \blue{Each $v_i \in V$  for $1 \le i \le \blue{n_{\mathcal{P}}}$ represents a cluster $C_i^{\mathcal{P}}$ of rows and columns of $A$.}
A vertex $v_i$ is connected to a vertex $v_j$ by a directed edge $e_{v_i \to v_j} \in E$ if and only if the block $A_{j,i}$ in $A$ is non-zero.\footnote{\red{The adjacency graph of a matrix $A$ with a partitioning $\mathcal{P}$ is essentially the \emph{quotient graph} of the adjacency graph of matrix $A$ with identity partitioning, where the equivalence relation is induced by partitioning $\mathcal{P}$.}}
\end{definition}}

\red{In \cref{fig:domain}, an example of the adjacency graph is illustrated. In the rest of the paper} \blue{we use vertex and node interchangeably for the elements of $V$ in the adjacency graph.}

\blue{The linear system in \cref{eqn:linsys} can also be represented using the adjacency graph of $A$, $G(V,E)$.
For a node $v_i \in V$, $\x(v_i) = \boldsymbol{x}_i$ denotes the vector of variables corresponding to cluster $C_i^{\mathcal{P}}$. Similarly, $\rhs(v_i) = \boldsymbol{b}_i$ denotes the vector of right hand sides corresponding to cluster $C_i^{\mathcal{P}}$. Also for an edge $e_{v_i \to v_j} \in E$, $\mat(e_{v_i \to v_j}) = A_{j,i}$ denotes the sub-matrix corresponding to cluster $C_i^{\mathcal{P}}$ of columns and cluster $C_j^{\mathcal{P}}$ of rows. For the example shown in \cref{fig:domain}, the following two notations represent the \orange{same} set of equations corresponding to the node $v_2$ \orange{and its incoming edges}.}
\blue{
\begin{subequations}
\begin{equation}
A_{2,1} \boldsymbol{x}_1 + A_{2,2} \boldsymbol{x}_2 + A_{2,3} \boldsymbol{x}_3 + A_{2,4} \boldsymbol{x}_4 = \boldsymbol{b}_2
\end{equation}
\begin{equation}
\begin{split}
\mat(e_{v_1 \to v_2}) \cdot \x (v_1) + \mat(e_{v_2 \to v_2}) \cdot \x (v_2) +\\ \mat(e_{v_3 \to v_2}) \cdot \x (v_3) + \mat(e_{v_4 \to v_2}) \cdot \x (v_4) &= \rhs (v_2)
\end{split}
\end{equation}
\end{subequations}
}

\begin{figure}[htbp] \centering
\includegraphics[width=0.7\textwidth]{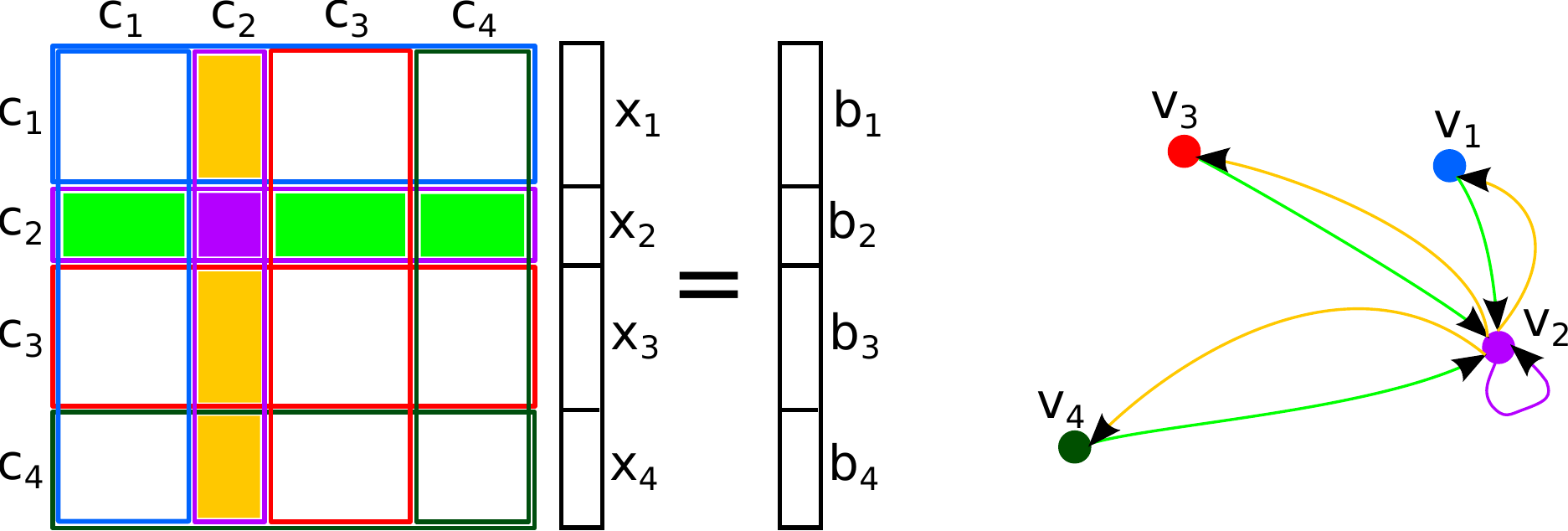}
\caption{\blue{Example of the adjacency graph (right) of a matrix (left). Vertices' colors are the same as their corresponding cluster of rows/columns in the matrix. Edges' colors are also in correspondence with the sub-blocks in the matrix.}}
\label{fig:domain} 
\end{figure}

\subsection{Elimination} \label{sec:elimination}
The \blue{block} Gauss elimination process, or \blue{block} LU factorization, can also be explained using the graph representation of the matrix. At \blue{step $i$} of the elimination process, a set of unknowns, \blue{$\x(v_i)$}, is eliminated from the system of equations. This corresponds to eliminating vertex $v_i$ from the adjacency graph $G(V,E)$. The self-edge from $v_i$ to itself corresponds to the pivot diagonal sub-block in the matrix. After eliminating $v_i$, for every pair of outgoing edge $e_{v_i \to v_j}$ to a vertex $v_j$ and incoming edge $e_{v_k \to v_i}$ from a vertex $v_k$, a new edge from $v_k$ to $v_j$ is created, corresponding to the Schur complement of the eliminated edges, that is
\blue{\begin{equation}
-\mat(e_{v_i \to v_j}) \cdot \mat(e_{v_i \to v_i})^{-1}  \cdot\mat(e_{v_k \to v_i}) = -A_{j,i} A_{i,i}^{-1} A_{i,k}
\end{equation}}
Note that if the edge between $v_k$ and $v_j$ exists before elimination, the Schur complement adds to the existing sub-block.

The process described above reveals the fact that during the elimination process many new edges are introduced in the graph. This corresponds to generating new non-zero blocks in the matrix during the LU factorization. The generation of many dense blocks is what makes the direct factorization of sparse matrices a prohibitive process. Essentially, a matrix $A$ can be sparse, while $L$ and $U$ in the LU factorization of $A$ are dense. In the next section, we explain how we can preserve the sparsity of the matrix during the elimination process \blue{by compressing} \red{the well-separated interactions}. \blue{This process is known as extended sparsification}.

\subsection{Key idea} \label{sec:keyIdea}
\red{An important} observation in the elimination process is the fact that fill-ins (i.e., new edges created during the elimination process) that correspond to well-separated vertices are \red{often numerically} low-rank. \red{For a linear system obtained from a discretized PDE, \blue{well-separated vertices} refers to \blue{points} that are physically far enough from each other. 
For a general sparse matrix two vertices are well-separated if their distance in the adjacency graph is large enough. It is formally defined in \cref{def:wellSep}. We replace such fill-ins with a sequence of low-rank \blue{matrices}.}

\red{
For example, consider the following symmetric linear system that is partitioned into 3 blocks
\begin{equation} \label{eqn:example}
\begin{pmatrix}
S & B & C\\
B^{\intercal} & P & \\
C^{\intercal} & &Q\\
\end{pmatrix}
\begin{pmatrix}
\boldsymbol{x}_1\\
\boldsymbol{x}_2\\
\boldsymbol{x}_3
\end{pmatrix}
=
\begin{pmatrix}
\boldsymbol{b}_1\\
\boldsymbol{b}_2\\
\boldsymbol{b}_3
\end{pmatrix}
\end{equation}}

\begin{figure}[tbhp]
\begin{center}
\subfloat[]{\includegraphics[width=0.3\textwidth]{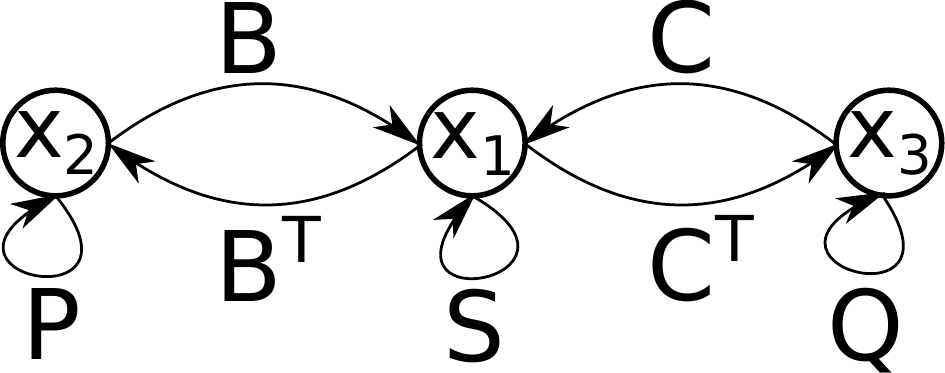}\label{fig:keyIdea:sub1}}
\hspace{3mm}
\subfloat[]{\includegraphics[width=0.3\textwidth]{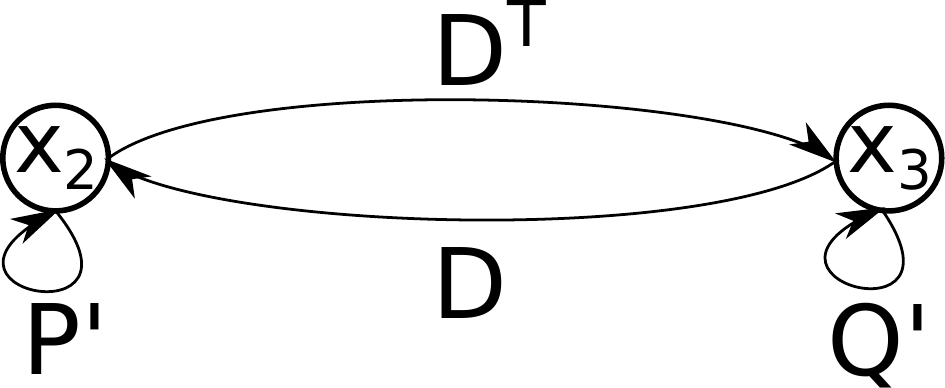}\label{fig:keyIdea:sub2}}
\hspace{3mm}
\subfloat[]{\includegraphics[width=0.3\textwidth]{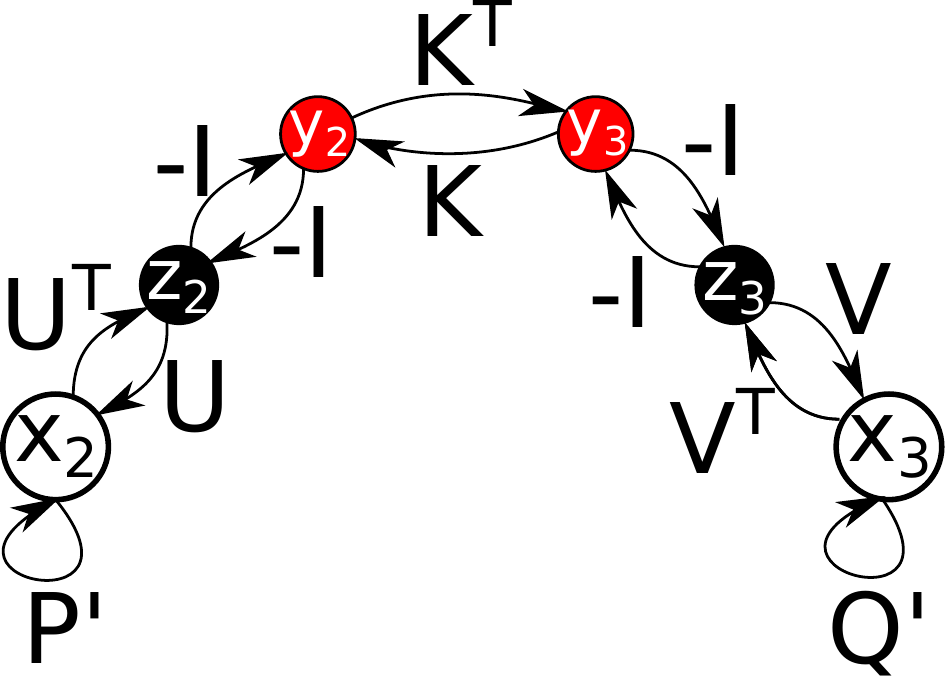}\label{fig:keyIdea:sub3}}
\caption{\red{An example of compression of a well-separated interaction. (a) The adjacency graph of the original linear system described in \cref{eqn:example}. (b) The resulting graph after eliminating $\boldsymbol{x}_1$ node. (c) The adjacency graph of the extended system. All edges are labeled with their corresponding block in the matrix.}}
\label{fig:keyIdea}
\end{center}
\end{figure}

\red{
In \cref{fig:keyIdea:sub1} the adjacency graph of the system of equations in \cref{eqn:example} is shown. Now, consider eliminating $\boldsymbol{x}_1$. The resulting system is as follows
\begin{equation} \label{eqn:example1}
\begin{pmatrix}
P' & D\\
D^{\intercal} & Q' \\
\end{pmatrix}
\begin{pmatrix}
\boldsymbol{x}_2\\
\boldsymbol{x}_3
\end{pmatrix}
=
\begin{pmatrix}
\boldsymbol{b}'_2\\
\boldsymbol{b}'_3
\end{pmatrix},
\end{equation}
where $D = -B^{\intercal}S^{-1}C$, $P' = P - B^{\intercal}S^{-1}B$,  $Q' = Q - C^{\intercal}S^{-1}C$, $\boldsymbol{b}'_2 = \boldsymbol{b}_2 - B^{\intercal}S^{-1} \boldsymbol{b}_1$, and $\boldsymbol{b}'_3 = \boldsymbol{b}_3 - C^{\intercal}S^{-1} \boldsymbol{b}_1$. The adjacency graph of the system of equations in \cref{eqn:example1} is depicted in \cref{fig:keyIdea:sub2}.
Nodes 2 and 3 can be considered well-separated \orange{(see \cref{def:wellSep})}. They get connected due to the elimination of node 1. We assume their interaction is low-rank, and can be written as
\begin{equation}\label{eqn:lowrank}
D \simeq U K V^{\intercal},
\end{equation}
where $U$ and $V$ are tall matrices. We can use any low-rank approximation in \cref{eqn:lowrank}, for example the singular value decomposition (SVD).
We combine \cref{eqn:example1,eqn:lowrank} to define a new set of equations, in which direct interaction of the nodes 2 and 3 is replaced by a sequence of low-rank interactions
\begin{equation} \label{eqn:extend}
\begin{pmatrix}
P'&&U&&&\\
&Q'&&V&&\\
U^{\intercal}&&&&-I&\\
&V^{\intercal}&&&&-I\\
&&-I&&&K\\
&&&-I&K^{\intercal}&\\
\end{pmatrix}
\begin{pmatrix}
\boldsymbol{x}_2\\
\boldsymbol{x}_3\\
\boldsymbol{z}_2\\
\boldsymbol{z}_3\\
\boldsymbol{y}_2\\
\boldsymbol{y}_3
\end{pmatrix}
=
\begin{pmatrix}
\boldsymbol{b}'_2\\
\boldsymbol{b}'_3\\
\boldsymbol{0}\\
\boldsymbol{0}\\
\boldsymbol{0}\\
\boldsymbol{0}
\end{pmatrix}
\end{equation}}

\red{In the above equation, we introduced extra variables ($\boldsymbol{y}_2, \boldsymbol{y}_3, \boldsymbol{z}_2$, and $\boldsymbol{z}_3$) to represent the far-field interactions. This technique is known as extended sparsification \cite{chandrasekaran2006fast2}. Note that the system of equations \cref{eqn:extend} is equivalent to the system of equations \cref{eqn:example1} up to the accuracy of \cref{eqn:lowrank}. In fact, if we do a Gaussian elimination on the matrix in \cref{eqn:extend} using $2\times2$ blocks starting from the right, we get back to the original matrix in \cref{eqn:example1}.}

\red{In an analogy with the fast multipole method, $\boldsymbol{y}_2$ and $\boldsymbol{y}_3$ can be considered as the multipole coefficients, and $\boldsymbol{z}_2$ and $\boldsymbol{z}_3$ as the local coefficients. In \cref{fig:keyIdea:sub3} the adjacency graph of the extended linear system is depicted. The red and black nodes correspond to the} \blue{extended part of the matrix in \cref{eqn:extend} (extra variables introduced above).} \blue{As a result of extended sparsification, nodes 2 and 3 are disconnected in \cref{fig:keyIdea:sub3}. \orange{Therefore,  in the general case $\boldsymbol{x}_2$ could be eliminated without introducing new fill-ins between $\boldsymbol{x}_3$ and other neighbors of $\boldsymbol{x}_2$.}\footnote{\blue{Note that in this example, we intentionally chose order of elimination $\boldsymbol{x}_1, \boldsymbol{x}_2, \boldsymbol{x}_3$, to introduce a new fill-in, and demonstrate the low-rank compression process. Using $\boldsymbol{x}_2, \boldsymbol{x}_1, \boldsymbol{x}_3$, as the order of elimination would result in no fill-ins for this example. Finding such an order of elimination, however, is a hard problem and may not be possible for practical matrices.}}}
\orange{As shown In \cref{lem:sparse}, removing well-separated interactions of nodes before elimination preserves the sparsity.}

\blue{In the general case, we start by a standard block Gauss elimination. When we create fill-ins, similar to the above example, we apply the extended sparsification (compression). This results in the creation of new nodes in the graph belonging to a coarser level.}

\blue{As we proceed with the elimination process at level $i$, edges between the auxiliary nodes at level $i-1$ are created. After the elimination process at level $i$ is completed, we proceed with the elimination at level $i-1$, and continue up to the root (level 0). Essentially, in this process we form a hierarchical approximation of $L$ and $U$ matrices (which are generally dense) in the LU factorization of $A$, without directly computing them.}

\blue{In addition, similar to the agglomeration process in multi-grid methods, we consider one red-node and one black-node for every pair of clusters (i.e., the well-separated interactions of each pair of clusters are compressed together). Therefore, the number of red-nodes at the parent level is half of the number of red-nodes at the current level.
We formally define the hierarchy of nodes in \cref{sec:HTree}, and the details of the algorithm in \cref{sec:alg}.}

\section{Hierarchical representation}\label{sec:HTree}
To form a hierarchical tree, \red{we recursively partition \orange{the rows/columns indices of the matrix}. This is formally defined as} a sequence of nested partitionings.

\begin{definition} (nested partitionings)
\blue{A sequence of partitionings $\mathcal{P}_0, \ldots , \mathcal{P}_l$ with $n_{{\mathcal{P}_i}} = 2^i$ is called nested if every cluster $C_j^{\mathcal{P}_i}$ is the union of two clusters $C_{j'}^{\mathcal{P}_{i+1}}$ and $C_{j''}^{\mathcal{P}_{i+1}}$ for $0 \le i < l$. Cluster $C_j^{\mathcal{P}_i}$ is then called the parent of two child clusters $C_{j'}^{\mathcal{P}_{i+1}}$ and $C_{j''}^{\mathcal{P}_{i+1}}$. We call $C_{j'}^{\mathcal{P}_{i+1}}$ and $C_{j''}^{\mathcal{P}_{i+1}}$ sibling clusters. The level of a cluster is defined as the index $i$ of its defining partitioning $\mathcal{P}_i$. The clusters associated to the finest partitioning, $\mathcal{P}_l$, have no children and are called leaf clusters, while the cluster associated to $\mathcal{P}_0$ has no parent and is called the root cluster. All other clusters have exactly two children and one parent.}\footnote{The binary subdivision is not necessary. We can generalize this to quad-tree, octree, etc.}
\end{definition}
\red{A visual example of nested partitionings with six levels (i.e., $\mathcal{P}_0, \ldots, \mathcal{P}_5$) is illustrated in \cref{sec:appB}.}
Now, we define the hierarchical tree\footnote{In fact, it is not a tree. As we see later, there are edges between nodes at each level.} (denoted by $\mathcal{H}$-tree) of a sparse matrix $A$ given a sequence of nested partitionings.

\begin{definition}\label{def:Htree}\blue{ (hierarchical tree)
Given a sparse matrix $A$, and a nested sequence of partitionings $\mathcal{P}_0, \ldots, \mathcal{P}_l$, the hierarchical tree ($\mathcal{H}$-tree) is defined as a directed graph with red and black vertices and two types of edges (parent-child and interaction edges).} 

\blue{The vertices of $\mathcal{H}$-tree are corresponding to the clusters associated with the nested partitionings. 
For every cluster $C_j^{\mathcal{P}_i}$ with $0 \le i < l$ and $1 \le j \le 2^{i}$ there is a corresponding black-node $b_j^{[i+1]}$ in the vertex set of $\mathcal{H}$-tree. Also, there is a pair of red-nodes ${r_0}_j^{[i+1]}$ and ${r_1}_j^{[i+1]}$ corresponding to the sibling clusters $C_{j'}^{\mathcal{P}_{i+1}}$ and $C_{j''}^{\mathcal{P}_{i+1}}$ (children of the cluster $C_j^{\mathcal{P}_i}$) in the vertex set of $\mathcal{H}$-tree.
We call such a pair of red-nodes a super-node, and denote it by \orange{$s_j^{[i+1]}$}, that is corresponding to the cluster $C_j^{\mathcal{P}_i}$.
\orange{$b_j^{[i+1]}$ is connected to ${r_0}_j^{[i+1]}$ and ${r_1}_j^{[i+1]}$ by parent-child edges. 
Hence, $b_j^{[i+1]}$ is also the parent of super-node $s_j^{[i+1]}$. Additionally, the red-node corresponding to $C_j^{\mathcal{P}_i}$ is connected to $b_j^{[i+1]}$ by a parent-child edge.} We denote the parent of a node $v$ (which can be a red, black, or super node) by $\mathbb{P}(v)$.}

\blue{We also consider one special red-node associated to the root cluster. The level of each vertex of $\mathcal{H}$-tree is \orange{denoted by a superscript index}. Additionally, the depth of $\mathcal{H}$-tree is defined as $l$.}

\blue{There is an interaction edge between two red-nodes with level $l$ (leaf red-nodes), if the corresponding vertices in the adjacency graph of $A$ with partitioning $\mathcal{P}_l$ are connected. Therefore, the subgraph of $\mathcal{H}$-tree induced by ${r_0}_j^{[l]}$ and ${r_1}_j^{[l]}$ for $1 \le j \le 2^{l-1}$ is the adjacency graph of $A$ with partitioning $\mathcal{P}_l$.}
\end{definition}

\blue{There is no interaction edge between non-leaf vertices (shown transparent in \cref{fig:HTree}) of an $\mathcal{H}$-tree before applying the elimination. Non-leaf nodes are reserved to be used for the extended sparsification similar to the \orange{red and black nodes in \cref{fig:keyIdea:sub3}}.}

\orange{Similar to the vertices of an adjacency graph, each node of an $\mathcal{H}$-tree also corresponds to a set of variables and equations. Leaf-nodes of the $\mathcal{H}$-tree correspond to the variables and equations in \cref{eqn:linsys} partitioned using $\mathcal{P}_l$. Non-leaf nodes of the $\mathcal{H}$-tree, however, correspond to the auxiliary variables and equations.
In \cref{sec:comp}, when we explain the extended sparsification in the general case, we will introduce the variables and equations corresponding to the non-leaf nodes of the $\mathcal{H}$-tree.}

An example of an $\mathcal{H}$-tree is depicted in \cref{fig:HTree}. \blue{A parent-child edge between two nodes is shown by a dashed line, whereas interaction edges are shown by solid lines. In the rest of the paper, we use edge and interaction edge interchangeably, while parent-child edges are explicitly mentioned.}

\begin{figure}[htbp] \centering
\includegraphics[width=.9\textwidth]{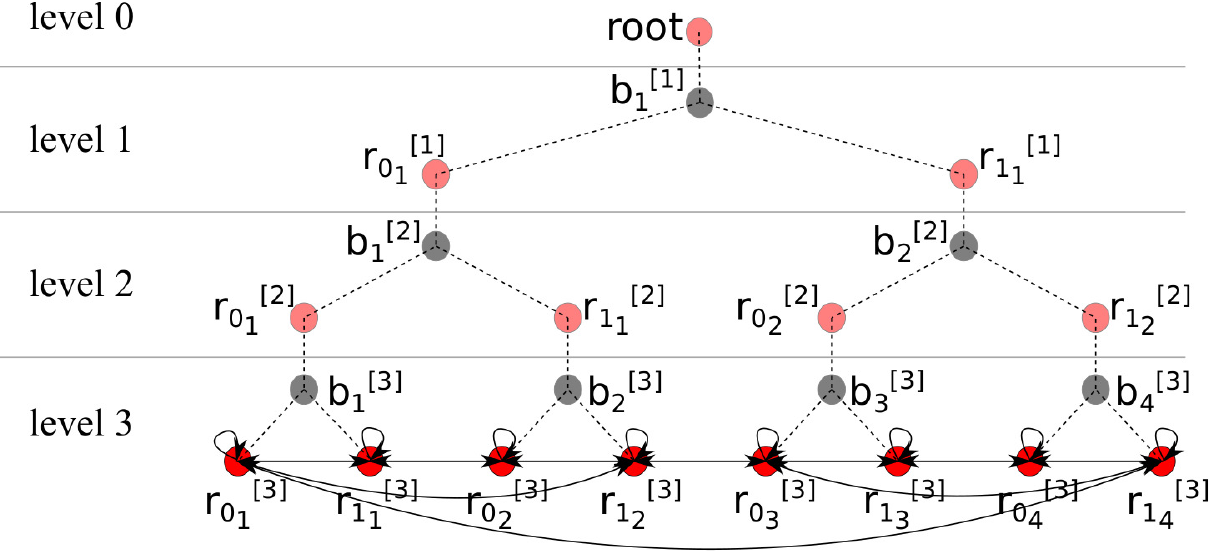}
\caption{\red{An example of a hierarchical tree. Dashed lines show parent-child \blue{edges}, and solid lines \blue{represent interaction edges}. Non-leaf nodes (shown transparent) have no interaction initially, and are reserved to represent the well-separated interactions at the level below them.}}
\label{fig:HTree} 
\end{figure}

\begin{definition}\label{def:adj}\blue{ (adjacent clusters)
Consider a sparse matrix $A$ and a nested sequence of partitionings $\mathcal{P}_0, \mathcal{P}_1, \mathcal{P}_2, \ldots, \mathcal{P}_l$. Two leaf clusters $C_p^{\mathcal{P}_l}$ and $C_q^{\mathcal{P}_l}$ are called adjacent (or neighbors) iff $A_{p,q}$ or $A_{q,p}$ are non-zero blocks (i.e., nodes $v_p$ and $v_q$ in the adjacency graph of $A$ with partitioning $\mathcal{P}_l$ are connected). Two clusters $C$ and $C'$ (not necessarily with the same level) are adjacent iff a leaf descendant\footnote{\blue{Cluster $C_1$ is a descendant of cluster $C_2$ iff $C_1 \subseteq C_2$.}} cluster of $C$ is adjacent to a leaf descendant cluster of $C'$.}
\end{definition}

\begin{definition}\label{def:wellSep} (well-separated nodes in the $\mathcal{H}$-tree)
\blue{Nodes $u$ and $v$ in the vertex set of an $\mathcal{H}$-tree (which can be red, black, or super-nodes at different levels)} are well-separated if their corresponding clusters are not adjacent. An interaction edge that connects two well-separated nodes is called a well-separated edge. \blue{Well-separated edges do not exist initially in the $\mathcal{H}$-tree, and are created as a result of elimination.}
\label{def:wellsep}
\end{definition}

\orange{Low-rank interaction is typically observed in physical systems when two clusters are sufficiently far apart from each other (e.g., well-separated clusters in fast multipole method). We can replace this condition by an equivalent distance requirement in the graph. However, for simplicity we weakened the requirement and declare that two clusters are well-separated only if they are not adjacent (\cref{def:wellSep}). If the partitioning of the domain is adequate, this simple definition is sufficient to approximate the more accurate separation requirement based on distance. In many cases, well-separated interactions are low-rank, but there is no guarantee. It depends on various details not studied in this work, such as the shape of the clusters.}
\section{Algorithm}\label{sec:alg}
\blue{In the previous section we defined the $\mathcal{H}$-tree, which is used to represent graphically the extended system\orange{, similar to the example provided in \cref{fig:keyIdea:sub3}}. In this section, we explain the details of the algorithm to compute a hierarchical representation of an approximate LU factorization of matrix $A$. Note that while $A$ is sparse, the $L$ and $U$ matrices are typically dense. However, we assume that they can be represented using a hierarchical tree (through the extended sparsification). We could find the matrices $L$ and $U$ through an elimination process, and compute their hierarchical representation using an extended sparsification. However, in order to achieve linear complexity, we perform elimination and the extended sparsification (compression) together.}

\red{The algorithm presented in this paper takes advantage of similar technique as in the inverse fast multipole method (IFMM) \cite{coulier2015inverse}. The IFMM can be used to compute the \blue{hierarchical representation of LU factorization} of a dense matrix which is \blue{given in hierarchical form (i.e., FMM matrix).
Essentially, in the IFMM, we are given an $\mathcal{H}$-tree that has interaction edges at all levels. This $\mathcal{H}$-tree represents a dense matrix.}}

\blue{In the sparse case, however, the $\mathcal{H}$-tree initially has interaction edges only at the leaf level. For both of the sparse and dense cases, we start with an $\mathcal{H}$-tree that represents matrix $A$, and end up with an $\mathcal{H}$-tree representing an approximate LU factorization of $A$. \red{In \cref{alg:fact} the overall factorization scheme is introduced. Various sub-algorithms are explained afterwards. In addition, a step by step example of the \blue{elimination} process on the $\mathcal{H}$-tree and the corresponding extended matrix is presented in \cref{sec:appA}. Similar to the standard LU factorization, after the elimination process we are able to efficiently compute $A^{-1}\boldsymbol{b}$ through forward and backward substitutions.}}

\IncMargin{1em}
\begin{algorithm}
\DontPrintSemicolon
\SetKwFunction{Input}{Input}
\SetKwFunction{Output}{Output}
\SetKwFunction{Initialize}{Initialize}
\SetKwFunction{MergeRedNodes}{MergeRedNodes}
\SetKwFunction{Eliminate}{Eliminate}
\SetKwFunction{Compress}{Compress}
\BlankLine
\blue{\Input: sparse matrix $A$}\;
\Initialize{$l$}\tcp*{\orange{form an initial H-tree with l levels}}
\For(\tcp*[f]{\orange{iterate over levels from leaf to root}}){$i\leftarrow l$ \KwTo $\blue{1}$}{
\For(\tcp*[f]{\orange{iterate over nodes at each level}}){$j\leftarrow 1$ \KwTo $2^{i-1}$}{
$s_j^{[i]} \leftarrow$ \MergeRedNodes{${r_0}_j^{[i]}, {r_1}_j^{[i]}$} \tcp*{\orange{form super-node}}
}
\For{$j\leftarrow 1$ \KwTo $2^{i-1}$}{
\Compress{$s_j^{[i]}$}\tcp*{\orange{compress well-separated interactions of super-node}}
\Eliminate{$s_j^{[i]}$}\tcp*{\orange{Gauss elimination for super-node}}
\Eliminate{$b_j^{[i]}$}\tcp*{\orange{block Gauss elimination for black-node}}
}
}
\blue{\Output: LU factorization of $A$ stored hierarchically in the $\mathcal{H}$-tree}\;
\caption{Factorization using $\mathcal{H}$-tree.} \label{alg:fact}
\end{algorithm}\DecMargin{1em}

\subsection{Initializing the $\mathcal{H}$-tree}\label{sec:initHtree}
The \verb!Initialize(!$l$\verb!)! function in \cref{alg:fact} consists of \blue{computing $l$ nested partitionings and form the $\mathcal{H}$-tree with depth $l$, as defined in \cref{def:Htree}}. \red{An example of an $\mathcal{H}$-tree and the corresponding matrix is depicted in \cref{fig:pictorial_alg0}}.
\blue{Leaf nodes and interaction edges of the $\mathcal{H}$-tree initially represent the given linear system of equations \cref{eqn:linsys} with partitioning $\mathcal{P}_l$. Through the elimination process, we extend the system of equations, and use non-leaf nodes of the $\mathcal{H}$-tree to represent the new variables and equations.}

\subsection{Forming super-nodes} \label{sec:merge}
The outer-loop in \cref{alg:fact} is over different levels from the bottom to the top of the tree. At each level, we start by merging red-siblings into super-nodes. \red{In \cref{alg:fact} this process is denoted by the function} \verb!MergeRedNodes()!. \blue{This process results in a coarser representation of the linear system}. 
\blue{We substitute interactions (i.e., edges) between any two pairs of red-nodes (${r_0}_j^{[i]}, {r_1}_j^{[i]}$ and ${r_0}_{j'}^{[i]}, {r_1}_{j'}^{[i]}$) with an interaction between super-nodes ($s_j^{[i]}$ and $s_{j'}^{[i]}$) as follows.
\begin{equation}
\mat(e_{s_j^{[i]} \to s_{j'}^{[i]}} ) =
\begin{pmatrix}
\mat(e_{ {r_0}_j^{[i]} \to {r_0}_{j'}^{[i]} } )&\mat(e_{ {r_1}_j^{[i]} \to {r_0}_{j'}^{[i]} } )\\
\mat(e_{ {r_0}_j^{[i]} \to {r_1}_{j'}^{[i]} } )& \mat(e_{ {r_1}_j^{[i]} \to {r_1}_{j'}^{[i]} } )
\end{pmatrix}
\end{equation}}

Similarly, the variable and right hand side vectors corresponding to a super-node is formed by concatenating the variables and right hand sides of its constituting red-nodes.
\begin{equation} \label{eqn:concatX}
\x(s_j^{[i]}) = \verb!concatenate!\left(\x({r_0}_j^{[i]}), ~\x({r_1}_j^{[i]})\right)
\end{equation}
\begin{equation} \label{eqn:concatRHS}
\rhs(s_j^{[i]}) = \verb!concatenate!\left(\rhs({r_0}_j^{[i]}), ~\rhs({r_1}_j^{[i]})\right)
\end{equation}
The process of merging red-nodes to form the super-nodes is illustrated in \cref{fig:merge}. \red{The merging process is also depicted in \cref{fig:pictorial_alg1} in the example provided in \cref{sec:appA}.}
\begin{figure}[htbp] \centering
\includegraphics[width=0.6\textwidth]{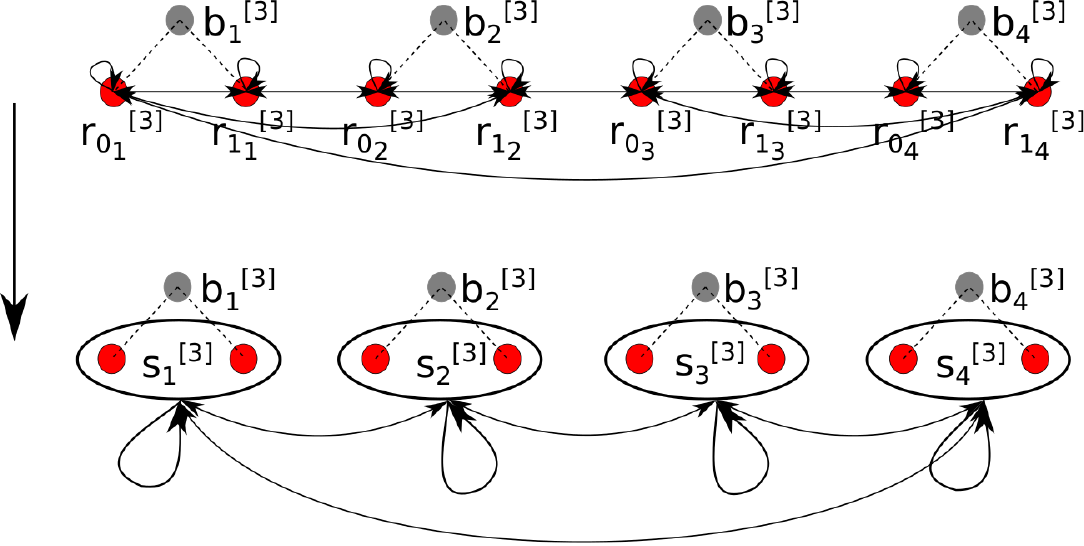}
\caption{An example of the merge process \red{corresponding to the} {\normalfont\texttt{MergeRedNodes()}} \red{function in \cref{alg:fact}}. The leaf red-nodes of the $\mathcal{H}$-tree in \cref{fig:HTree} (top) are merged to super-nodes (bottom).}
\label{fig:merge}
\end{figure}
\subsection{Compressing well-separated edges} \label{sec:comp}
The next sub-algorithm to consider is the \emph{compression}. \red{In \cref{alg:fact} this process is denoted by the function} \verb!Compress()!. During the compression, well-separated interactions of a super-node are pushed to the parent level nodes \red{which represent a set of auxiliary variables. This is essentially the extended sparsification method which we discussed in the example in \cref{sec:keyIdea} (i.e., replacing the well-separated edges in \cref{fig:keyIdea:sub2} by the sequence of edges between red and black nodes at the parent level as shown in \cref{fig:keyIdea:sub3}).}

 Assume we are at level $i$, and \red{ about to apply }\verb!Compress(!$s_j^{[i]}$\verb!)!. \red{Also, assume $s_j^{[i]}$ is of size $m$ (i.e., corresponds to $m$ variables and $m$ equations)}, and interacts with (i.e., has an edge to) $t$ well-separated nodes $p_1, p_2, \ldots, p_t$ with sizes $m_1, m_2, \ldots, m_t$, respectively. \blue{As it will be clear later, the $p_k$ nodes} for $k=1,\ldots,t$ can \blue{either} be a red-node (at the parent level) or a super-node \red{(at the same level)}.
Assume blocks $A_1, A_2, \ldots, A_t$  are associated to the outgoing well-separated edges from $s_j^{[i]}$ to  $p_1, p_2, \ldots, p_t$, respectively, \blue{i.e., $A_k = \mat(e_{s_j^{[i]}\to p_k})$,} where $A_k$ is an $m_k$ by $m$ matrix. Similarly, $B_1^{\intercal}, B_2^{\intercal}, \ldots, B_t^{\intercal}$ are associated to the incoming well-separated edges to $s_j^{[i]}$, \blue{i.e., $B_k^{\intercal} = \mat(e_{p_k\to s_j^{[i]}})$,} where $B_k$ is an $m_k$ by $m$ matrix. \red{This is depicted schematically in \cref{fig:comp} \blue{(left)}.}

\red{
Similar to the example in \cref{sec:keyIdea}, we assume well-separated edges can be approximated using a low-rank factorization. We compress all well-separated interactions of a super-node together. We then introduce auxiliary variables, and replace the well-separated edges by new edges between the auxiliary variables (i.e., going from the left configuration to the right configuration in \cref{fig:comp}).}
\begin{figure}[htbp] \centering
\includegraphics[width=0.99\textwidth]{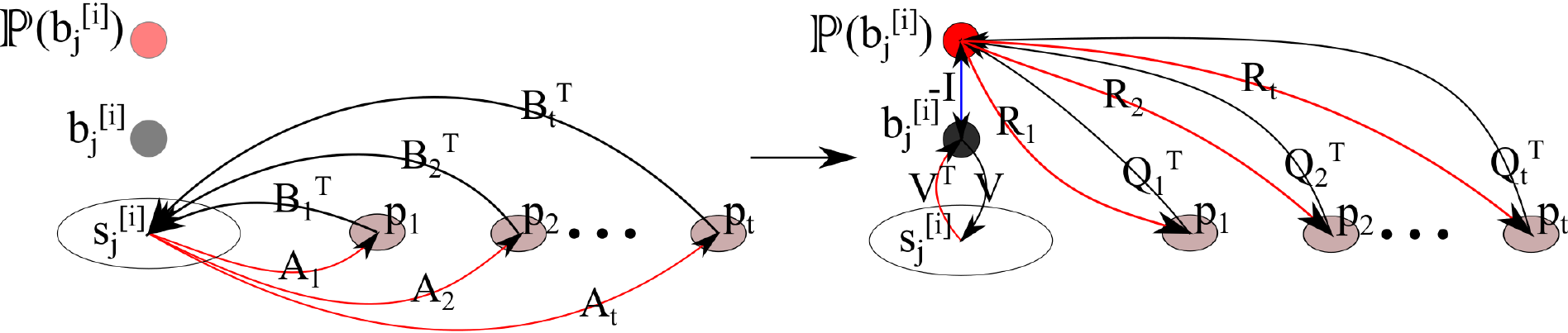}
\caption{Schematic of the compression process for a super-node $s_j^{[i]}$ \red{corresponding to the} {\normalfont\texttt{Compress()}} \red{function in \cref{alg:fact}}. Well-separated interactions are replaced with low-rank interactions with the red parent.}
\label{fig:comp} 
\end{figure}

\red{In order to compress all $t$ well-separated edges together, we form a temporary matrix} by vertically concatenating $A_1, \ldots, A_t$ as well as $B_1, \ldots, B_t$. Use a low-rank approximation method (e.g., SVD) and write:
\begin{equation}
\begin{bmatrix}
A_1\\
\vdots\\
A_t\\
B_1\\
\vdots\\
B_t
\end{bmatrix}
\simeq
\begin{bmatrix}
R_1\\
\vdots\\
R_t\\
Q_1\\
\vdots\\
Q_t
\end{bmatrix}
V^{\intercal},
\label{eqn:lra}
\end{equation}
where $R_k$ and $Q_k$ are $m_k$ by $r$ matrices for $k=1,2,\ldots, t$, and $V$ is an $m$ by $r$ matrix. $r$ is the rank in the above low-rank approximation. From \cref{eqn:lra} we can write
\blue{
\begin{equation}
A_k \simeq R_k V^{\intercal} \text{ and } B_k^{\intercal} = V Q_k^{\intercal} \quad \text{for}\quad k=1,2,\ldots,t
\end{equation}}
$\x(s_j^{[i]})$ contributes to the equation \blue{corresponding} to a node $p_k$ with a term like $A_k \cdot \x(s_j^{[i]})$ for $k=1,2,\ldots, t$. \blue{We can rewrite this term as}
\begin{equation} \label{eqn:equiv}
A_k \cdot \x(s_j^{[i]}) \simeq \left(R_k V^{\intercal} \right) \cdot \x(s_j^{[i]}) = R_k \left( V^{\intercal} \cdot \x(s_j^{[i]}) \right)
\end{equation}
\blue{Similarly, the contribution of $\x(p_1), \ldots, \x(p_t)$ in the equation corresponding to the node $s_j^{[i]}$ can be written as}
\begin{equation} \label{eqn:net}
\sum_{k=1}^t B_k^{\intercal} \cdot \x(p_k) \simeq \sum_{k=1}^t(V Q_k^{\intercal}) \cdot \x(p_k)  = V \sum_{k=1}^t Q_k^{\intercal} \cdot \x(p_k)
\end{equation}
\blue{Now, we apply the extended sparsification, and introduce new variables and equations as follows}
\begin{equation}
\orange{\x( \mathbb{P}( b_j^{[i]} ) ) := V^{\intercal} \cdot \x(s_j^{[i]}) \; \Rightarrow } \; \; V^{\intercal} \cdot \x(s_j^{[i]}) - \x( \mathbb{P}( b_j^{[i]} ) ) = {\boldsymbol{0}}
\label{eqn:anterpol}
\end{equation}
\begin{equation}
\orange{\x(b_j^{[i]}) := \sum_{k=1}^t Q_k^{\intercal} \cdot \x(p_k) \; \Rightarrow } \; \; \sum_{k=1}^t Q_k^{\intercal} \cdot \x(p_k)-\x(b_j^{[i]}) = {\boldsymbol{0}}
\label{eqn:interpol}
\end{equation}
\blue{In the above, we defined two new vector of variables, $\x( \mathbb{P}( b_j^{[i]} ) )$ and $\x(b_j^{[i]})$, each of size $r$ (similar to the vectors $\boldsymbol{y}_2$ and $\boldsymbol{z}_2$ in the example of \cref{sec:keyIdea}).
We assign equations \cref{eqn:anterpol,eqn:interpol} to the black-node $b_j^{[i]}$ and the red-node $\mathbb{P}( b_j^{[i]} )$, respectively. Therefore, 
\begin{equation}\label{eqn:rhszero}
\rhs(b_j^{[i]})=\rhs(\mathbb{P}( b_j^{[i]})) = {\boldsymbol{0}}
\end{equation}
Now, we apply the following edge updates:
\begin{itemize}
\item Remove edges from $s_j^{[i]}$ to $p_k$ and vice versa for $k=1,2,\ldots, t$.
\item Add edges from $\mathbb{P}( b_j^{[i]})$ to $p_k$ with blocks $R_k$ for $k=1,2,\ldots, t$.
\item Add edges from $p_k$ to $\mathbb{P}( b_j^{[i]})$ with blocks $Q_k^{\intercal}$ for $k=1,2,\ldots, t$.
\item Add an edge from $s_j^{[i]}$ to $b_j^{[i]}$ and vice versa with blocks $V^{\intercal}$ and $V$, respectively.
\item Add an edge from $b_j^{[i]}$ to $\mathbb{P}( b_j^{[i]})$ and vice versa with blocks $-\mathbb{I}_{r}$ (minus identity matrix of size $r$).
\end{itemize}
Therefore, in the compression process all well-separated edges connected to the super-node $s_j^{[i]}$ are substituted with edges to/from $\mathbb{P}( b_j^{[i]} )$, as shown in \cref{fig:comp}.
\orange{Note that each step of the compression process described here is in some sense ``half'' of the extended sparsification step as described in \cref{sec:keyIdea}.
For example, in \cref{fig:pictorial_alg2} there is a well-separated interaction between $s_2^{[3]}$ and $s_4^{[3]}$. After the compression step, this interaction is substituted with a well-separated edge between ${r_1}_1^{[2]}$ and $s_4^{[3]}$ (\cref{fig:pictorial_alg3}). Ultimately, when we are about to eliminate $s_4^{[3]}$, we further compress this edge and make connection between ${r_1}_1^{[2]}$ and ${r_1}_2^{[2]}$ (\cref{fig:pictorial_alg6}), which is now the same as the example in \cref{fig:keyIdea}.}}

\blue{As a result of the compression process on a super-node $s_j^{[i]}$, we defined $2r$ new auxiliary variables}, \orange{corresponding to two nodes $b_j^{[i]}$ and  $\mathbb{P}( b_j^{[i]})$, each of size $r$}. Note that if the matrix is symmetric, $A_k = B_k$ for $k=1,\ldots,t$. Therefore, we would not need to concatenate $B_k$'s in \cref{eqn:lra}, and only half of the above calculations are required.

\subsection{Elimination} \label{sec:elimAlg}
After compressing all well-separated edges, we apply the standard elimination. 
As a result of the compression process, the super-node  $s_j^{[i]}$ is only connected to its \blue{original} neighbor nodes.
This is a key property of the algorithm that preserves the sparsity of the matrix, \orange{which results in a slightly larger system of equations (a constant number times the original size of the matrix)}. The elimination process for a node is explained in \cref{sec:elimination}. We first eliminate the super-node $s_j^{[i]}$, and then eliminate its black-parent $b_j^{[i]}$. \red{In the example provided in \cref{sec:appA}, \cref{fig:pictorial_alg2,fig:pictorial_alg4,fig:pictorial_alg5,fig:pictorial_alg7} illustrate the graph and matrix after the elimination process.}

\subsection{Solve}
\red{After the factorization part is completed,  we can solve for multiple right hand sides. The solve process consists of two steps: a forward and a backward traversal of all nodes. This is identical to the standard forward and backward substitutions in the LU factorization. In the forward traversal we visit all nodes in the order they have been eliminated, and in the backward traversal we visit nodes in the exact reverse order. The solve process is introduced in \cref{alg:solve}.}

Note that in the factorization part we introduced auxiliary variables and equations (i.e., all variables and equations associated to non-leaf nodes). We denote this \emph{extended} system of equations by $A_e {\boldsymbol{x}_e} = {\boldsymbol{b}_e}$, \blue{which is equivalent to the system \cref{eqn:linsys} up to the accuracy of \cref{eqn:lra} (i.e., eliminating the auxiliary variables from the extended system recovers the original system of equations).}
In the solve part, we have to solve for all variables (original and auxiliary variables, i.e., ${\boldsymbol{x}_e}$) even though we are just interested in the original variables, $\boldsymbol{x}$. The number of extra variables is limited (see \cref{lem:lincomp}), and is of the same order as the number of original variables. 

\IncMargin{1em}
\begin{algorithm}[!htbp]
\DontPrintSemicolon
\SetKwFunction{Input}{Input}
\SetKwFunction{Output}{Output}
\SetKwFunction{SetRHS}{SetRHS}
\SetKwFunction{SplitVar}{SplitVar}
\SetKwFunction{SolveL}{SolveL}
\SetKwFunction{SolveU}{SolveU}
\BlankLine
\blue{\Input: An $\mathcal{H}$-tree representing an approximate LU factorization of $A$ and a right hand side vector $\boldsymbol{b}$}\;
\SetRHS{} \tcp*{\orange{set leaf \rhs s to b and non-leaf \rhs s to 0}}
\tcc{\orange{forward substitution}}
\For(\tcp*[f]{\orange{iterate over levels from leaf to root}}){$i\leftarrow l$ \KwTo $\blue{1}$}
{
	\For(\tcp*[f]{\orange{iterate over nodes at each level}}){$j\leftarrow 1$ \KwTo $2^{i-1}$}
	{
		\SolveL{$s_j^{[i]}$} \tcp*{\orange{update \rhs ~of the super-node}}
		\SolveL{$b_j^{[i]}$} \tcp*{\orange{update \rhs ~of the black-node}}
	}
}
\tcc{\orange{backward substitution}}
\For(\tcp*[f]{\orange{iterate over levels from root to leaf}}){$i\leftarrow \blue{1}$ \KwTo $l$}
{
	\For(\tcp*[f]{\orange{iterate nodes with reverse order}}){$j\leftarrow 2^{i-1}$ \KwTo $1$}
	{
		\SolveU{$b_j^{[i]}$} \tcp*{\orange{solve for variables of the black-node}}
		\SolveU{$s_j^{[i]}$} \tcp*{\orange{solve for variables of the super-node}}
		\SplitVar{$s_j^{[i]}$} \tcp*{\orange{split solution between the constituting red-nodes}}
	}
}
\blue{\Output: $\boldsymbol{x} \simeq A^{-1} \boldsymbol{b}$}
\caption{Solve for a given right hand side using $\mathcal{H}$-tree.}\label{alg:solve}
\end{algorithm}\DecMargin{1em}
The solve algorithm begins with \blue{\texttt{SetRHS()}, that is to set the right hand side of all nodes in the $\mathcal{H}$-tree}. \blue{As explained in \cref{def:Htree}, leaf red-nodes of the $\mathcal{H}$-tree correspond to the original equation \cref{eqn:linsys}; therefore, the right hand side of each leaf red-node is a sub-vector} of ${\boldsymbol{b}}$ determined by the leaf-partitioning of the $\mathcal{H}$-tree, $\mathcal{P}_l$.
\blue{Based on \cref{eqn:rhszero}, the right hand side of every non-leaf red-node and black-node is ${\boldsymbol{0}}$. The right hand sides of super-nodes are computed by \cref{eqn:concatRHS}.}

\blue{After setting the right hand side vectors, we apply functions} \verb!SolveL()! and \verb!SolveU()! to all super-nodes and black-nodes. \blue{The function} \verb!SolveL()! \blue{(see \cref{alg:solveL}) is applied through a forward traversal, and updates the right hand side vectors. The function} \verb!SolveU()! \blue{(see \cref{alg:solveU}) is applied through a backward traversal, and solve for the variables of each node.} After solving for variables of a super-node, we split the solution between its two constituting red-nodes (denoted by function \verb!SplitVar()!) \red{according to \cref{eqn:concatX}}. \blue{When \cref{alg:solve} is completed, the solution vector $\boldsymbol{x}$ is formed by concatenating variable vectors of all leaf red-nodes.}
\IncMargin{1em}
\begin{algorithm}[!htbp]
\DontPrintSemicolon
\SetKwProg{Fn}{Function}{}{}
\SetKwFunction{Input}{Input}
\SetKwFunction{Output}{Output}
\SetKwData{Mat}{Mat}
\SetKwData{node}{node}
\SetKwFunction{OutGoingEdges}{OutGoingEdges}
\SetKwFunction{SolveL}{SolveL}
\SetKwFunction{Order}{Order}
\SetKwFunction{RHS}{RHS}
\BlankLine
\Fn{\SolveL{\node $p$}}
{\blue{
$f \leftarrow$ \Mat{$e_{p\to p}$}$^{-1}$ $\cdot$ \RHS{$p$} \;
\For{$e_{p\to q} \in$ \OutGoingEdges{$p$}}
{
	\If(\tcp*[f]{\orange{if q is eliminated after p}}){\Order{q}$>$\Order{$p$}}
	{
		\RHS{q} $\leftarrow$ \RHS{q} $-~$\Mat{$e_{p\to q}$} $\cdot$ $f$
	}
}}
}
\caption{\red{Forward traversal (update the right hand sides)}}\label{alg:solveL}
\end{algorithm}\DecMargin{1em}
\IncMargin{1em}
\begin{algorithm} [!htbp]
\DontPrintSemicolon
\SetKwFunction{Input}{Input}
\SetKwFunction{Output}{Output}
\SetKwData{Mat}{Mat}
\SetKwProg{Fn}{Function}{}{}
\SetKwData{node}{node}
\SetKwFunction{InComingEdges}{InComingEdges}
\SetKwFunction{SolveU}{SolveU}
\SetKwFunction{Order}{Order}
\SetKwFunction{RHS}{RHS}
\SetKwFunction{Var}{Var}
\BlankLine
\Fn{\SolveU{\node $p$}}
{
\Var{$p$}$\leftarrow\RHS{$p$}$\;\blue{
\For{$e_{q\to p} \in$ \InComingEdges{$p$}}
{
	\If(\tcp*[f]{\orange{if q is eliminated after p}}){\Order{$q$}$>$\Order{$p$}}
	{
		\Var{$p$} $\leftarrow$ \Var{$p$} $-~$\Mat{$e_{q\to p}$} $\cdot$ \Var{$q$}\;
	}
}
\Var{$p$} $\leftarrow$ \Mat{$e_{p\to p}$}$^{-1} \cdot$ \Var{$p$}\;
}}
\caption{\red{Backward traversal (solve for variables)}}\label{alg:solveU}
\end{algorithm}\DecMargin{1em}

In \cref{alg:solveL,alg:solveU} \verb!OutGoingEdges(!$p$\verb!)! and \verb!InComingEdges(!$p$\verb!)!, respectively, denote the set of all outgoing and incoming edges of a node $p$ \blue{in the $\mathcal{H}$-tree}. The function \verb!Order(!$p$\verb!)! returns the order of elimination of a node $p$, i.e., if in \cref{alg:fact} a node $q$ is eliminated after a node $p$, then \verb!Order(!$q$\verb!)!$>$\verb!Order(!$p$\verb!)!.
\section {Linear complexity} \label{sec:lin}
In this section we show that the block sparsity of the extended matrix is preserved through the elimination process. Therefore, the factorization algorithm has provable linear complexity provided that the block sizes (and thus the rank of the low-rank approximations) are bounded.

\blue{In \cref{def:wellSep} we defined well-separated nodes in an $\mathcal{H}$-tree. In this section, we generalize this concept, and define distance of nodes in a given $\mathcal{H}$-tree.}

\begin{definition}\label{def:distance} \blue{(distance of nodes in $\mathcal{H}$-tree)
Consider nodes $u$ and $v$ in the vertex set of an $\mathcal{H}$-tree of a sparse matrix $A$ with sequence of nested partitionings $\mathcal{P}_0, \ldots, \mathcal{P}_l$. Using \cref{def:Htree}, $u$ and $v$ correspond to clusters $C_j^{\mathcal{P}_i}$ and $C_{j'}^{\mathcal{P}_{i'}}$ for some $0 \le i,i' \le l$, $1\le j \le 2^i$, and $1 \le j' \le 2^{i'}$. Assume $i \le i'$, and $C_{j'}^{\mathcal{P}_{i'}}$ is a descendant of cluster $C_{k}^{\mathcal{P}_{i}}$. The distance between nodes $u$ and $v$ is defined as the length of (i.e., number of edges) the minimum path between nodes $v_j$ and $v_k$ in the adjacency graph of $A$ with partitioning $\mathcal{P}_i$.}
\end{definition}

\begin{corollary}\label{cor:dist}
\blue{Nodes $u$ and $v$ in the vertex set of an $\mathcal{H}$-tree are well-separated iff their distance is greater than 1.}
\end{corollary}

\blue{Note that other criteria are possible to define well-separated nodes (see \cref{sec:conclusion}).}
\begin{theorem} (preservation of sparsity)\label{lem:sparse}
In \cref{alg:fact}, we never \red{create} an edge between two nodes with distance greater than 2.
\end{theorem}
\begin{proof}
\orange{Elimination can results in connecting nodes at large distances. However,
in \cref{alg:fact}, before applying elimination we remove all edges to nodes with distance larger than 1 (i.e., the well-separated edges). Therefore, after eliminating a node we create edges between nodes with distance at most 2.}
\end{proof}

\Cref{lem:sparse} shows that for each node we need to process at most $\kappa_1+\kappa_2$ edges\blue{, where $\kappa_1$ and $\kappa_2$ are, respectively, the maximum number of super-nodes at distance 1 and 2 from a super-node.} Note that $\kappa_1$ and $\kappa_2$ depend on the original matrix sparsity pattern, and are independent of the size of the matrix. To establish linear complexity of the factorization, we need to bound the size of the nodes \blue{(i.e., number of variables associated to them)} in the $\mathcal{H}$-tree.

For matrices arising from the discretization of a PDE, well separated edges correspond to the interaction of points that are physically far from each other. Therefore, if the Green's function of the associated PDE is smooth enough, one can expect a well-separated interaction to be \red{numerically} low-rank. We provide numerical evidence in \cref{sec:numerical} to support this \red{argument}. \orange{Note that the low-rank property of well-separated nodes depends on the quality of partitioning and the definition of well-separation. These are topics for followup studies.}

For general sparse matrices we can guarantee the linear complexity through bounding the rank growth. This is explained in the next theorem.\\
\begin{theorem} (linear complexity condition)\label{lem:lincomp} 
Consider \blue{$d_i$} to be the maximum size of super-nodes at level $i$ of an $\mathcal{H}$-tree with \blue{depth $l$} resulted from \cref{alg:fact} such that
\begin{equation} \label{eqn:geometric}
\blue{d_i} \le \alpha^{l-i} \blue{d_l}\quad\text{and}\quad\orange{d_l = \mathcal{O}({n}/{2^l} ) = \mathcal{O}(1),}
\end{equation}
where $0 < \alpha < \sqrt[3]{2}$ is a constant number.
\orange{Also assume $\kappa_1$ and $\kappa_2$, the maximum  number of super-nodes at distance 1 and 2 from a super-node, are $\mathcal{O}(1)$ quantities.}
Under these conditions the cost of the algorithm is linear with respect to the problem size.
 \end{theorem}
\begin{proof}
For a given super-node at level $i$ the compression cost is $\mathcal{O}\left(\kappa_2 \blue{d_i^3} \right)$, and the elimination cost is $\mathcal{O}\left(\kappa_1^2 \blue{d_i^3} \right)$. Note that the required memory scales with $\blue{d_i^2}$ for each node. Ignoring the constant factors $\kappa_1$ and $\kappa_2$, the order of the total cost of factorization is as follows:
\begin{equation} \label{eqn:cost}
\text{factorization cost} = \mathcal{O}\left( \sum_{i=1}^l 2^{i-1} \blue{d_i^3} \right)
\end{equation}
Plug \cref{eqn:geometric} in \cref{eqn:cost}:
\begin{equation} \label{eqn:complexity}
\begin{split}
\text{factorization cost} &= \mathcal{O}\left( \sum_{i=1}^l 2^{i-1} \alpha^{3(l-i)} \blue{d_l^3} \right)\\ 
&= \mathcal{O}\left( 2^{l-1} \blue{d_l^3} \sum_{i'=0}^{l-1} \left( \frac{\alpha^3}{2} \right)^{i'} \right) \quad\text{change of variable: $i'=l-i$}\\
&= \mathcal{O} ( 2^l \blue{d_l^3})
\end{split}
\end{equation}
Note that for $\alpha < \sqrt[3]{2}$ we have $\sum_{i'=0}^{l-1} \left( \frac{\alpha^3}{2} \right)^{i'} = \mathcal{O}(1)$. Furthermore, $\blue{d_l}$ is the number of variables in super-nodes at \blue{the} leaf level which is $\mathcal{O}({n}/{2^l} )$. Therefore:
\begin{equation}
\text{factorization cost} = \mathcal{O} \left( n \blue{d_l^2} \right),\quad \text{factorization memory} = \mathcal{O} \left( n \blue{d_l} \right)
\end{equation}
\end{proof}
\section {Numerical results} \label{sec:numerical}
We have implemented the algorithm described in \cref{sec:alg} in C++. The code (we call it LoRaSp\footnote{{\bf Lo}w {\bf Ra}nk {\bf Sp}arse solver.}) can be downloaded from \url{bitbucket.org/hadip/lorasp}. We use Eigen \cite{eigenweb} as the backend for linear algebra calculations, and SCOTCH \cite{pellegrini1996scotch} for graph partitioning. We present results for various benchmarks, where LoRaSp is used as a direct solver, or as a preconditioner in conjunction with an iterative solver. 
\subsection{LoRaSp as a stand-alone solver} \label{sec:standalone}
In this section we employ LoRaSp as a stand-alone solver. The accuracy of the solver depends on the accuracy of the low-rank approximations during the compression step as explained in \cref{sec:comp}. Any low-rank approximation method can be used for the compression. Here, we use SVD. For every well-separated interaction, we first compute the SVD, and then truncate the singular values at some point. There are many possible criteria to truncate singular values. We discuss some possible criteria. \Cref{fig:decay} shows the decay of singular values for blocks corresponding to the interaction between randomly chosen well-separated nodes at different levels of an $\mathcal{H}$-tree. The tree corresponds to a matrix obtained from the second-order uniform discretization of the Poisson equation:
\begin{equation}
\begin{split}
\nabla . \left( \nabla T \right)  &= f \quad \text{ in }\mathcal{D}, \\
T &= {\boldsymbol{0}}  \quad \text{ on }\partial \mathcal{D}
\end{split}
\label{eqn:Pois}
\end{equation}
The domain $\mathcal{D}$ is a three-dimensional unit cube. The matrix size is 32,768, and the depth of the corresponding $\mathcal{H}$-tree is 11.
Evidently, singular values have exponential decay at different levels of the tree. The zero (up to machine precision) singular values are not shown in the plot.
\begin{figure}[htbp] \centering
\includegraphics[width=0.60\textwidth]{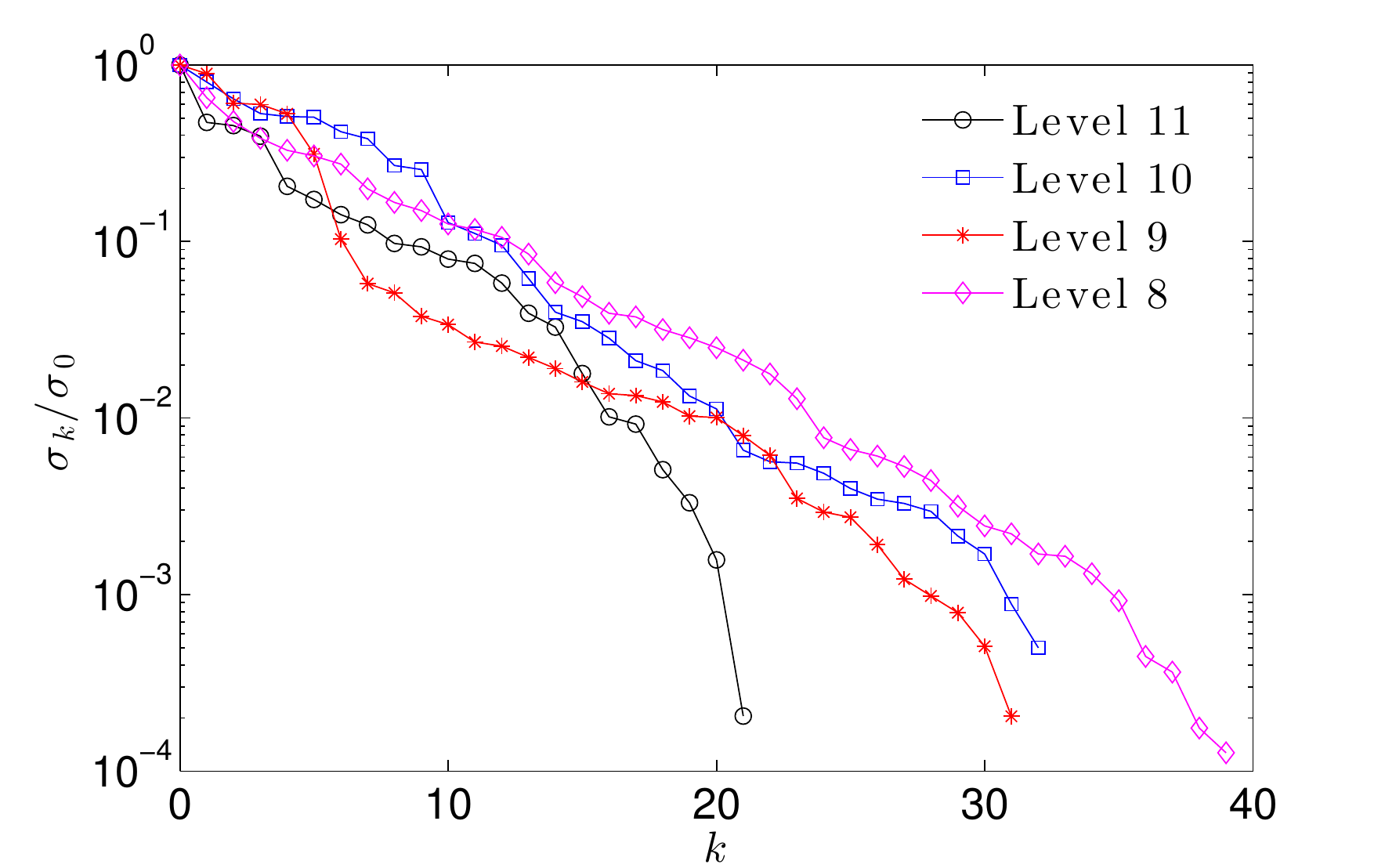}
\caption{Decay of singular values of random blocks corresponding to well-separated interactions \orange{(to be compressed)} at different levels of an $\mathcal{H}$-tree with 11 levels \orange{resulted from \cref{alg:fact}.}}
\label{fig:decay} 
\end{figure}

To demonstrate the linear complexity of the method, we considered a sequence of problems with a growing number of variables. Consider the following sequence of uniform discretization of the domain $\mathcal{D}$ in \cref{eqn:Pois}: $32\times32\times16$, $32\times32\times32$, $64\times32\times32$, $64\times64\times32$, $64\times64\times64$, $128\times64\times64$, and $128\times128\times64$. The matrix size is increased by a factor of 2 in the consecutive problems. Hence, to keep the size of the leaf super-nodes constant among all problems, we consider $\mathcal{H}$-trees with depth $10, 11, 12, 13, 14, 15, 16$ for this sequence of problems, respectively. In general, the depth of $\mathcal{H}$-tree should scale linearly with $\log_2 n$, where $n$ is the size of matrix. 

Well-separated edges corresponding to a block $B$\red{, as shown in \cref{eqn:lra},} with singular-values $\sigma_0, \sigma_1, ...$, are compressed by keeping only the singular-values that satisfy:
\begin{equation}
\frac{\sigma_k}{\sigma_0} \ge \epsilon
\label{eqn:cutOff1}
\end{equation}
Smaller values of $\epsilon$ lead a to more accurate approximation of each block, and consequently a more accurate approximation of the final solution. For a given linear system $A{\boldsymbol{x}}={\boldsymbol{b}}$, the precision of any solution $\tilde{\boldsymbol{x}}$ is quantified by the relative error and relative residual defined as follows
\begin{equation}
\text{error} = \frac{\| {\boldsymbol{x}} - \tilde{\boldsymbol{x}}\|_2} { \| {\boldsymbol{x}} \|_2 },\quad \text{residual} = \frac{ \| A\tilde{\boldsymbol{x}} - {\boldsymbol{b}} \|_2} {\| {\boldsymbol{b}} \|_2}
\label{eqn:accres}
\end{equation}
\Cref{fig:directSolve_varyN:sub2} shows that smaller values of $\epsilon$ (i.e., more accurate low-rank approximations) result in a more accurate estimation of the solution to the linear system in the cost of larger factorization and solve times, as shown in \cref{fig:directSolve_varyN:sub1}. For a constant $\epsilon$, the time spent for the factorization and solve parts \blue{are asymptotically linear with the problem size}. \red{Note that for smaller values of $\epsilon$ the linear scaling is achieved for larger values of $n$.} In addition, the error and residual of the estimated solution for a fixed $\epsilon$ barely change with the problem size (see \cref{fig:directSolve_varyN:sub2}).
\begin{figure}[htbp]
\begin{center}
\subfloat[]{\includegraphics[width=0.48\textwidth]{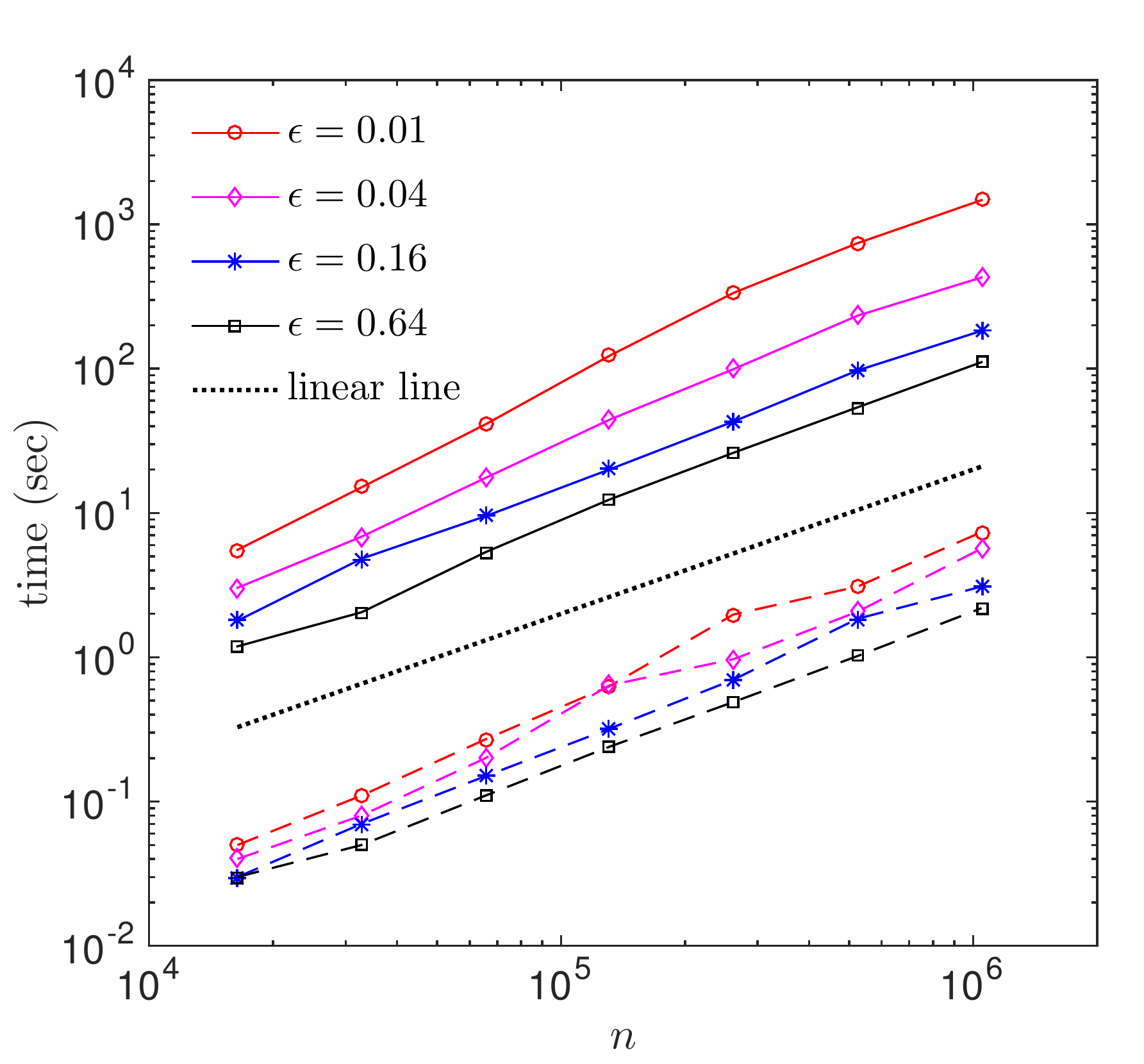}\label{fig:directSolve_varyN:sub1}} \quad
\subfloat[]{\includegraphics[width=0.48\textwidth]{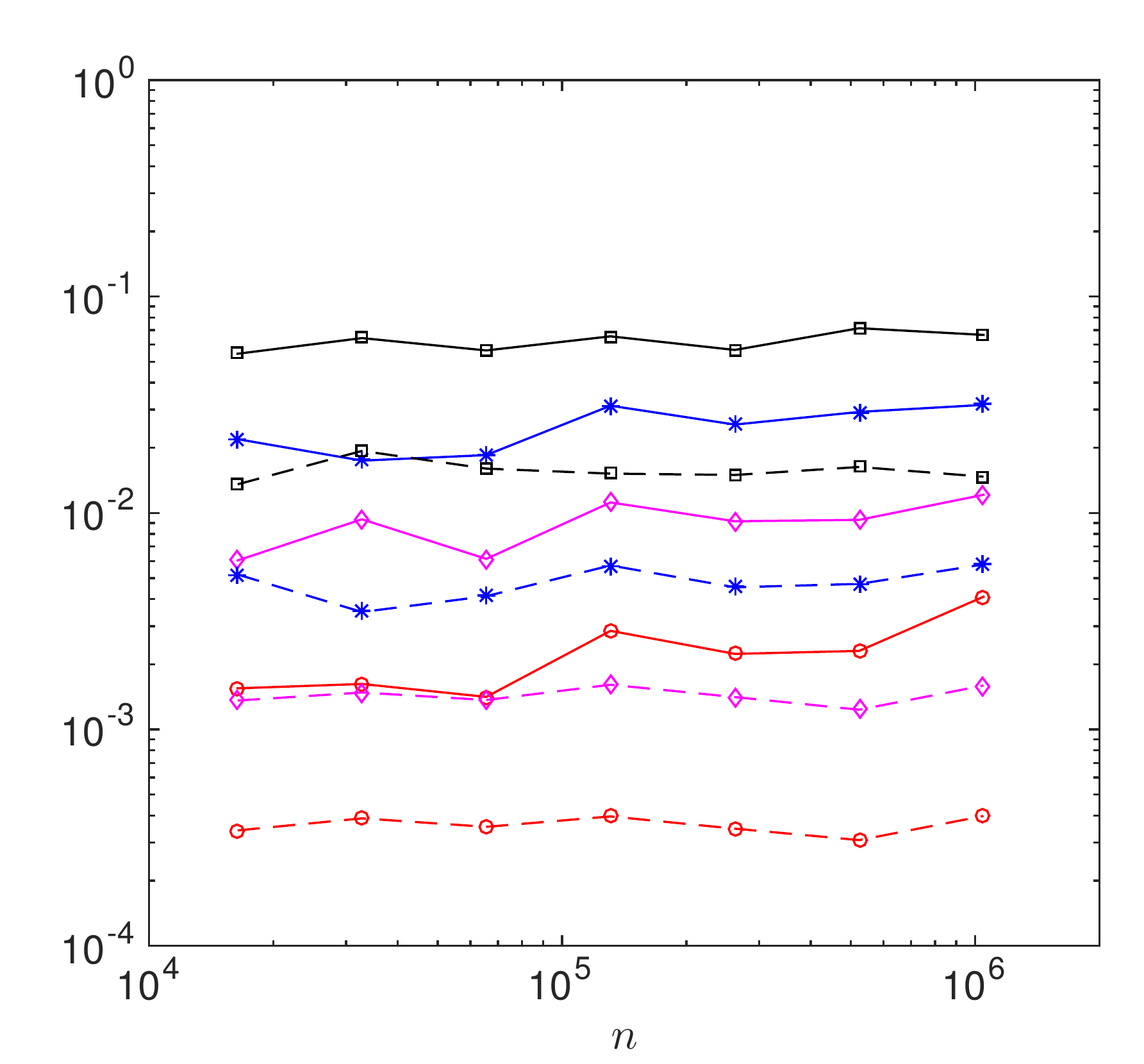}\label{fig:directSolve_varyN:sub2}}
\caption{\red{Performance of the solver for different problem sizes using different levels of precision (shown by different colors and symbols) in low-rank approximations. (a) Time spent on factorization (solid line) and solve (dashed line) parts, (b) error (solid line) and residual (dashed line) of the solution.}}
\label{fig:directSolve_varyN} 
\end{center}
\end{figure}

As it is clear from \cref{fig:directSolve_varyN}, we can obtain more accurate solutions by decreasing the parameter $\epsilon$ in \cref{eqn:cutOff1}. To show the convergence of the solver, we picked a fixed problem size, and measured the \blue{accuracy} of the estimated solution as $\epsilon$ decreases. In addition, for comparison purposes, we consider a 2D variation of \cref{eqn:Pois} which is discretized using a finite volume approach with Voronoi tessellation. The points are drawn from a random uniform distribution in the $\left[0,1\right]^2$ interval. \red{The discretization results in a matrix $A=DB$, where $D$ is a diagonal matrix with inverse of the Voronoi cells on the diagonal, and $B$ is a symmetric matrix. We apply the factorization directly to $B$.} Note that the average \red{number of non-zeros per row for a matrix corresponding to} a 2D Voronoi discretization is 7, which is the same as for a uniform second order 3D discretization.\footnote{We can show this by double counting the angles in a 2D Voronoi tessellation, once through points, and once through triangles of the corresponding Delaunay triangulations.}


In \cref{fig:convergence} the convergence of the solution for a 3D Poisson problem with \red{$n=1.3\times10^5$ (corresponding to a $64\times64\times32$ grid)} and a 2D Poisson problem with the same size (corresponding to Voronoi tessellation) are shown. The error and residual decrease proportional to the precision in the low-rank approximations, $\epsilon$. Furthermore, it is clear that the residual is smaller than the error. The ratio of the error to residual is generally an increasing function with respect to the condition number of the matrix. For the same number of points, the condition number of a 3D discretization is lower than that of a 2D discretization. Therefore, the error and residual are closer in the 3D case in comparison to the 2D case as illustrated in \cref{fig:convergence:sub2}.

\Cref{fig:convergence:sub1} demonstrates the factorization time as a function of the low-rank approximation precision, $\epsilon$. Three dimensionality leads to a higher number of well-separated interactions; hence, after every elimination more new fill-ins are introduced in the 3D case. \blue{This results in a higher factorization time compared to the 2D case.}

\begin{figure}[htbp]
\begin{center}
\subfloat[]{\includegraphics[width=0.48\textwidth]{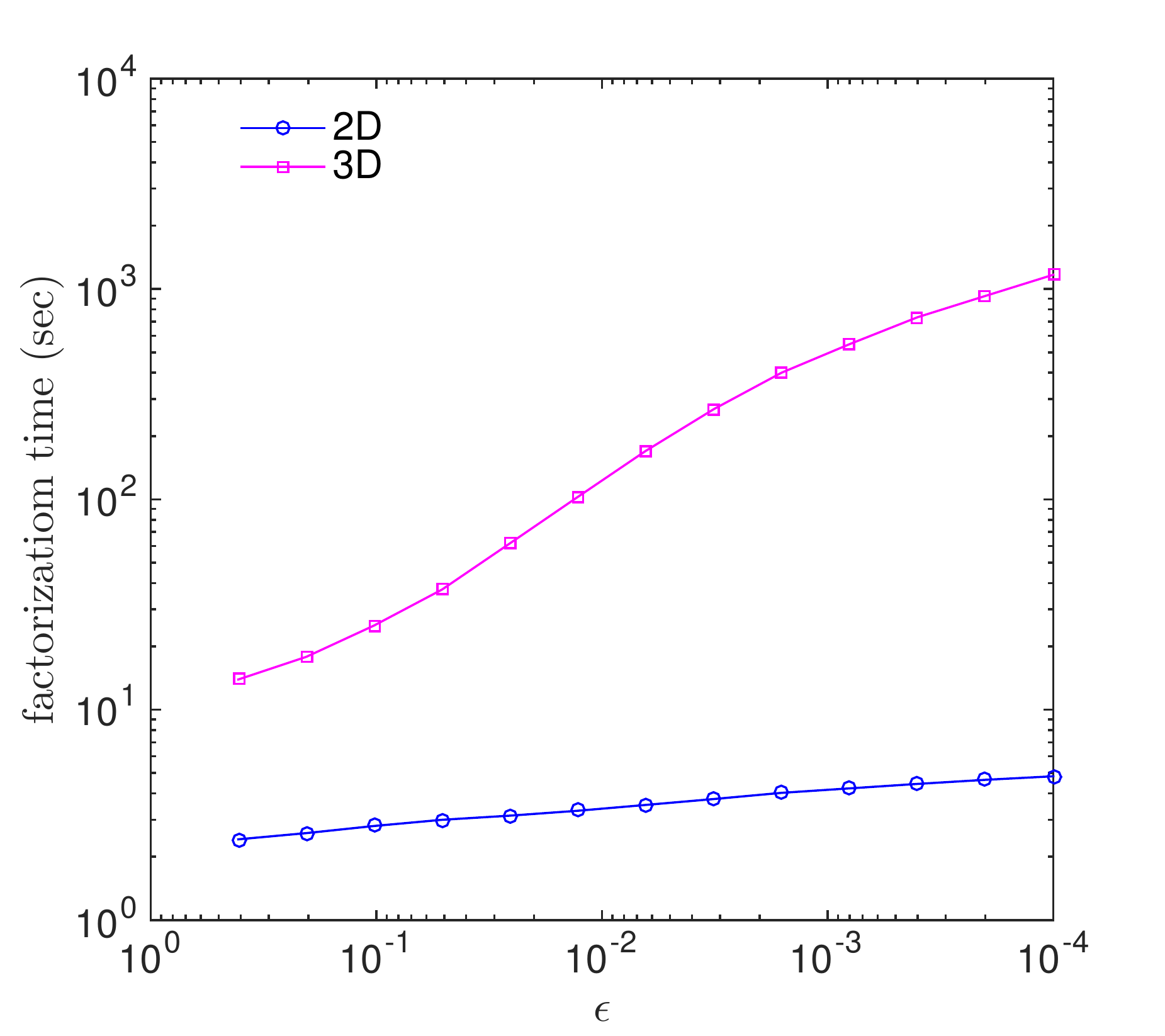}\label{fig:convergence:sub1}} \quad
\subfloat[]{\includegraphics[width=0.48\textwidth]{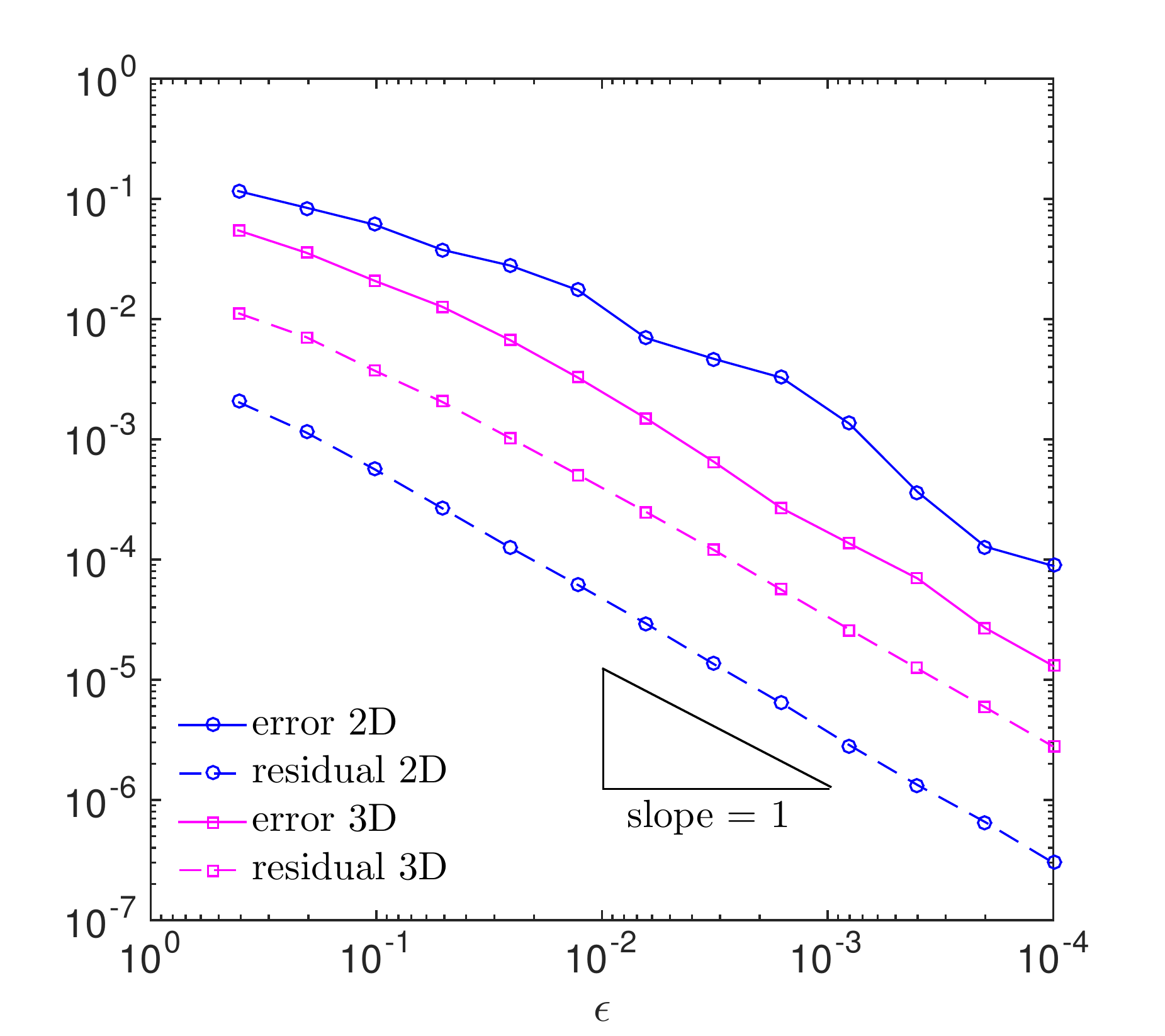}\label{fig:convergence:sub2}}
\caption{(a) Factorization time as a function of the low-rank approximation precision ($\epsilon$) for the 3D and 2D Poisson problems of size \red{$1.3\times10^5$}; (b) error and residual of the solution.}
\label{fig:convergence}
\end{center}
\end{figure}

\Cref{fig:time_components} shows the breakdown of the time spent on different parts of the algorithm, namely, low-rank approximation of well-separated blocks (here, SVD), general matrix multiplication (gemm), and computing the inverse of \red{pivot blocks}. Clearly, for the 3D case most of the time is spent on SVD, which is known to be an expensive algorithm for low-rank approximation. This can be improved significantly if faster low-rank approximation methods are employed. We discuss some of the alternative methods in \cref{sec:conclusion}.

In \cref{fig:ranks}, the average ranks of the well-separated interactions at each level are depicted as a function of the low-rank approximation precision. Note that for both the 3D and 2D cases an $\mathcal{H}$-tree with depth \red{$l=13$} is used, i.e., there are 16 variables per leaf red-nodes on average. The rank for the 3D case increases dramatically compared to the 2D case when a more accurate solution is desired. If \blue{$d_L$} is the maximum rank among all levels (i.e., maximum size of a red-node), similar to the analysis of \cref{lem:lincomp}, the factorization complexity is $\mathcal{O}(n \blue{d_L^2})$. From \cref{fig:decay,fig:ranks} we can observe that $\blue{d_L} = \mathcal{O} \left(\log \frac{1}{\epsilon} \right)$. Therefore, we have:
\begin{equation}
\text{factorization cost} = \mathcal{O} \left(n\log^2 \frac{1}{\epsilon} \right)\text{,} \quad \text{factorization memory} = \mathcal{O} \left( n \log \frac{1}{\epsilon}\right)
\end{equation}

\begin{figure}[htbp]
\begin{center}
\subfloat[]{\includegraphics[width=0.49\textwidth]{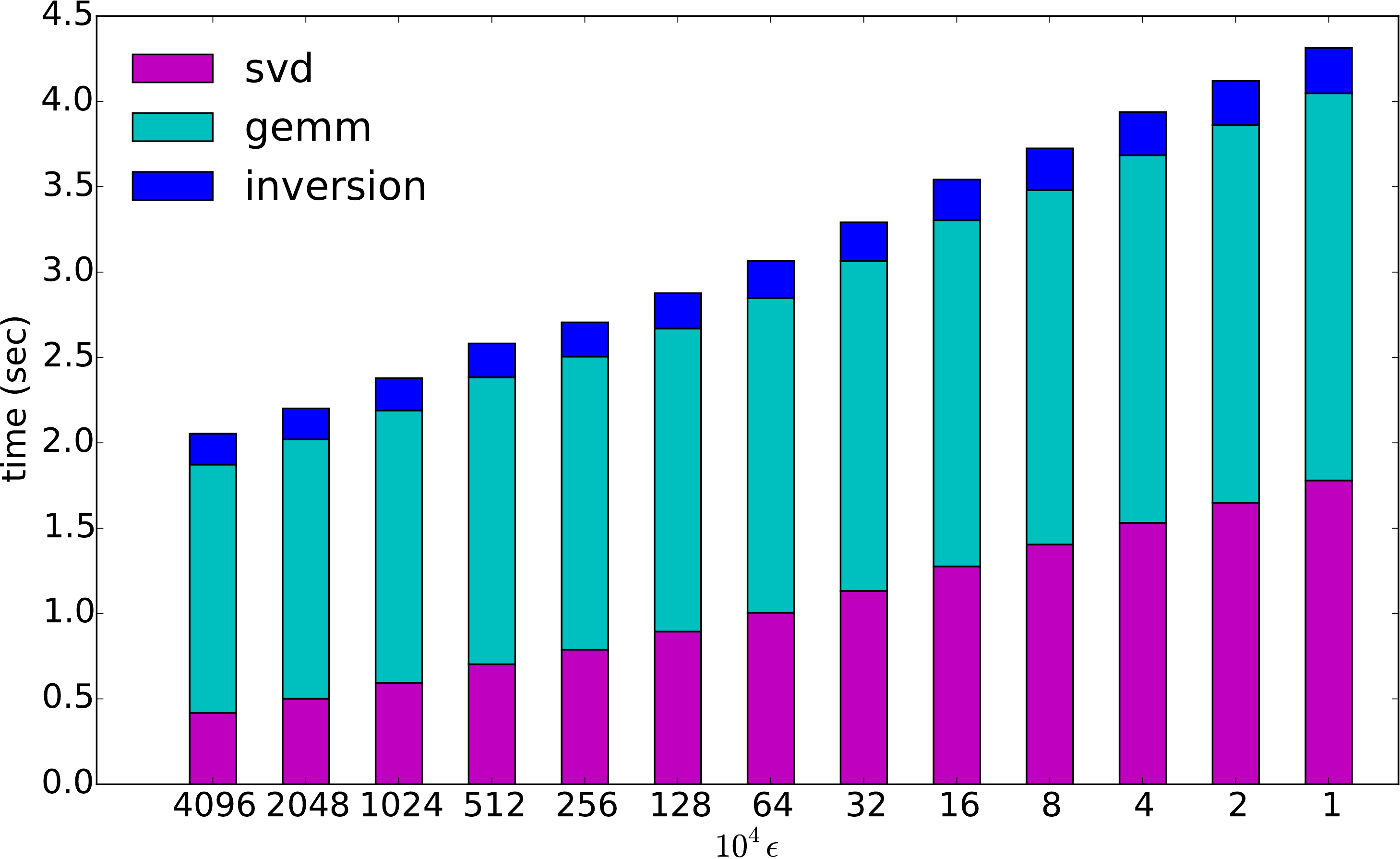}\label{fig:time_components:sub1}}
\subfloat[]{\includegraphics[width=0.49\textwidth]{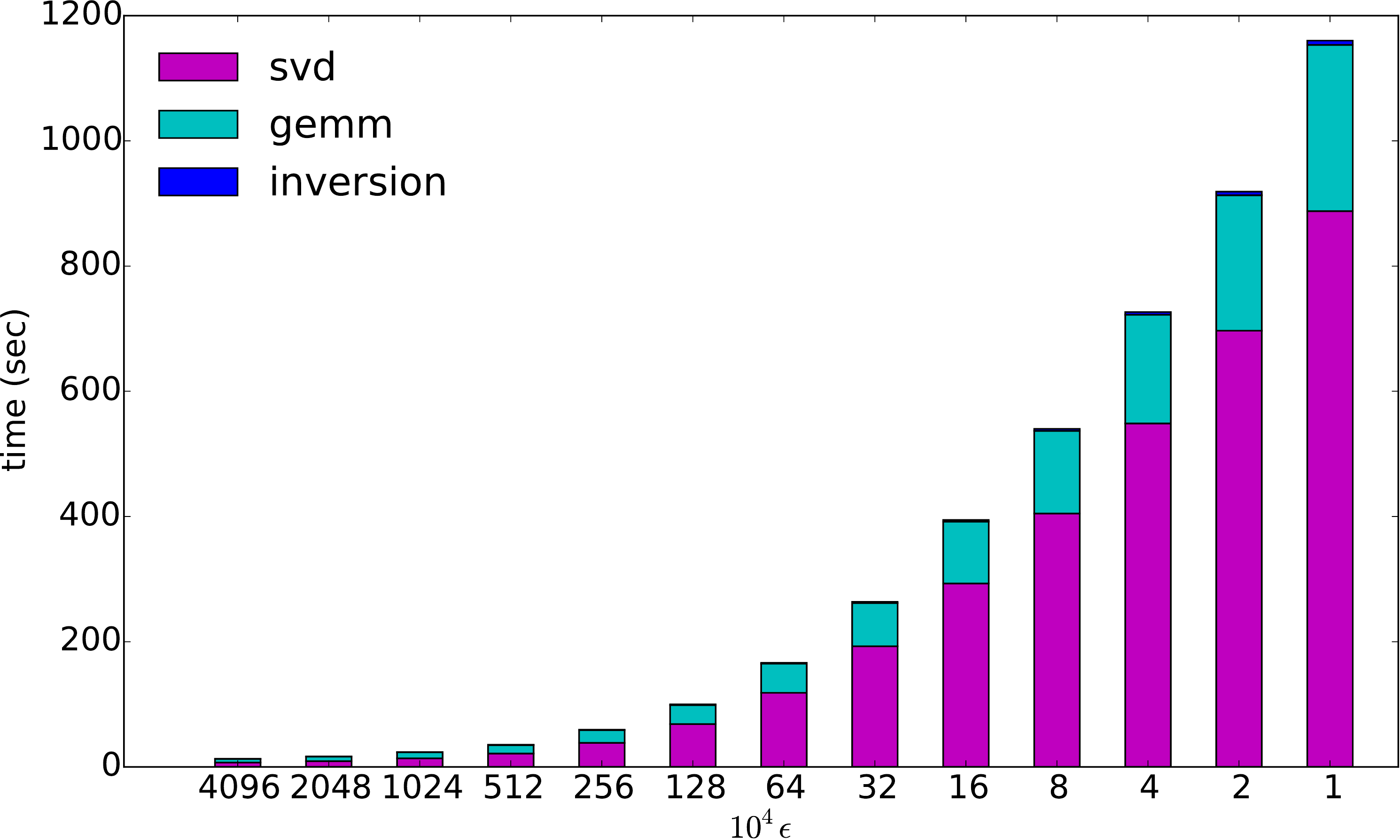}\label{fig:time_components:sub2}}
\caption{Breakdown of the time spent on different parts as a function of the low-rank approximation precision for the (a) 2D and (b) 3D cases.}
\label{fig:time_components}
\end{center}
\end{figure}

\begin{figure}[htbp]
\begin{center}
\subfloat[]{\includegraphics[width=0.45\textwidth]{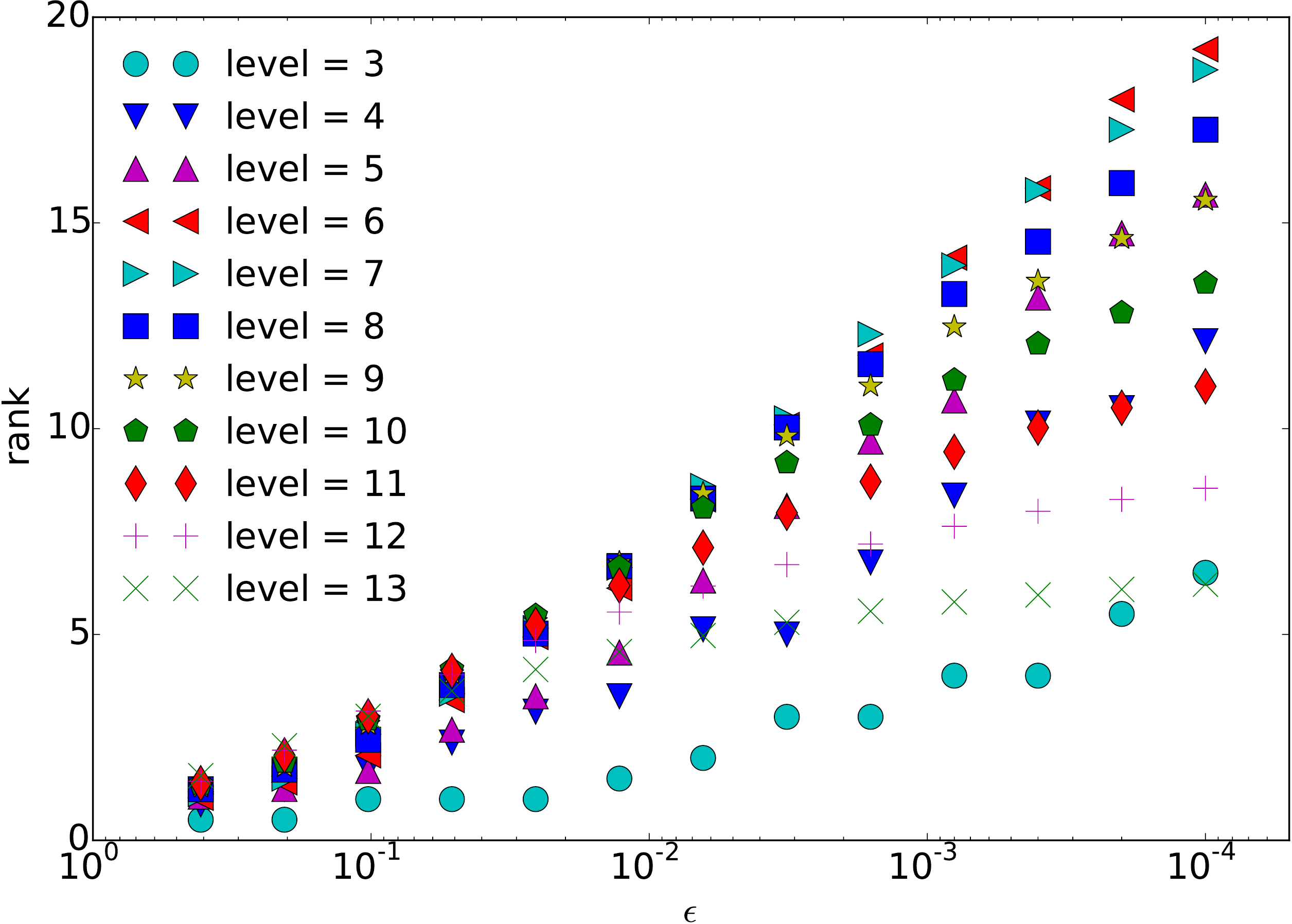}\label{fig:ranks:sub1}}
\subfloat[]{\includegraphics[width=0.45\textwidth]{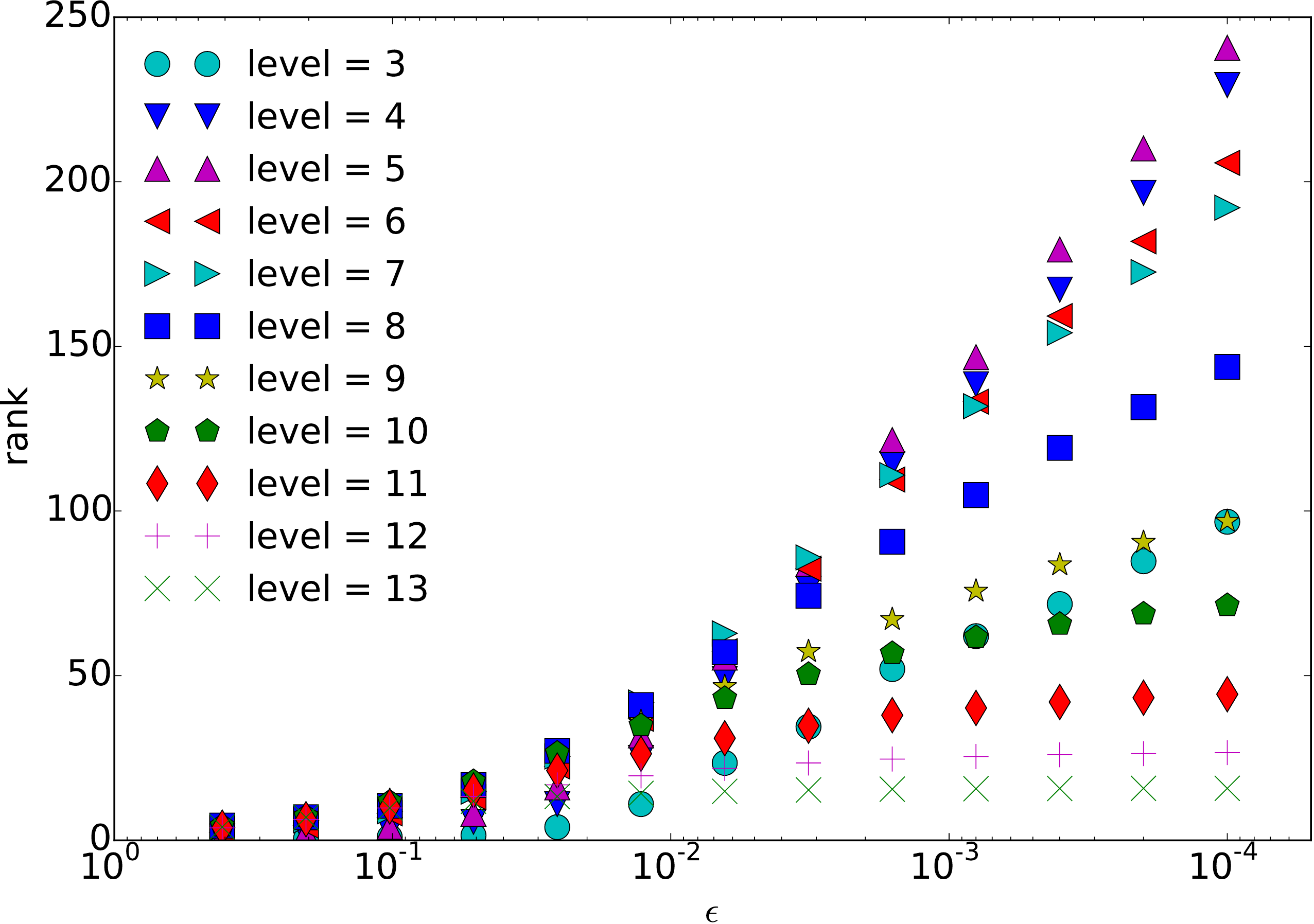}\label{fig:ranks:sub2}}
\caption{Average rank of well-separated interactions per level as a function of the low-rank approximation precision for the (a) 2D and (b) 3D cases.}
\label{fig:ranks}
\end{center}
\end{figure}

\blue{At the end of this section, we present another benchmark in which LoRaSp is used as a stand-alone solver to obtain solutions with floating point single precision accuracy. We consider \cref{eqn:Pois} discretized on uniform 2D grids of different sizes, and use the proposed solver with $\epsilon =10^{-4}$ for low-rank approximation as introduced in \cref{eqn:cutOff1}. For different problem sizes we pick depth of the $\mathcal{H}$-tree such that the average size of a super-node at leaf level is 64. The \orange{factorization and} solve times as functions of the matrix size \orange{are} depicted in \cref{fig:pois2d:sub1} demonstrating linear complexity of the solver. The relative residual of the solution, as defined in \cref{eqn:accres} in all cases is less than $10^{-6}$ (see \cref{fig:pois2d:sub2}).}

\begin{figure}[htbp]
\begin{center}
\subfloat[]{\includegraphics[width=0.45\textwidth]{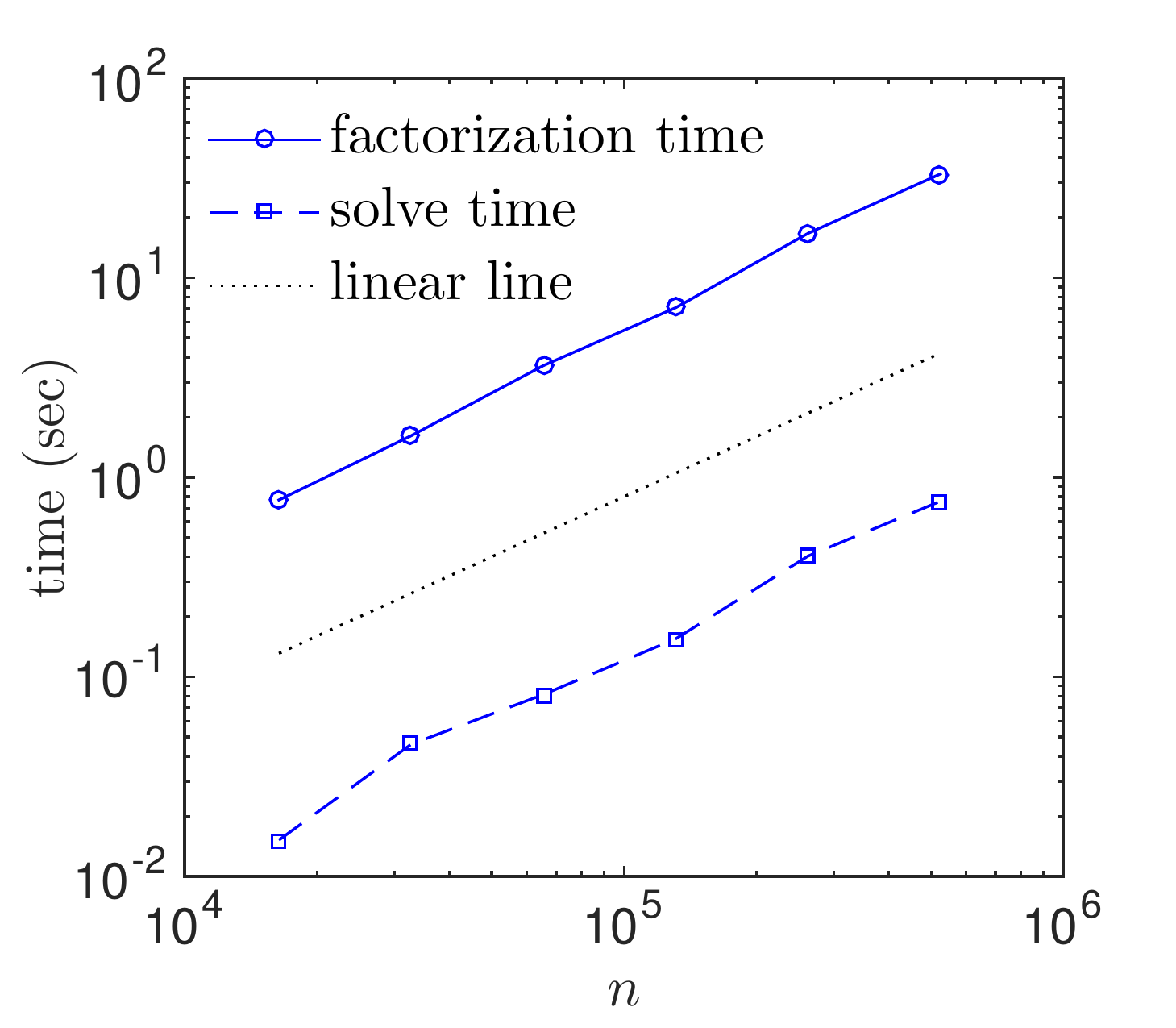}\label{fig:pois2d:sub1}}
\subfloat[]{\includegraphics[width=0.45\textwidth]{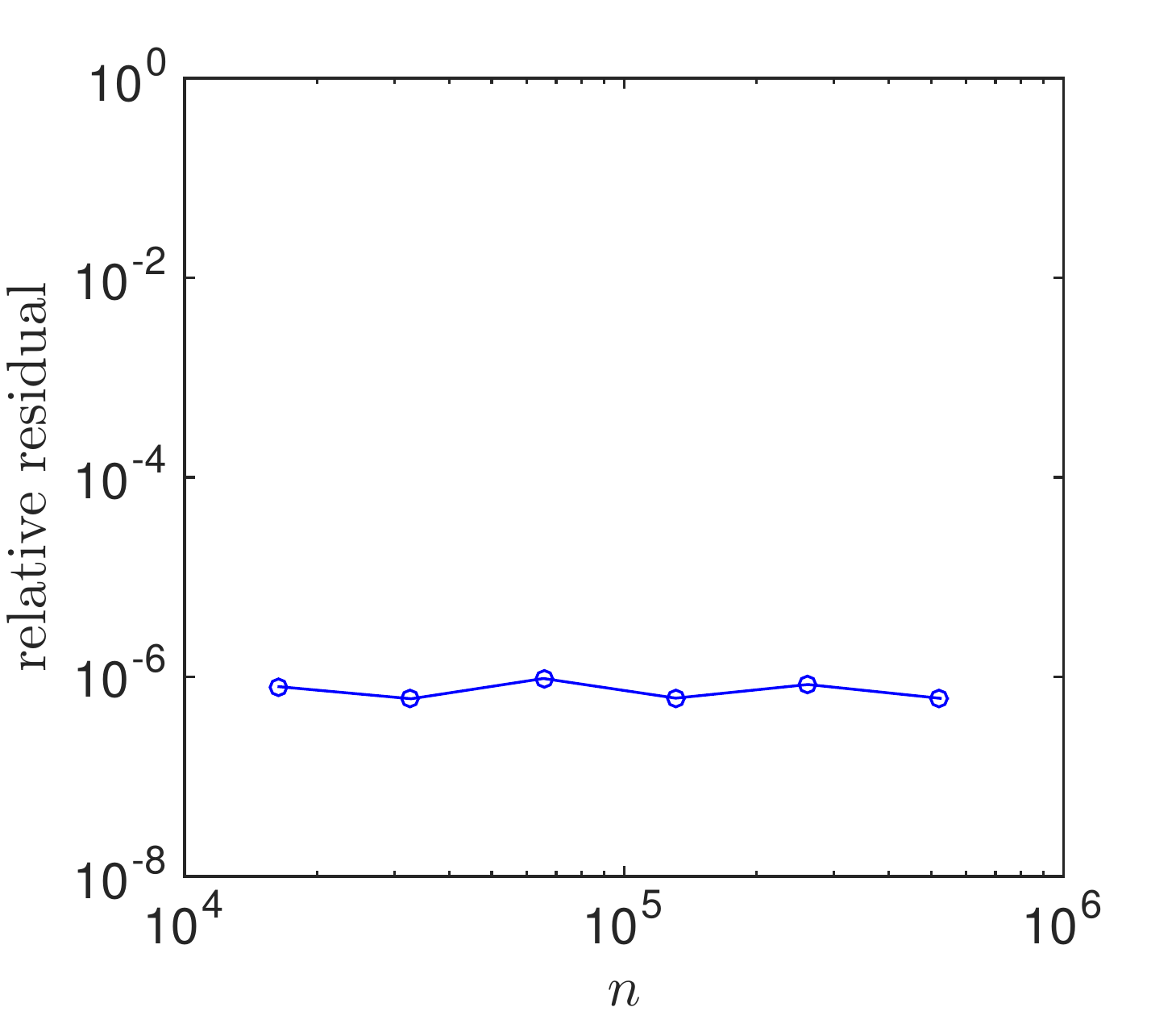}\label{fig:pois2d:sub2}}
\caption{\blue{LoRaSp \orange{factorization and} solve times as stand-alone solver for discretized Poisson equation on grids with different sizes (a), and the relative residual of the solution (b).} \orange{Low-rank precision is $\epsilon =10^{-4}$, resulting in relative residual less than $10^{-6}$ and relative error $\sim 10^{-4}$.}}
\end{center}
\end{figure}

\subsection{LoRaSp as a preconditioner} \label{sec:precond}
In \cref{sec:standalone} we showed that for a fixed low-rank approximation precision, and therefore solution accuracy, the total cost of the algorithm grows linearly with \red{the} problem size. However, as suggested by \cref{fig:convergence,fig:time_components} obtaining a high accuracy solution may be expensive for some problems. One standard remedy in that case is to use the low-accuracy solver as a high-accuracy preconditioner in conjunction with an iterative solver. This is particularly very appealing here, since the factorization part is completely separated form the solve part. Therefore, we can factorize the matrix only once, and apply the (cheap) solve part at every iteration. Here, we use the GMRES method \cite{saad1986gmres} as the iterative solver in conjunction with the proposed algorithm as a preconditioner \red{for all benchmarks. Note that for matrices with specific properties such as a symmetric positive definite (SPD) matrix, one can use a more optimized iterative method (e.g., conjugate gradient in the case of SPD matrix).}

\subsubsection{Poisson equation (structured grid)} \label{sec:Poisson}
As the first benchmark, we consider the sequence of 3D Poisson problems introduced in \cref{sec:standalone}. We use $\epsilon = 10^{-1}$ to factorize the matrix, and find an approximation $\tilde{A}$ of $A^{-1}$. Factorization and solve times are shown in \cref{fig:directSolve_varyN:sub1}. We solve the system of equation $\tilde{A}A{\boldsymbol{x}} = \tilde{A}{\boldsymbol{b}}$ through GMRES afterwards. Since the solve part of the proposed algorithm is much cheaper compared to the factorization part, each GMRES iteration is also relatively cheap. In \cref{fig:sparsity} the sparsity pattern of the original and preconditioned matrices are shown. The preconditioned matrix, $\tilde{A}A$, \blue{approaches} the identity matrix as $\epsilon$ decreases.

\begin{figure}[!htbp]
\begin{center}
\subfloat[]{\includegraphics[width=0.35\textwidth]{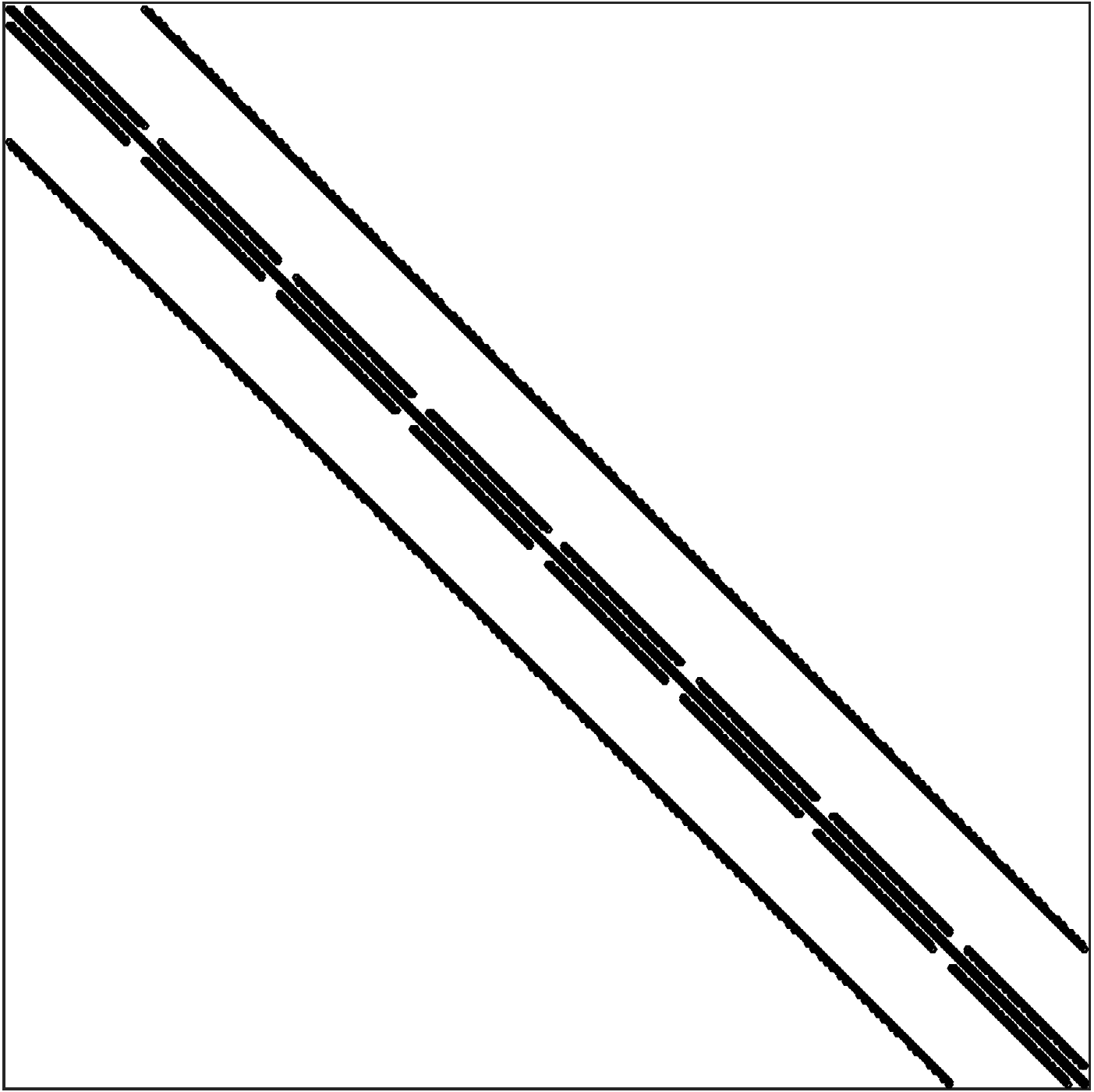}\label{fig:sparsity:sub1}} \quad
\subfloat[]{\includegraphics[width=0.38\textwidth]{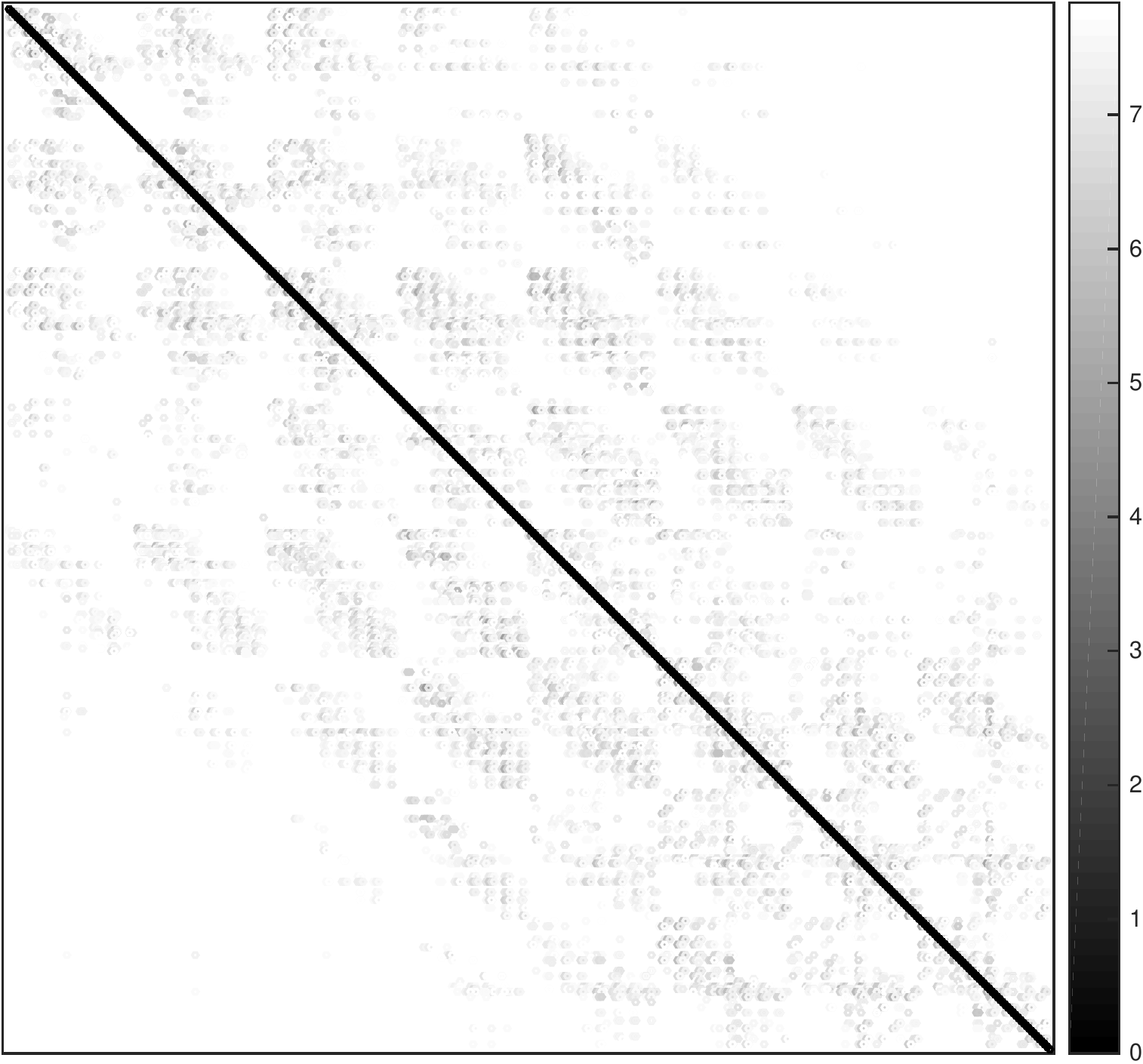}\label{fig:sparsity:sub2}}
\caption{\label{fig:sparsity}(a) Sparsity pattern of a matrix $A$ obtained from a discretization of \cref{eqn:Pois}. (b) Sparsity pattern of the preconditioned matrix $\tilde{A}A$ using LoRaSp with $\epsilon = 10^{-1}$.  A non-zero entry $a$ is colored by $-\log |a|$ \red{(i.e., larger values are darker)}. }
\end{center}
\end{figure}

In \cref{fig:residuals}, the GMRES residual as a function of the iteration number is plotted for different problem sizes, when LoRaSp with $\epsilon=10^{-1}$ is used as a preconditioner. At every iteration of the preconditioned GMRES method we have an approximation $\tilde{\boldsymbol{x}}$ of the solution. The residual in this case is defined as follows:
\begin{equation}
\text{preconditioned GMRES residual} = \frac{ \| \tilde{A}b - \tilde{A}A\tilde{x} \|_2 } { \| \tilde{A} b \|_2}
\label{eqn:GMRES_residual}
\end{equation}
The number of iterations that GMRES needs to converge slightly grows with size of the problem. This is due to growth of the condition number of the problem. \red{Similar to multi-grid methods, we can use specific knowledge about the underlying PDE and discretization to obtain problem-size independent convergence. This is the topic of our future work, and is not discussed in this paper.} In \cref{fig:cond} the condition number, $\kappa(A)$, is plotted as a function of the matrix size. The condition numbers are approximated using the 1-norm \cite{hager1984condition, higham2000block}. Note that for a matrix of size $n$ corresponding to the second order finite difference discretization of the Poisson equation, we expect the condition number to grow as $n^{2/3}$. The $n^{2/3}$ trend is also depicted in \cref{fig:cond}.

\begin{figure}[htbp] \centering
\includegraphics[width=0.9\textwidth]{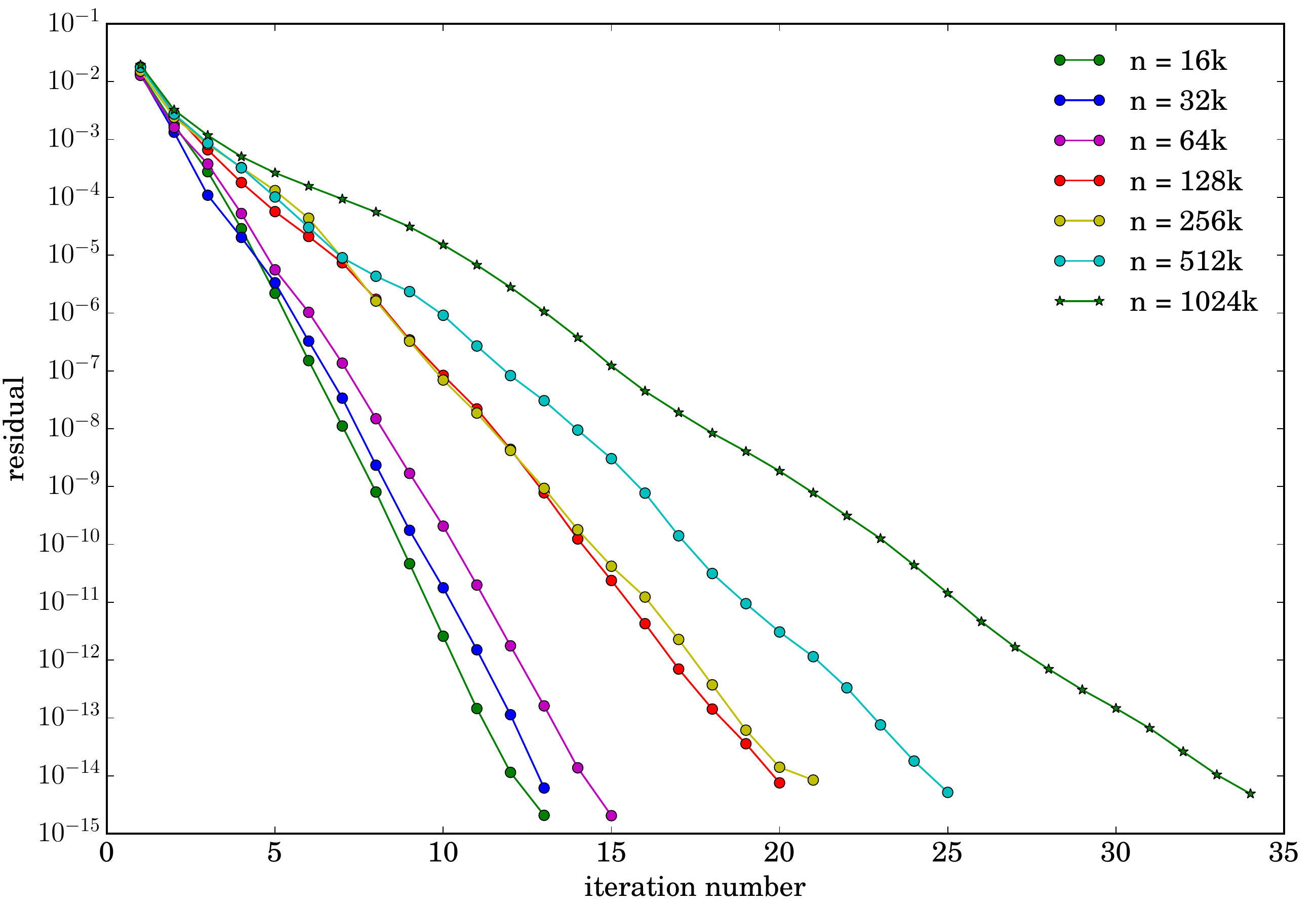}
\caption{GMRES residual as a function of the iteration number for different problem sizes. LoRaSp is used as a preconditioner with $\epsilon = 10^{-1}$.}
\label{fig:residuals}
\end{figure}

\begin{figure}[htbp] \centering
\includegraphics[width=0.4\textwidth]{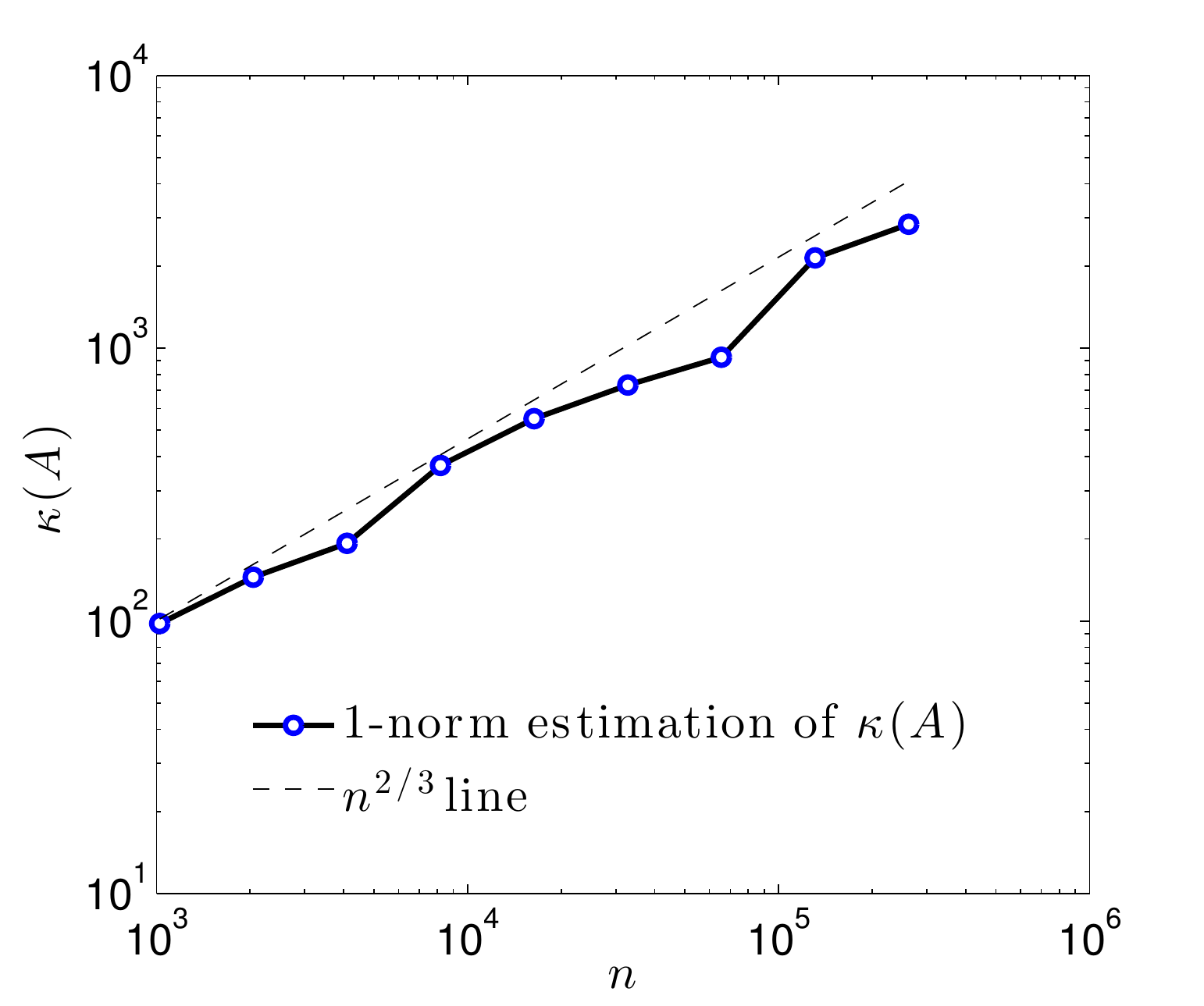}
\caption{Condition number versus matrix size for a sequence of matrices obtained from the discretization of \cref{eqn:Pois}.}
\label{fig:cond}
\end{figure}

\subsubsection{Variable coefficient Poisson equation (structured grid)}
As our next benchmark we consider the variable coefficient Poisson equation with periodic boundary conditions discretized on a three-dimensional uniform grid:
\begin{equation}
\nabla \cdot \left( \phi \nabla T \right) = f
\label{eqn:VCP}
\end{equation}
In the above equation, the scalar fields $\phi$ and $f$ are given, and we solve for $T$, similar to \cref{eqn:Pois}. We consider three cases for the coefficient field, $\phi$:
\begin{itemize}
\item {\bf case 1}: At each point of the domain, $\phi$ is drawn from a uniform distribution, $\text{unif}(0,1)$, independently.
\item {\bf case 2}: At each point of the domain, $\rho$ is drawn from a uniform distribution, $\text{unif}(0,1)$, independently. $\phi$ is then defined as $\phi = \frac{1}{\rho}$.
\item {\bf case 3}: At each point of the domain, $\phi$ is drawn from a uniform distribution, $\text{unif}(-1,1)$, independently.
\end{itemize}
For the two first cases, the corresponding matrices are symmetric negative definite. Case 2 shows up in the numerical simulation of a variable-density flow in the low-Mach number limit \cite{choi1993application, guillard1999behaviour, pouransari2015parallel}, where in that case $T$ and $\rho$ are hydrodynamic pressure and density of the flow, respectively. The third case, however, results in an indefinite matrix. In \cref{fig:VCP_prop:sub1}, the norm of the eigenvalues of the matrices corresponding to a $16^3$ grid for all cases are shown. In the third case, nearly half of the eigenvalues are positive, and half of them are negative, corresponding to the left and right sides of the red curve in \cref{fig:VCP_prop:sub1}.

\Cref{fig:VCP_prop:sub2} shows the 1-norm approximation of the condition number of the matrices for all cases. Evidently, a larger grid results in a higher condition number. Also, as expected, the condition number in case 2 is higher than case 1, and the condition number in case 3 is higher than case 2.

\begin{figure}[htbp]
\begin{center}
\subfloat[]{\includegraphics[width=0.47\textwidth]{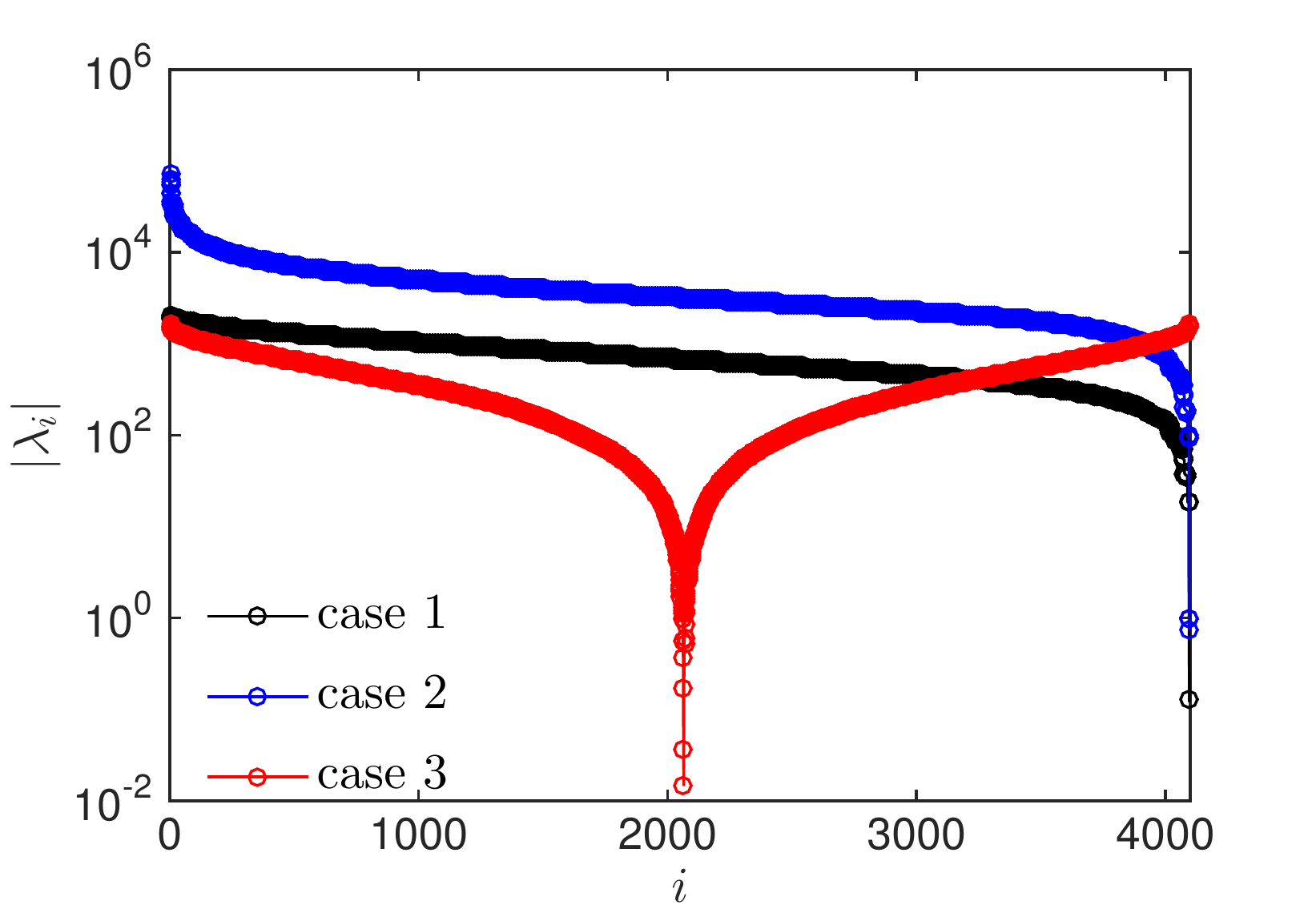}\label{fig:VCP_prop:sub1}} \quad
\subfloat[]{\includegraphics[width=0.47\textwidth]{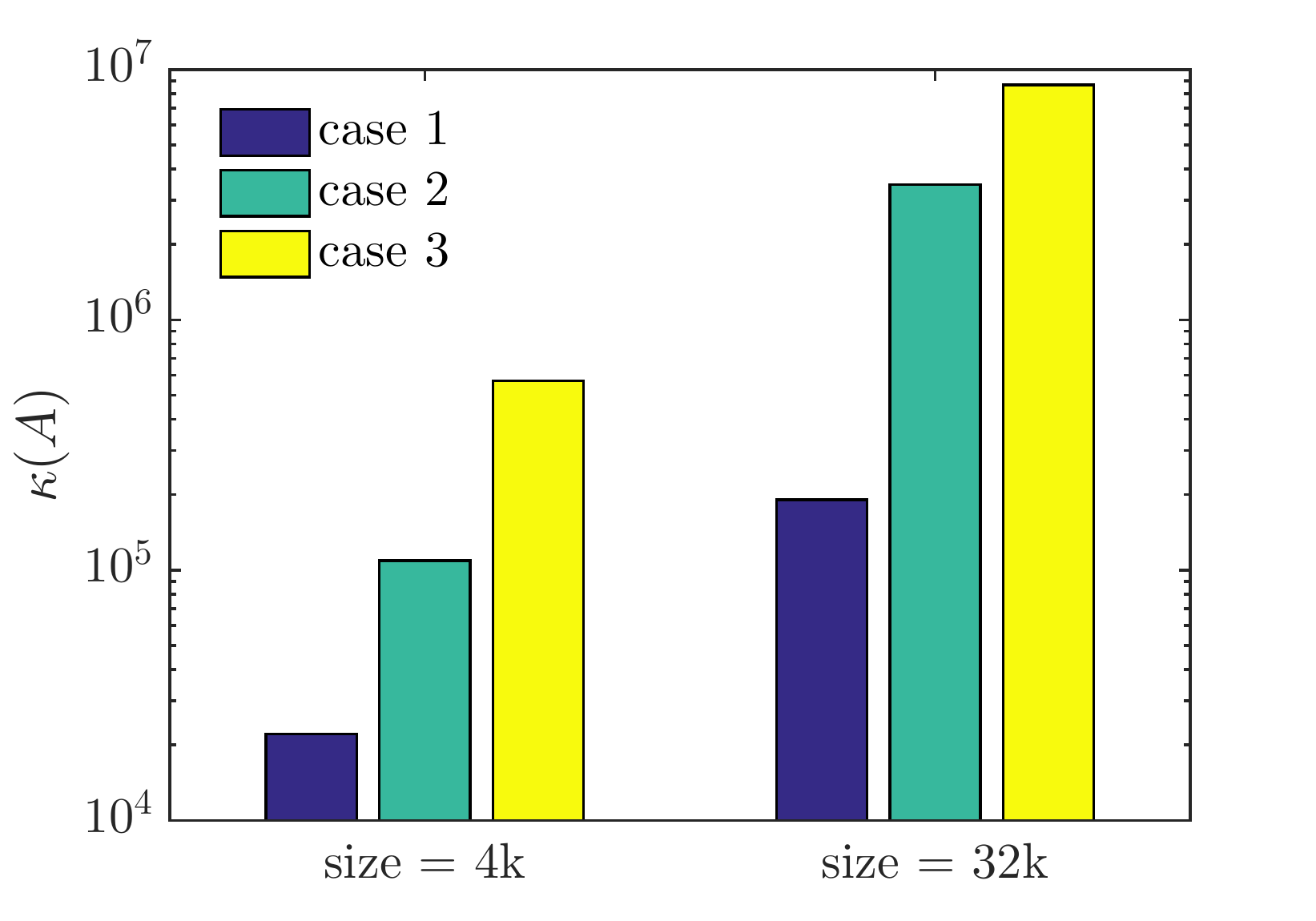}\label{fig:VCP_prop:sub2}}
\caption{\label{fig:VCP_prop} Properties of the matrices corresponding to the discretization of \cref{eqn:VCP}: (a) Absolute value of the eigenvalues for the $16^3$ grid. (b) The 1-norm approximation of the condition number for $16^3$ and $32^3$ grids.}
\end{center}
\end{figure}

We used LoRaSp as a preconditioner in conjunction with GMRES. A summary of the results for the two first cases is provided in \cref{tab:VCP}. The convergence criterion of  GMRES, with residual defined in \cref{eqn:GMRES_residual}, is set to $10^{-14}$. For cases 1 and 2, we used LoRaSp with low-rank precision $\epsilon = 10^{-1}$ as defined in \cref{eqn:cutOff1}. Similar to \cref{sec:Poisson}, we increase the depth of the $\mathcal{H}$-tree with $\log n$, where $n$ is the size of matrix. Similar to the results presented in \cref{sec:Poisson}, the factorization time has an almost linear complexity with the size of the matrix. As depicted in \cref{fig:VCP_prop:sub2}, the condition number grows rapidly from case 1 to case 2, and with the size of the matrix. This explains a slight growth of the number of iterations, and relative error.

\begin{table}[!htbp]
  \begin{center}
  \begin{tabular}{lccccc}
      \toprule
      \bf case &\bf  fact.\ time &\bf GMRES time & \bf tot.\ time& \bf \# iters & \bf rel.\ error\\\midrule
	{\bf 1 (4k)} & 0.44 & 0.06 & 0.50& 10 & 1.7e-11\\
	{\bf 1 (32k)} & 4.62 & 1.07 & 5.69& 15 & 2.6e-10\\
	{\bf 1 (256k)} & 37.08 & 16.21 & 53.29& 15 & 8.6e-10\\
	{\bf 2 (4k)} & 0.40 & 0.07 & 0.47& 12 & 6.6e-11\\
	{\bf 2 (32k)} & 3.72 & 1.31 & 5.03& 20 & 9.2e-10\\
	{\bf 2 (256k)} & 33.52 & 19.44 & 52.95& 30 & 1.2e-9\\ \bottomrule
  \end{tabular}
  \caption{GMRES performance using LoRaSp as a preconditioner to solve a variable coefficient Poisson equation. Matrix sizes, corresponding to $16^3$, $32^3$, and $64^3$ grids are written in parentheses. Times are reported in seconds.}
  \label{tab:VCP}
  \end{center}
\end{table}

Case 3 corresponds to an indefinite matrix, with a large condition number. This is typically a more difficult problem, compared to the first two cases. Since case 3 is inherently a harder problem compared to cases 1 and 2, we choose a higher low-rank precision $\epsilon = 10^{-3}$. In \cref{fig:VCP_indef:sub1}, the averaged rank of interactions for each level is plotted. We also plot the compression ratio for each level, which is defined as follows:
\begin{equation}
\text{compression ratio}~(l) = \frac{ \langle \text{interaction rank} \rangle_l} { \langle \text{size of super-nodes} \rangle_l},
\end{equation} 
where $\langle~\rangle_l$ denotes averaging in level $l$ of the $\mathcal{H}$-tree. Note that it is clear from \cref{fig:VCP_indef:sub1} that even though the rank is increasing, the compression ratio is approximately 0.6--0.7 for all levels. Hence, the algorithm takes advantage of low-rank structures appropriately.

In \cref{fig:VCP_indef:sub2}, the preconditioned GMRES residual defined in \cref{eqn:GMRES_residual} is plotted as a function of the iteration number. GMRES for case 3 does not converge \red{when no preconditioner or a diagonal preconditioner is used.} We also applied ILU \cite{saad1994ilut} as a preconditioner for GMRES. We tried various (including very large) values for the fill parameter in ILU. No convergence was obtained when ILU is used as a preconditioner.

\begin{figure}[htbp]
\begin{center}
\subfloat[]{\includegraphics[width=0.42\textwidth]{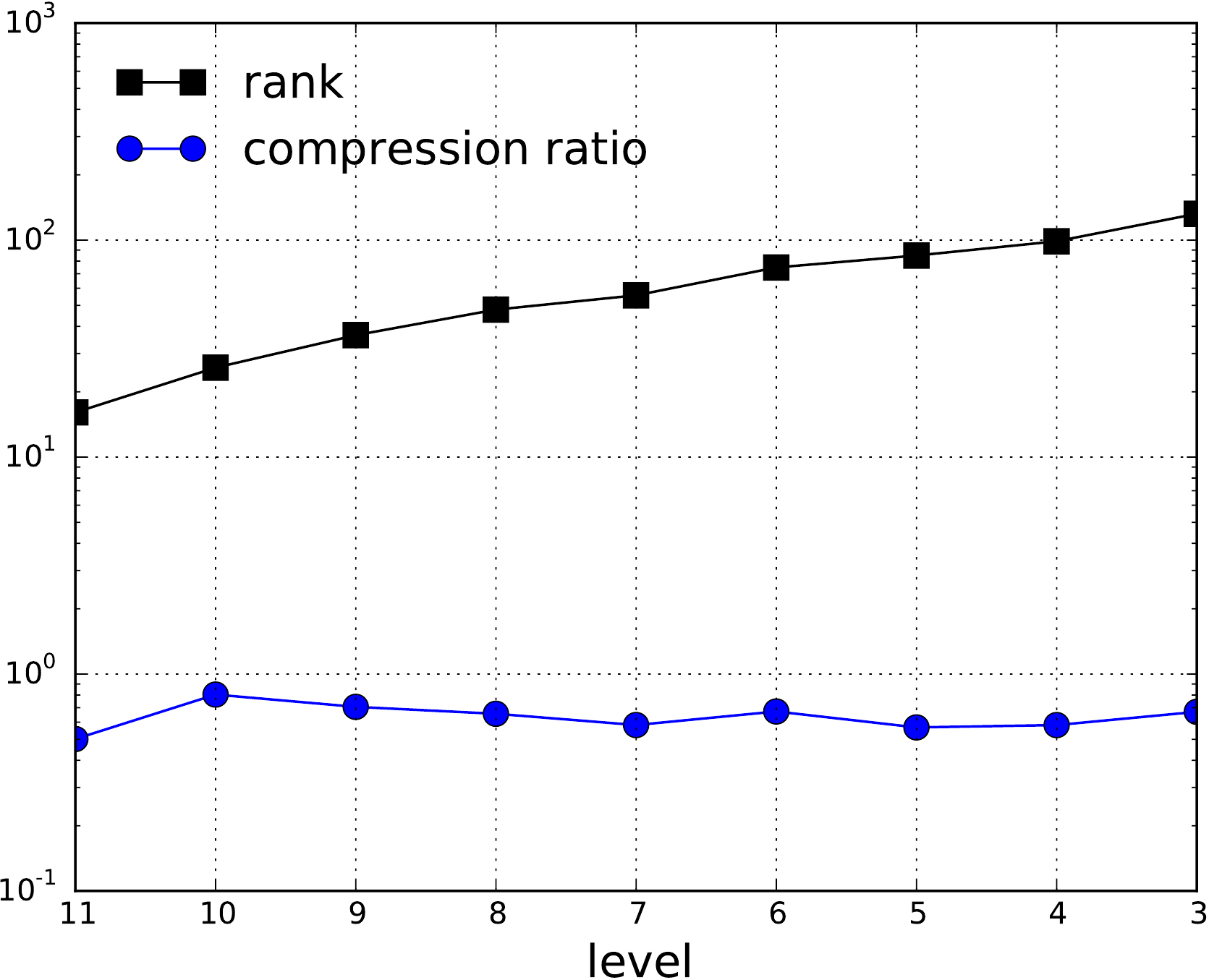}\label{fig:VCP_indef:sub1}} \quad
\subfloat[]{\includegraphics[width=0.55\textwidth]{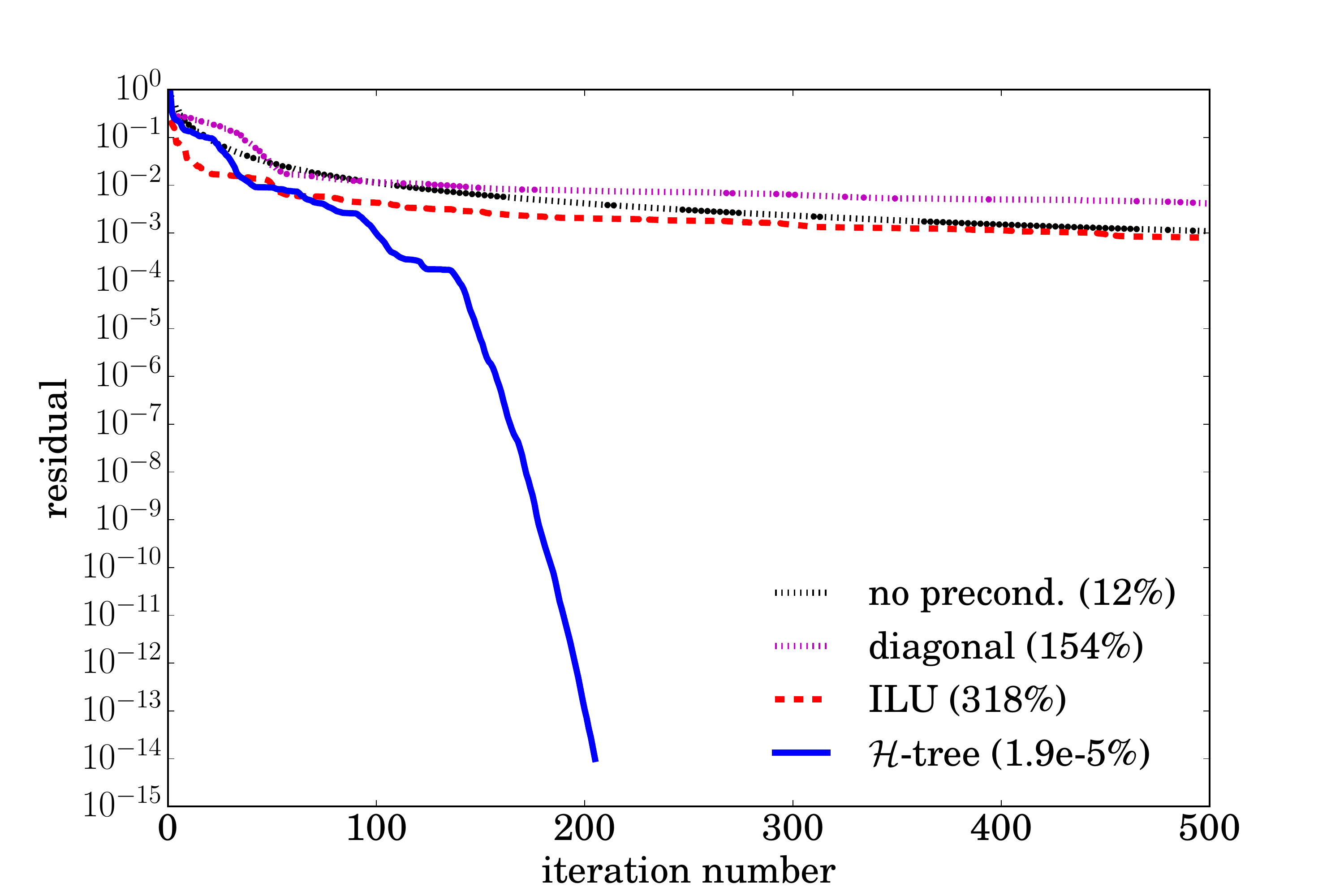}\label{fig:VCP_indef:sub2}}
\caption{\label{fig:VCP_indef} Results corresponding to the case 3 for a $32^3$ grid. (a) Averaged rank and compression ratio in the $\mathcal{H}$-tree. (b) GMRES residual as a function of the iteration number using various preconditioners. The relative error at the end of the iterations is shown in parentheses for each case. Although the residuals may be small, the relative error for some of the methods is very large.}
\end{center}
\end{figure}

\subsubsection{Elasticity equation (unstructured grid)}
Our \red{next} benchmark is obtained from an unstructured grid (see \cite{ghosh2009feti}) to solve the three-dimensional elasticity equation:
\begin{equation}
(\lambda+\mu) \nabla( \nabla . {\boldsymbol{u}} ) + \mu \nabla^2 {\boldsymbol{u}} + {\boldsymbol{F}} = 0
\end{equation}
The matrix is symmetric with size $n=334,956$, and total of $10,977,198$ non-zero entries. This problem is significantly more difficult in comparison to the previous benchmark\red{s}. We used various preconditioning strategies to evaluate the performance of our proposed solver. In \cref{tab:stiffness} a summary of the results is provided. We consider $10^{-12}$ to be the convergence criterion for the GMRES residual, and consider a maximum of 500 iterations.

When no preconditioner is used, GMRES does not converge, and we obtain a solution with 2\% relative error after 500 iterations. Employing a diagonal preconditioner (i.e., ignore all non-diagonal entries of the matrix, and approximate $A^{-1}$ with the inverse of its diagonal part) also does not lead to convergence. A 1.4\% relative error after 500 GMRES iterations is obtained, which is an improvement in comparison to the case without a preconditioner.

Next, we tried ILU as a preconditioner for GMRES. We used \red{dual}-threshold ILU with drop tolerance fixed equal to the GMRES convergence precision, and varying fill values (1, 2, etc.). ILU with fill value of 1 does not converge after 500 iterations; however, for fill values greater than 1, convergence is obtained before 500 iterations. Increasing the fill value leads to \blue{a} higher factorization time.

Finally, we used the proposed algorithm as a preconditioner. For the low-rank approximation, we used a variation of \cref{eqn:cutOff1}. Consider we want to find a low-rank approximation of a block $B$ with singular-value decomposition $B = USV^{\intercal}$. We keep the first $k$ singular-values (and therefore, singular vectors) such that, $k$ is the smallest integer that:
\begin{equation}
\frac{ \| B - U_k S_k V_k^{\intercal} \|_F } { \| \text{all levels below} \|_F } < \epsilon
\label{eqn:curOff6}
\end{equation}
The subscript $k$ in $U_k$ and $V_k$ means keeping only the first $k$ columns, and in $S_k$ means keeping the first $k$ singular-values. $\|~\|_F$ refers to the Frobenius norm. \red{Therefore,} $\| \text{all levels below} \|_F$ refers to square root of the sum of Frobenius norm squared of all blocks at the current level as well as the levels below. Both the above criteria and the one in \cref{eqn:cutOff1} work properly, and lead to the same conclusions; however, the above method is slightly more efficient. For this benchmark, we used a sequence of decreasing values for $\epsilon$ in \cref{eqn:curOff6}: $\epsilon_1 = 1024\times10^{-7}, \epsilon_2 = 256\times10^{-7},  \epsilon_3 = 64\times10^{-7},  \epsilon_4 = 16\times10^{-7},  \epsilon_5 = 4\times10^{-7}, \epsilon_6 = 1\times10^{-7}$.

\begin{table}[!htbp]
  \begin{center}
  \begin{tabular}{cccccc}
      \toprule
      \bf precond. &\bf  fact.\ time &\bf GMRES time & \bf tot.\ time& \bf \# iters & \bf rel.\ error\\\midrule
	{\bf none} & 0 & 115.9 & 115.9& 500 & 2.1e-2\\
	{\bf diagonal} & 0.02 & 116.3 & 116.3& 500 & 1.4e-2\\
	{\bf ILU 1} & 98.6 & 156.7 & 255.3& 500 & 5.3e-06\\ \midrule
	{\bf ILU 2} & 211.1 & 132.0 &343.1& 408&4.9e-10\\
	{\bf ILU 3} & 313.0 &102.2 &415.2& 309&5.7e-10\\
	{\bf ILU 4} & 399.9 &82.3 &482.2& 245&3.2e-10\\
	{\bf $\mathcal{H}$-tree $\epsilon_1$} & 90.6 &298.9&389.5& 467&2.3e-10\\
	{\bf $\mathcal{H}$-tree $\epsilon_2$} & 93.5 &299.6&393.1& 466&2.2e-10\\
	{\bf $\mathcal{H}$-tree $\epsilon_3$} & 114.6 &295.5& 410.1& 367& 2.4e-10\\
	{\bf $\mathcal{H}$-tree $\epsilon_4$} & 166.0 &130.6&296.6& 144&3.5e-10\\
	{\bf $\mathcal{H}$-tree $\epsilon_5$} & 379.8 &74.9&454.7& 46&1.3e-10\\
	{\bf $\mathcal{H}$-tree $\epsilon_6$} & 961.1 &43.8&1004.9& 16& 2.1e-10\\ \bottomrule
  \end{tabular}
  \caption{GMRES performance using various preconditioners for a matrix of size 330k with more than 10 million non-zeros obtained from a 3D unstructured discretization of the elasticity equation. Times are reported in seconds.}
  \label{tab:stiffness}
  \end{center}
\end{table}

\Cref{fig:residual_stiffness} illustrates the variation of residual versus the number of iterations using various preconditioners. Clearly, diagonal preconditioner accelerates convergence compared to the case with no preconditioner. ILU preconditioners bring about convergence faster than diagonal. The $\mathcal{H}$-tree based preconditioners lead to a significant acceleration in convergence. Decreasing $\epsilon$ (and similarly increasing fill value in the ILU) results in a shorter iteration time at the cost of a more expensive factorization as listed in \cref{tab:stiffness}. In practice, one should pick an intermediate value for $\epsilon$ (and similarly for the fill value when ILU is used) to get an optimal total runtime (i.e., factorization + GMRES iterations).

\begin{figure}[htbp] \centering
\includegraphics[width=0.8\textwidth]{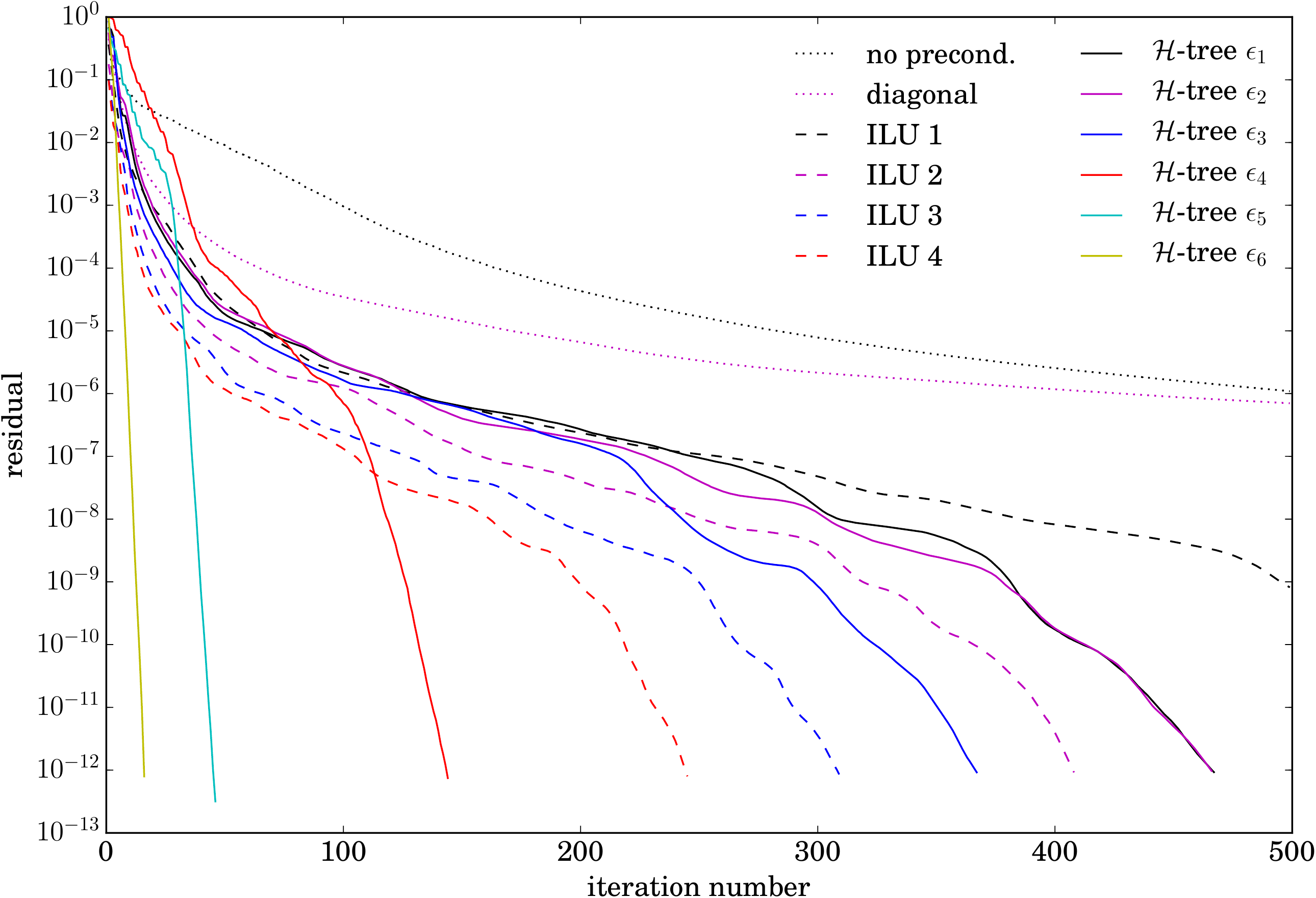}
\caption{GMRES residual as a function of iteration number using various preconditioners for a matrix of size 330k with more than 10 million non-zeros obtained from a 3D unstructured discretization of the elasticity equation.}
\label{fig:residual_stiffness}
\end{figure}

In \cref{fig:optimization_stiffness} the breakdown of the total time is plotted for the cases that convergence is achieved. The non-monotonic functionality of the total time as a function of $\epsilon$ is clear form this figure. For instance, for this set of $\epsilon$ values, $\epsilon_4 = 16\times10^{-7}$ has the optimal time, which is more efficient in comparison to the best time of the ILU. Note that the current implementation of the algorithm is completely sequential. There are various optimizations to enhance the performance of the solver. We discuss some of the possible optimizations in \cref{sec:conclusion}.
Typically, $\epsilon$ meets its optimal value when  factorization and \blue{iteration} times are almost equal. This is also evident in \cref{fig:optimization_stiffness}.

\begin{figure}[htbp] \centering
\includegraphics[width=0.6\textwidth]{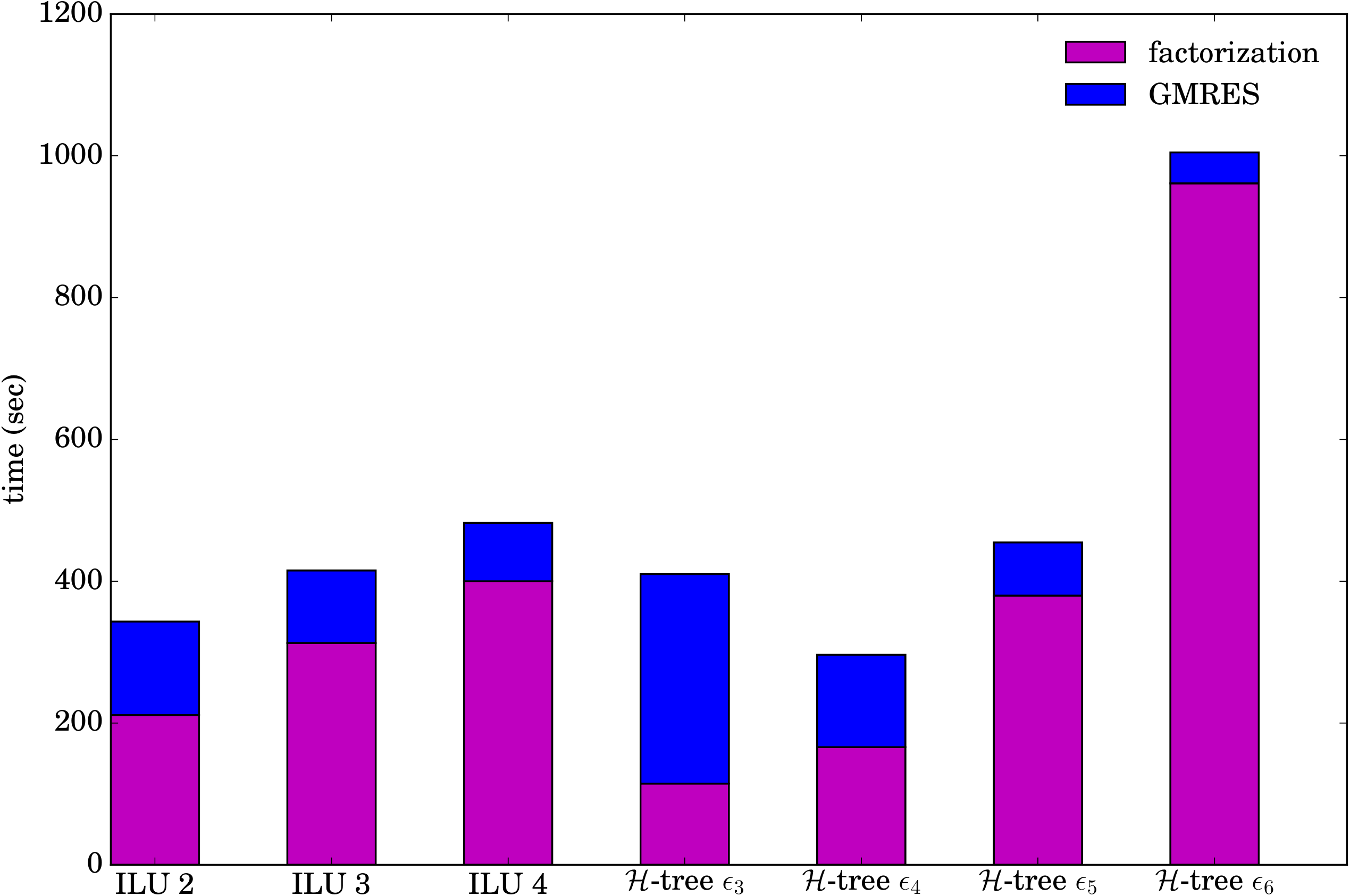}
\caption{Total solve time using various preconditioners for a matrix of size 330k with more than 10 million non-zeros obtained from a 3D unstructured discretization of the elasticity equation.}
\label{fig:optimization_stiffness}
\end{figure}

\subsubsection{Advection-diffusion problem (non-symmetric linear system)}
\red{In our last benchmark, we consider the advection-diffusion problem in a cubic domain of size $L^3$ with Dirichlet boundary conditions. This problem is governed by the following PDE.
\begin{equation}
\frac{\partial T}{\partial t} + \boldsymbol{U} \cdot \nabla T = \nu \nabla^2 T + s
\end{equation}
In the above equation $\boldsymbol{U}=(U_x,U_y,U_z)$ is a constant velocity vector, $\nu$ is the diffusion coefficient, and $s$ is a source term. Using an implicit numerical time integration method, and assuming $U_x = U_y = U_z = U$, the above equation transforms to the following non-dimensional form
\begin{equation} \label{eqn:advdiff}
\sigma T + \mathcal{R} \nabla T - \nabla^2 T = g,
\end{equation}
where $\sigma = \frac{L^2}{\nu \Delta t}$ and $\mathcal{R} = \frac{L U}{\nu}$ assuming $\Delta t$ is the time step.}

\red{Discretizing \cref{eqn:advdiff} using a central finite differencing method results in symmetric and skew-symmetric matrices for the diffusion and advection terms, respectively. Therefore, the full system of equations is represented by a non-symmetric matrix.}

\red{We verified the accuracy of the $\mathcal{H}$-tree solver (similar to \cref{fig:convergence}) for a case with $\sigma = 0$ and $\mathcal{R} = 1$ on a $32^3$ grid. \Cref{fig:nonSym:sub1} shows the error and residual as a function of the low-rank precision parameter $\epsilon$ as defined in \cref{eqn:cutOff1}. Similar to the symmetric case, residual and error are proportional to $\epsilon$.}

\red{Furthermore, we compare the convergence of the GMRES solver for the advection-diffusion problem when $\mathcal{H}$-tree and ILU are used as preconditioner. 
We consider a sequence of problems with $\sigma=1$ and varying $\mathcal{R}$ on a $32^3$ grid. In \cref{fig:nonSym:sub2} the numbers of GMRES iterations required to converge to a solution with residual less than $10^{-10}$ are shown for different cases. For the $\mathcal{H}$-tree preconditioner a low-rank precision of $\epsilon = 10^{-1}$ is used, while for the ILU drop tolerance is equal to the GMRES convergence precision and fill parameter is set to 3 (the minimum fill parameter such that all cases converge with less than 500 iterations). A 1-norm approximation of the condition number of the system is also illustrated in \cref{fig:nonSym:sub2}. For the problem studied here, the condition number (and therefore, the number of GMRES iterations) is a non-monotonic function of $\mathcal{R}$. $\mathcal{H}$-tree preconditioner exhibits a stable number of iterations and accuracy for all cases. The ILU preconditioner, however, gives rise to a larger number of iterations and higher variation with $\mathcal{R}$. Furthermore, for the case with $\mathcal{R} = 1024$ the ILU preconditioner results in a solution with final residual (defined in \cref{eqn:accres}) of order $10^{-2}$, while the preconditioned residual (defined in \cref{eqn:GMRES_residual}) is less than $10^{-10}$. Such discrepancy means the preconditioner (in this case ILU) is ill-conditioned.}

\begin{figure}[htbp]
\begin{center}
\subfloat[]{\includegraphics[width=0.48\textwidth]{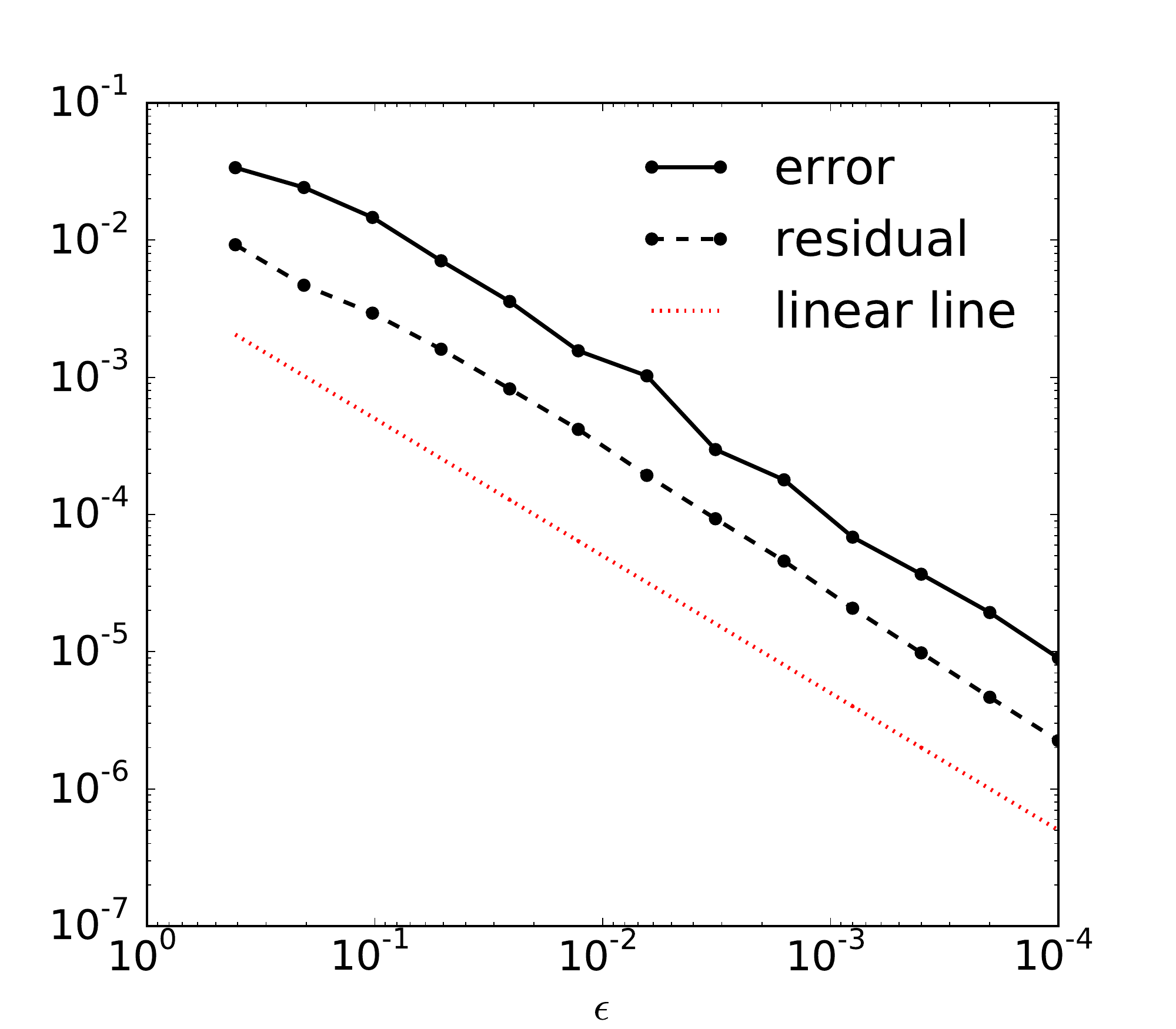}\label{fig:nonSym:sub1}} \quad
\subfloat[]{\includegraphics[width=0.47\textwidth]{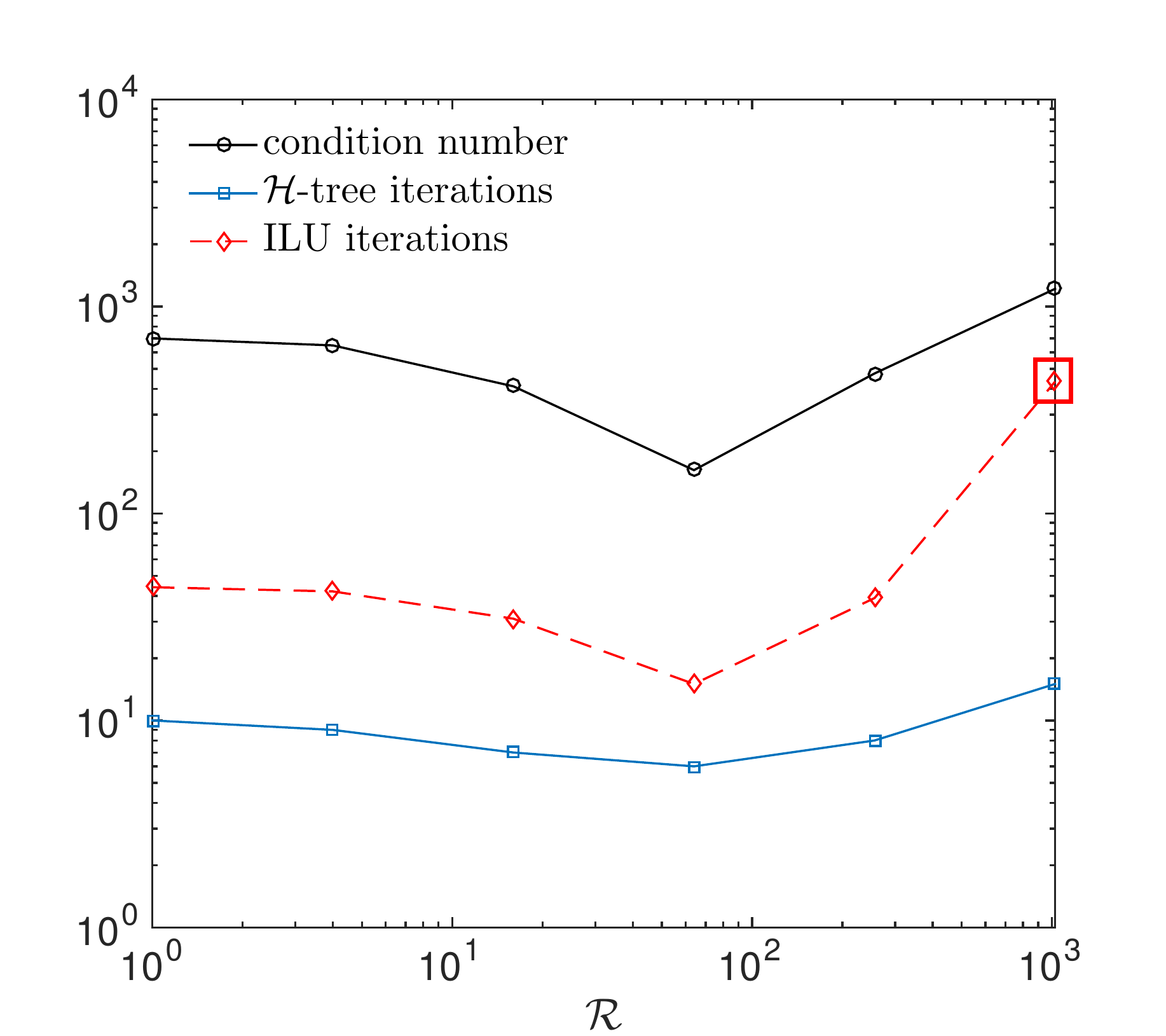}\label{fig:nonSym:sub2}}
\caption{\label{fig:nonSym}\red{(a) Error and residual of the solution of \cref{eqn:advdiff} as a function of low-rank precision $\epsilon$ using $\mathcal{H}$-tree solver. (b) Number of GMRES iterations using ILU and $\mathcal{H}$-tree as preconditioner. For $\mathcal{R}=1024$, the ILU preconditioner results in a final residual $10^{-2}$, while the pre-conditioned residual is less than $10^{-10}$. This case is highlighted by a red square in the figure. The 1-norm approximation of the matrix condition number is also plotted.}}
\end{center}
\end{figure}

\section {Conclusion and future works} \label{sec:conclusion} 
We proposed a new algorithm to solve sparse linear systems \red{with numerically low-rank structures} in linear time. The algorithm is based on the LU factorization of the sparse matrix, where the \red{matrices} $L$ and $U$ are computed and stored using a hierarchical low-rank structure. The accuracy of the factorization is determined a priori. For a precision tolerance $\epsilon$, the complexities of the factorization cost and memory are $\mathcal{O} \left(n\log^2 {1}/{\epsilon} \right)$ and $\mathcal{O} \left( n \log {1}/{\epsilon}\right)$, respectively.

The proposed algorithm is fully algebraic and preserves the sparsity of the original matrix \red{during the elimination}. Therefore, it can be considered as an extension to the ILU method. In the ILU factorization, new fill-ins are ignored. In the proposed algorithm, however, new fill-ins are compressed using low-rank approximations. Compressed fill-ins form a new set of equations ---at a coarser level--- which \red{are} factorized through elimination. Furthermore, the multilevel process of the factorization is similar to AMG, where the original system is solved at different levels (grid size).

We provided various benchmarks to illustrate the performance of the proposed algorithm. We used matrices obtained from the discretization of the Poisson and elasticity equations on structured and unstructured grids, respectively. \red{A non-symmetric benchmark corresponding to the advection-diffusion problem is also presented.} The proposed factorization method is used both as a stand-alone solver with tunable accuracy, and as a preconditioner in conjunction with \red{an iterative method (e.g., GMRES)}.

There are various aspects of the method which are general, and can be modified to optimize the solver for \red{particular} matrices without loosing the properties demonstrated in this paper. Here is \red{a} list of \red{some} aspects that can be modified in the algorithm:
\begin{enumerate}
\item Partitioning of the sparse matrix graph is generic. Higher quality partitioning in general results in higher accuracy solution. If the matrix is associated to a physical grid, the physical coordinates of the solution points can be used to improve the quality of partitioning.
\item Here, we used a binary tree corresponding to the recursive bi-partitioning of the graph. Many other options are possible, e.g., using octree if the matrix comes from a three-dimensional problem, \red{or an adaptive tree with arbitrary number of children per node.}
\item We used SVD for low-rank representation of the well-separated \red{nodes}. Other low-rank \red{approximation} methods could be used as well, e.g., randomized SVD \cite{halko2011finding}, randomized block algorithm \cite{voronin2015randomized}, adaptive cross approximation \cite{bebendorf2003adaptive}, rank-revealing QR/LU \cite{chan1987rank}, etc. 
\item Having a low-rank representation of a well-separated block, there are different measures to define the error. For instance, we used the ratio of the singular values to the largest singular value as a measure of accuracy in the low-rank approximation in \cref{eqn:cutOff1}. One can use different criteria, e.g., absolute singular values, ratio of singular values to largest singular value of the whole level or the full matrix, Frobenius norm of the low-rank block, etc.
\item We defined two super-nodes as well-separated \blue{if their distance is greater than 1 (see \cref{cor:dist})}. One can change this definition, and make it \red{stronger}. For example, \red{define} two super-nodes \red{as} well-separated if their distance is at least \red{3}. This is similar to the fill value parameter in the ILU.
\item At each level, we can use any ordering to eliminate super-nodes and black-nodes. Here we used a \blue{generic} ordering. There are various other orderings which reduce the calculation cost, including the minimum degree, minimum deficiency, nested dissection, etc. \cite{davis2006direct}. The complexity of the algorithm remains linear irrespective of the ordering. This is particularly an alluring property for a parallel implementation of the algorithm.
\end{enumerate}
\clearpage
\appendix 
\appendixnotitle
\label{sec:appA}
\red{In this section, we present an example of the factorization process (see \cref{alg:fact}) for one level. Each figure represents one step of the algorithm. The $\mathcal{H}$-tree is shown on the right, and the corresponding extended matrix is shown on the left.}
\begin{figure}[tbhp] \centering
\includegraphics[width=0.32\textwidth]{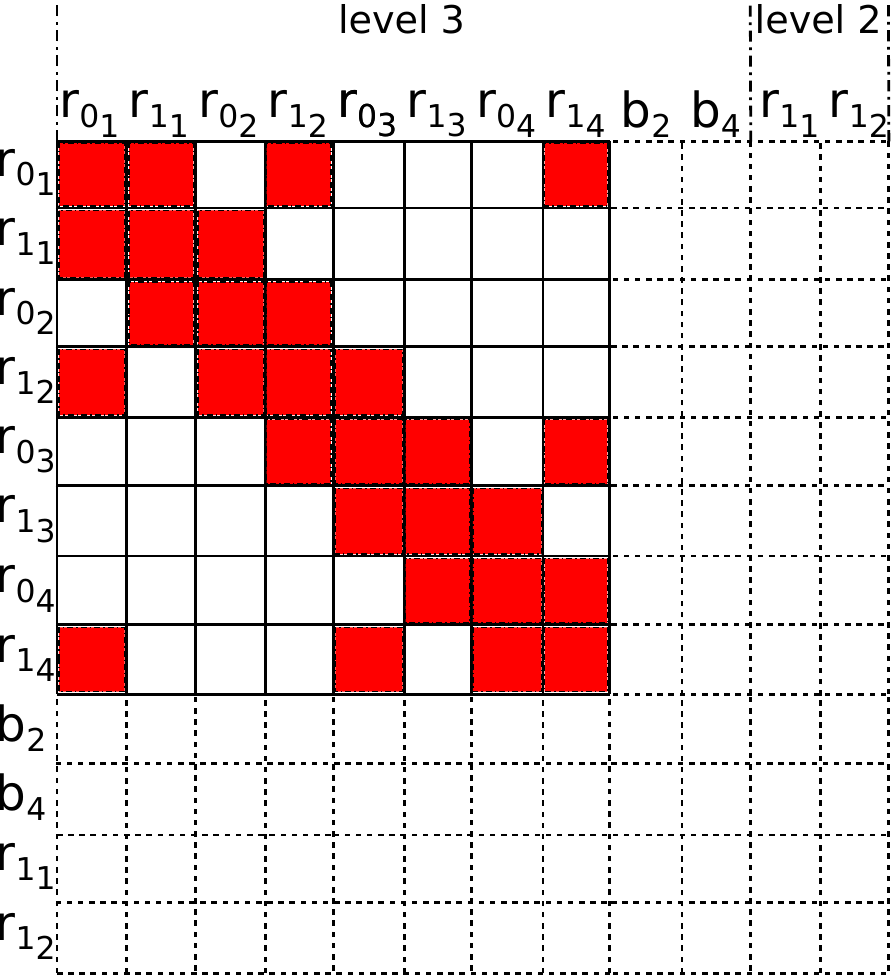}
\hspace{3mm}
\includegraphics[width=0.63\textwidth]{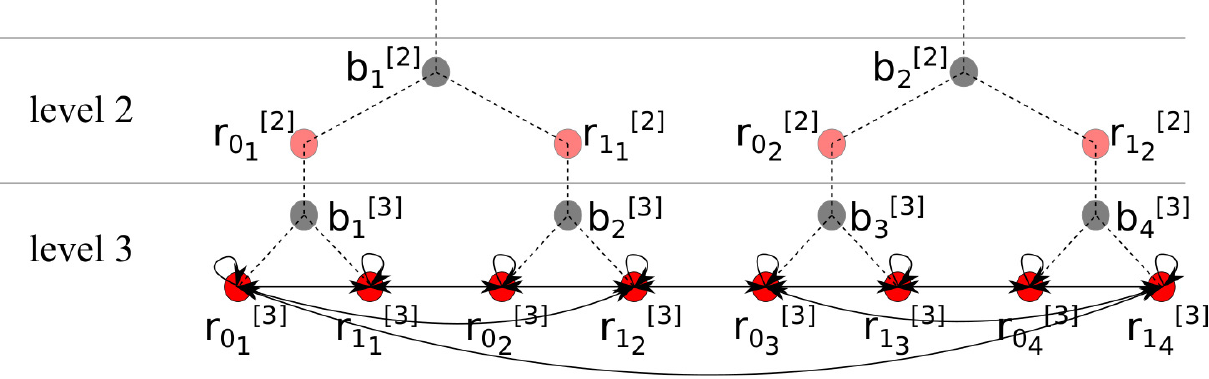}
\caption{Original matrix (left) and the corresponding adjacency graph (right).}
\label{fig:pictorial_alg0}
\end{figure}

\begin{figure}[tbhp] \centering
\includegraphics[width=0.32\textwidth]{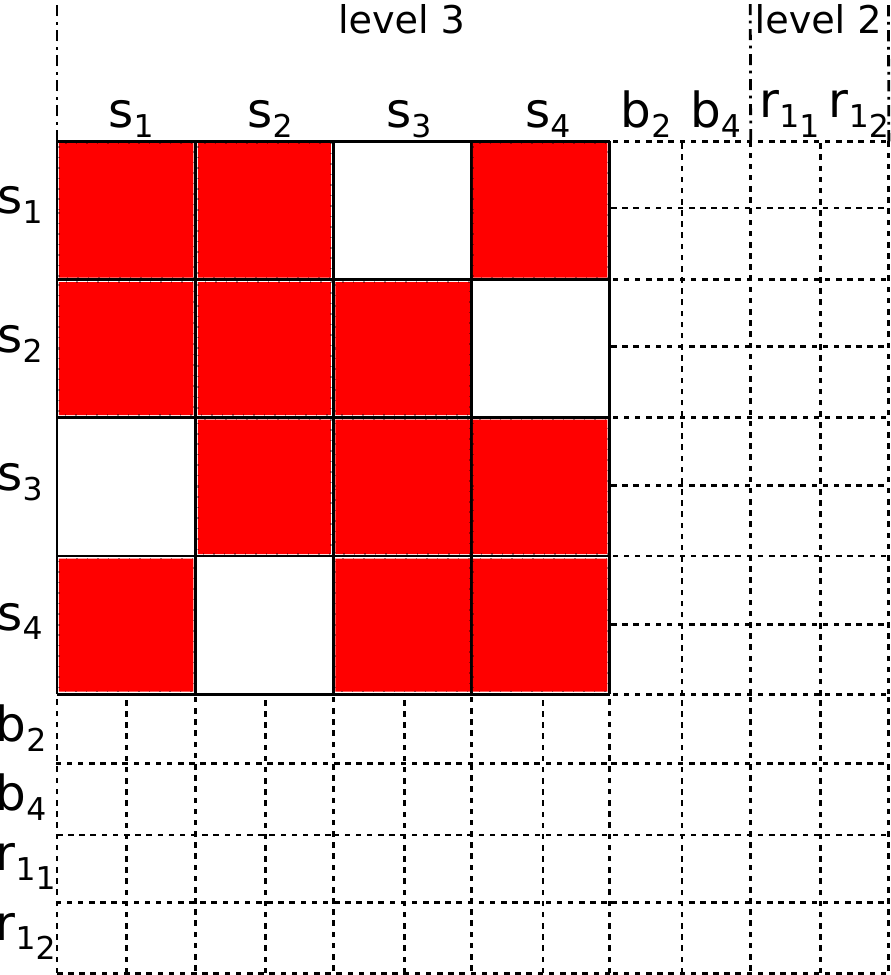}
\hspace{3mm}
\includegraphics[width=0.63\textwidth]{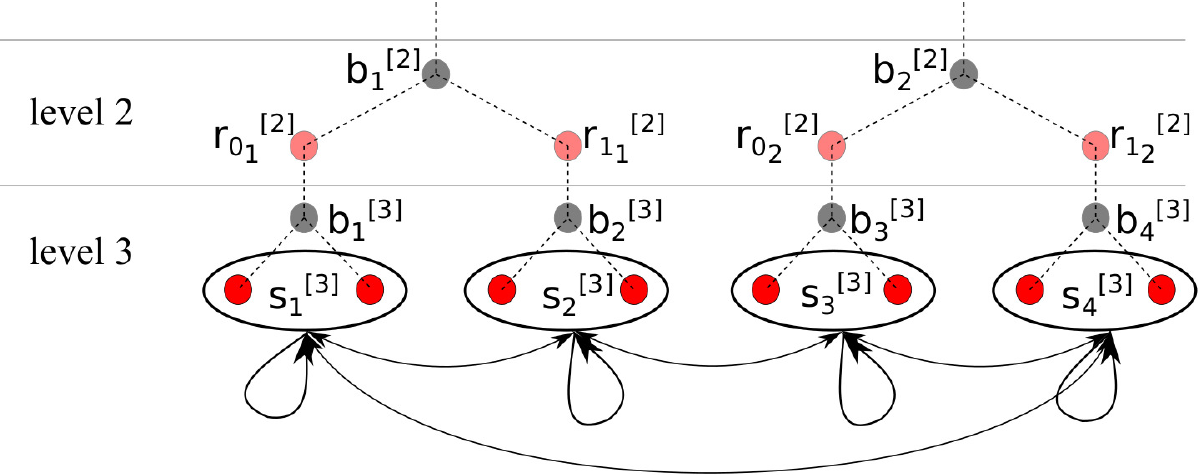}
\caption{Creating super-nodes at level 3 (see \cref{sec:merge}), where the super-node $s_i^{[3]}$ consists of red-nodes ${r_0}_i^{[3]}$ and ${r_1}_i^{[3]}$.}
\label{fig:pictorial_alg1}
\end{figure}

\begin{figure}[tbhp] \centering
\includegraphics[width=0.32\textwidth]{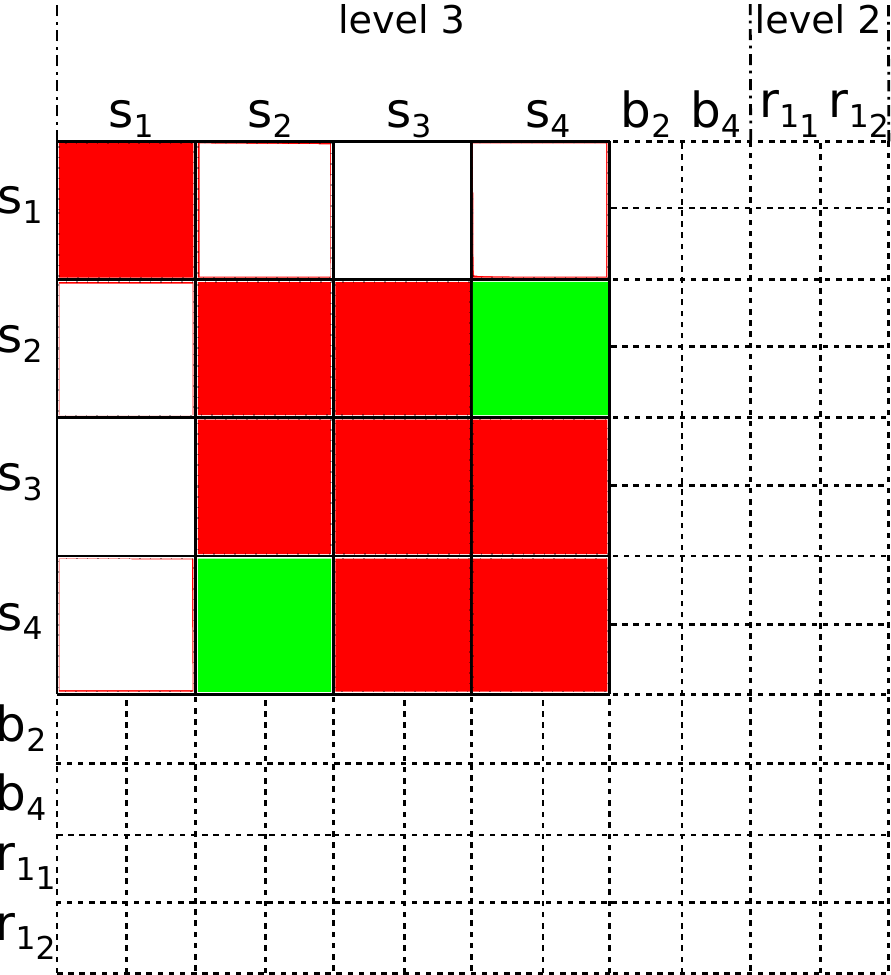}
\hspace{3mm}
\includegraphics[width=0.63\textwidth]{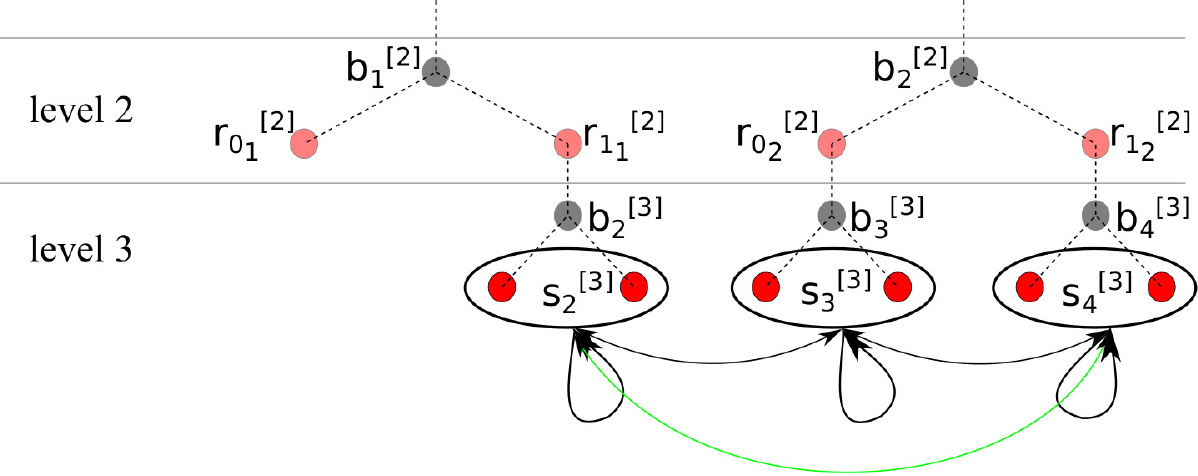}
\caption{Eliminating $s_1^{[3]}$ (see \cref{sec:elimAlg}). Green edges (and their corresponding blocks in the matrix) represent a numerically low-rank interaction between two well-separated nodes to be compressed.}
\label{fig:pictorial_alg2}
\end{figure}

\begin{figure}[tbhp] \centering
\includegraphics[width=0.32\textwidth]{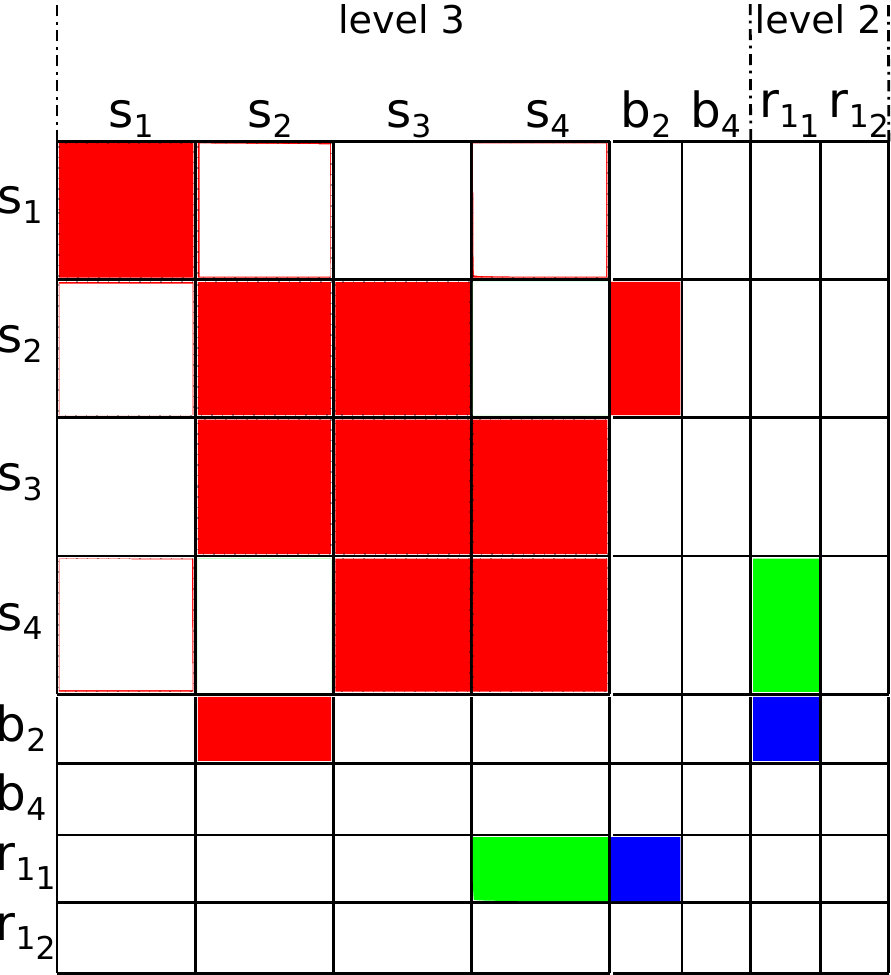}
\hspace{3mm}
\includegraphics[width=0.63\textwidth]{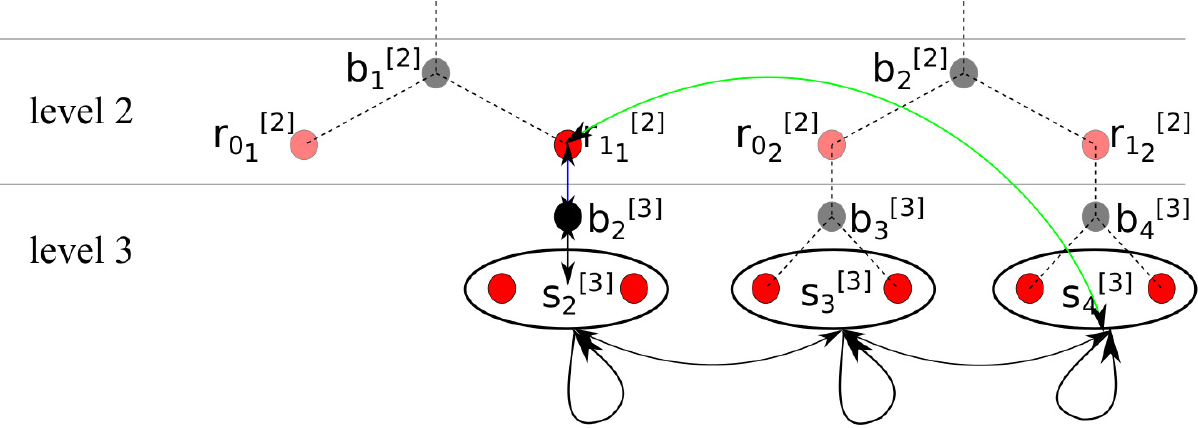}
\caption{Compressing the well-separated edge between $s_2^{[3]}$ and $s_4^{[3]}$ (see \cref{sec:comp}). Blue edges (and their corresponding blocks in the matrix) correspond to minus identity.}
\label{fig:pictorial_alg3}
\end{figure}

\begin{figure}[tbhp] \centering
\includegraphics[width=0.32\textwidth]{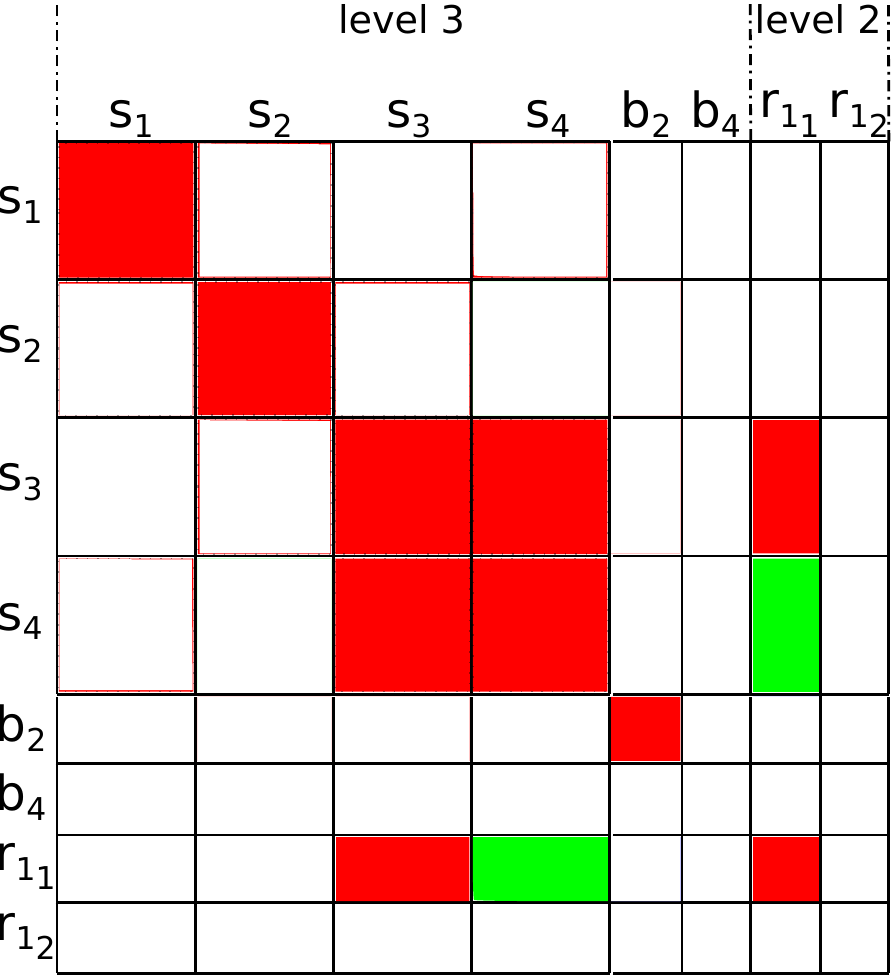}
\hspace{3mm}
\includegraphics[width=0.63\textwidth]{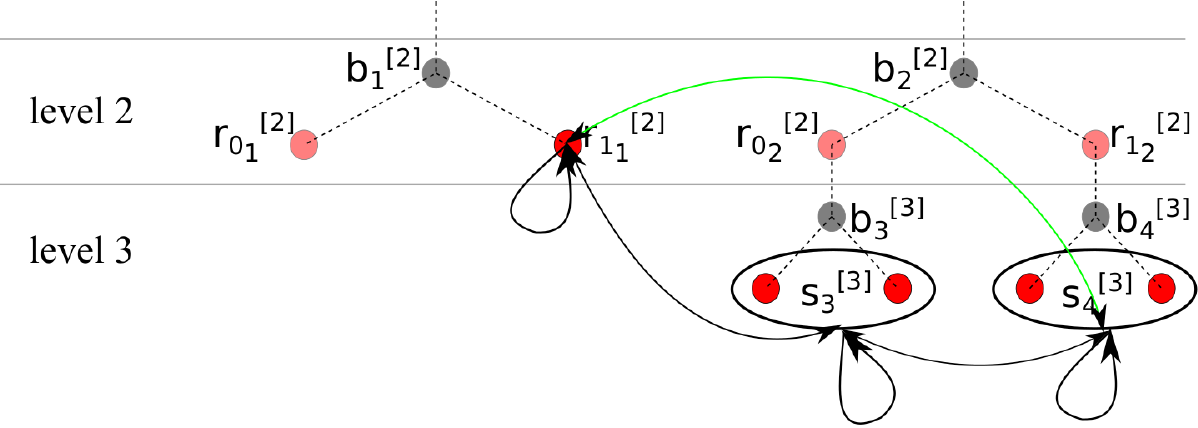}
\caption{Eliminating $s_2^{[3]}$ and $b_2^{[3]}$ (see \cref{sec:elimAlg}).}
\label{fig:pictorial_alg4}
\end{figure}

\begin{figure}[tbhp] \centering
\includegraphics[width=0.32\textwidth]{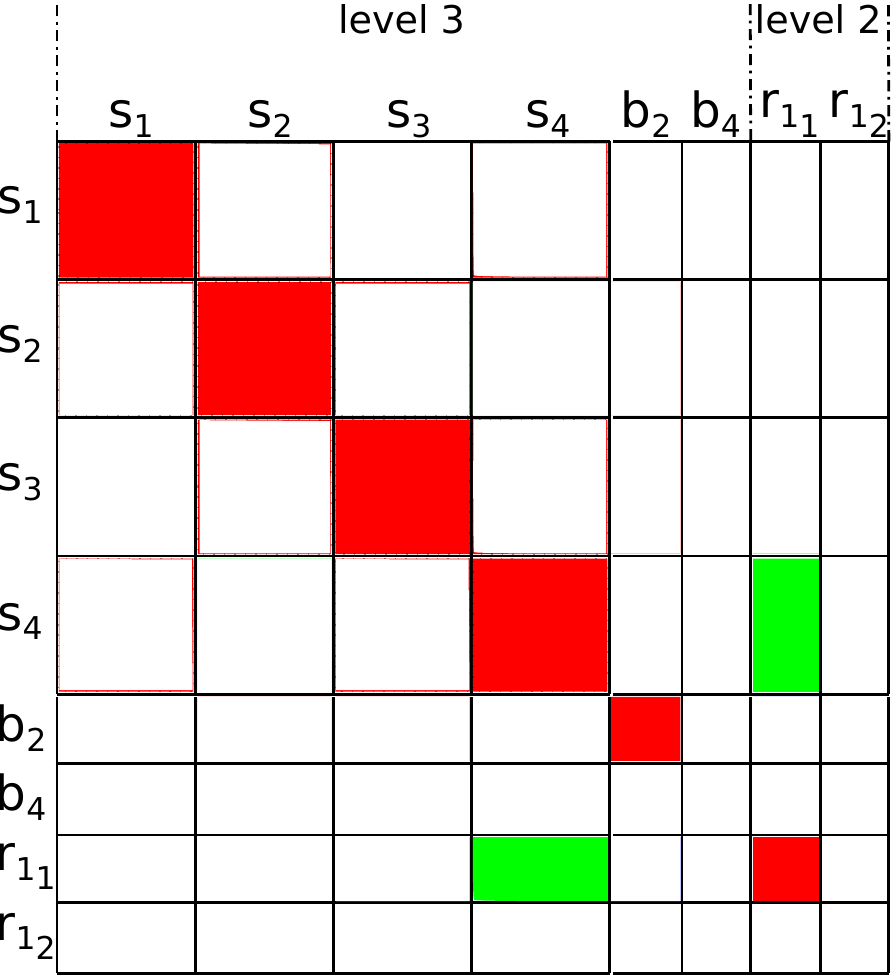}
\hspace{3mm}
\includegraphics[width=0.63\textwidth]{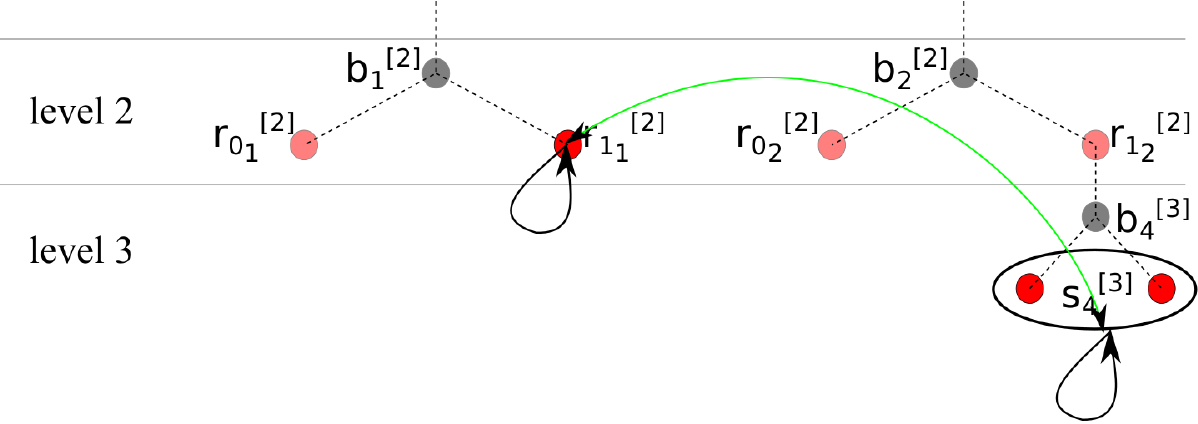}
\caption{Eliminating $s_3^{[3]}$ (see \cref{sec:elimAlg}).}
\label{fig:pictorial_alg5}
\end{figure}

\begin{figure}[tbhp] \centering
\includegraphics[width=0.32\textwidth]{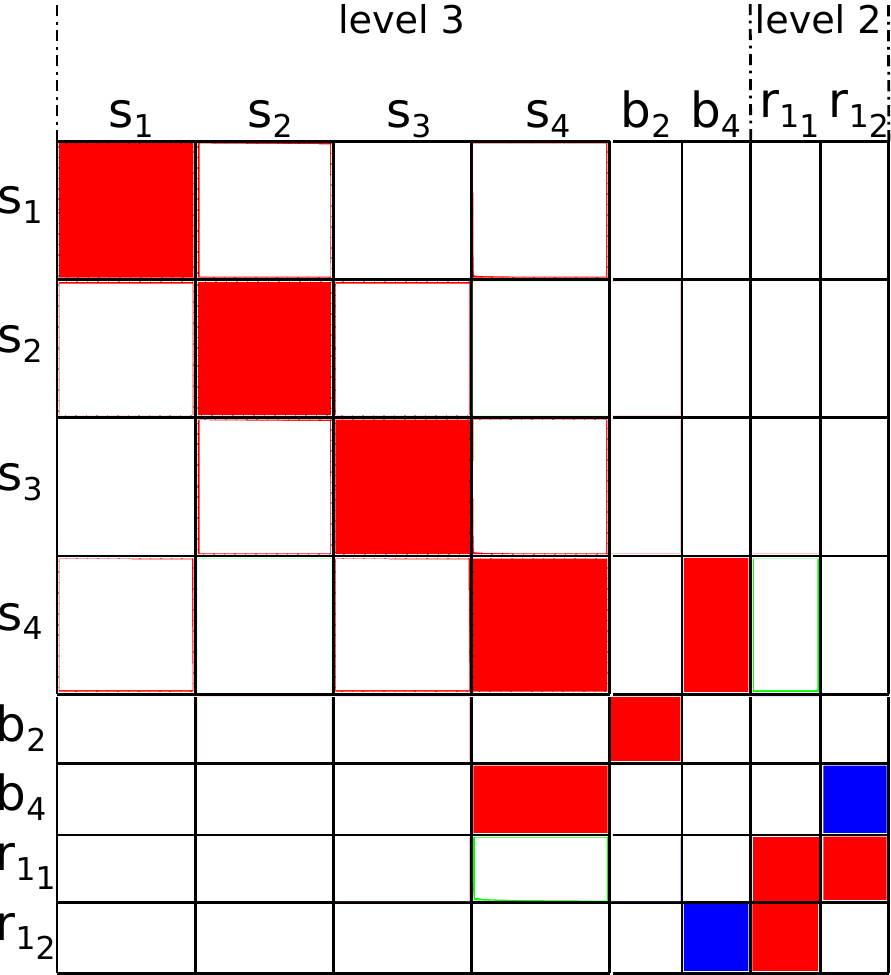}
\hspace{3mm}
\includegraphics[width=0.63\textwidth]{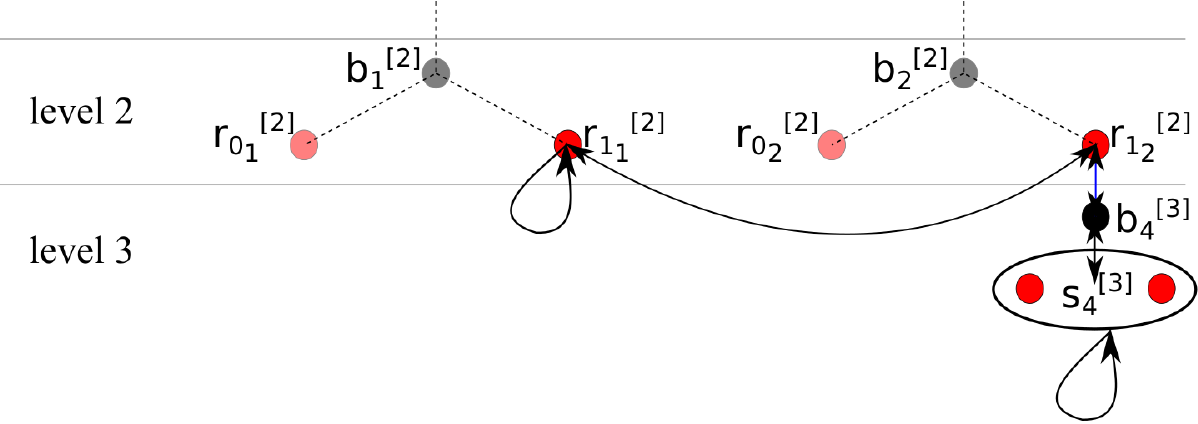}
\caption{Compressing the well-separated edge between ${r_1}_1^{[2]}$ and $s_4^{[3]}$ (see \cref{sec:comp}).}
\label{fig:pictorial_alg6}
\end{figure}

\begin{figure}[tbhp] \centering
\includegraphics[width=0.32\textwidth]{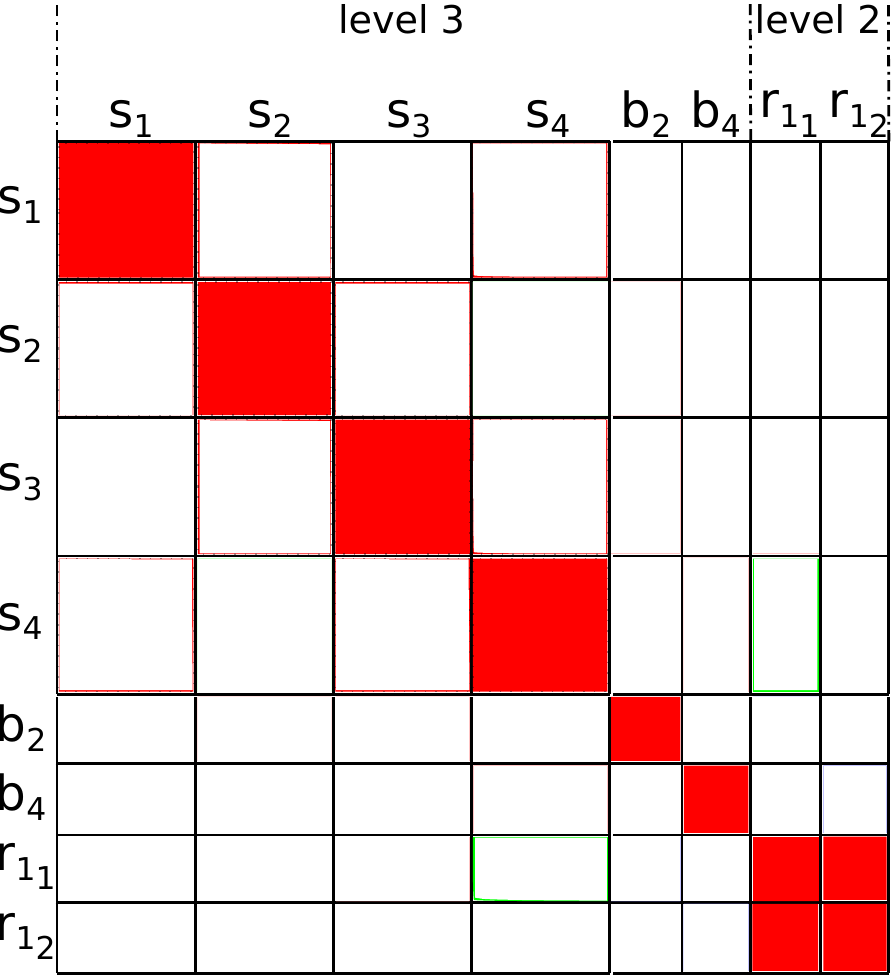}
\hspace{3mm}
\includegraphics[width=0.63\textwidth]{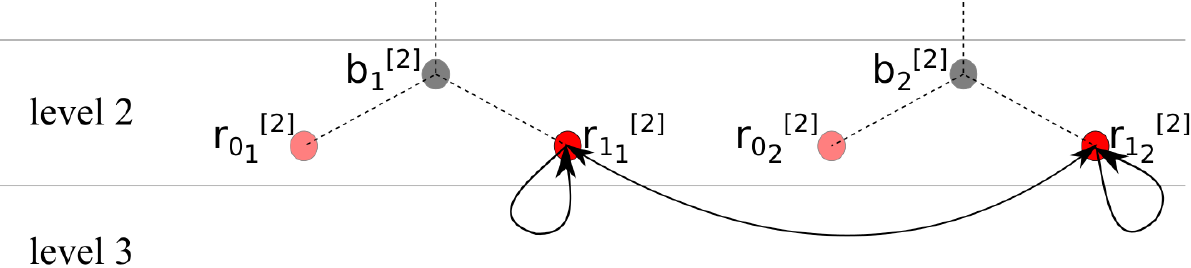}
\caption{Eliminating $s_4^{[3]}$ and $b_4^{[3]}$ (see \cref{sec:elimAlg}).}
\label{fig:pictorial_alg7}
\end{figure}
\clearpage

\appendixnotitle
\label{sec:appB}
\red{In this appendix we provide a graphical example of a nested partitioning using the SCOTCH library for a sparse matrix corresponding to discretization of \cref{eqn:Pois} on a 2D Voronoi grid.}
\begin{figure}[tbhp] \centering
\includegraphics[width=0.4\textwidth]{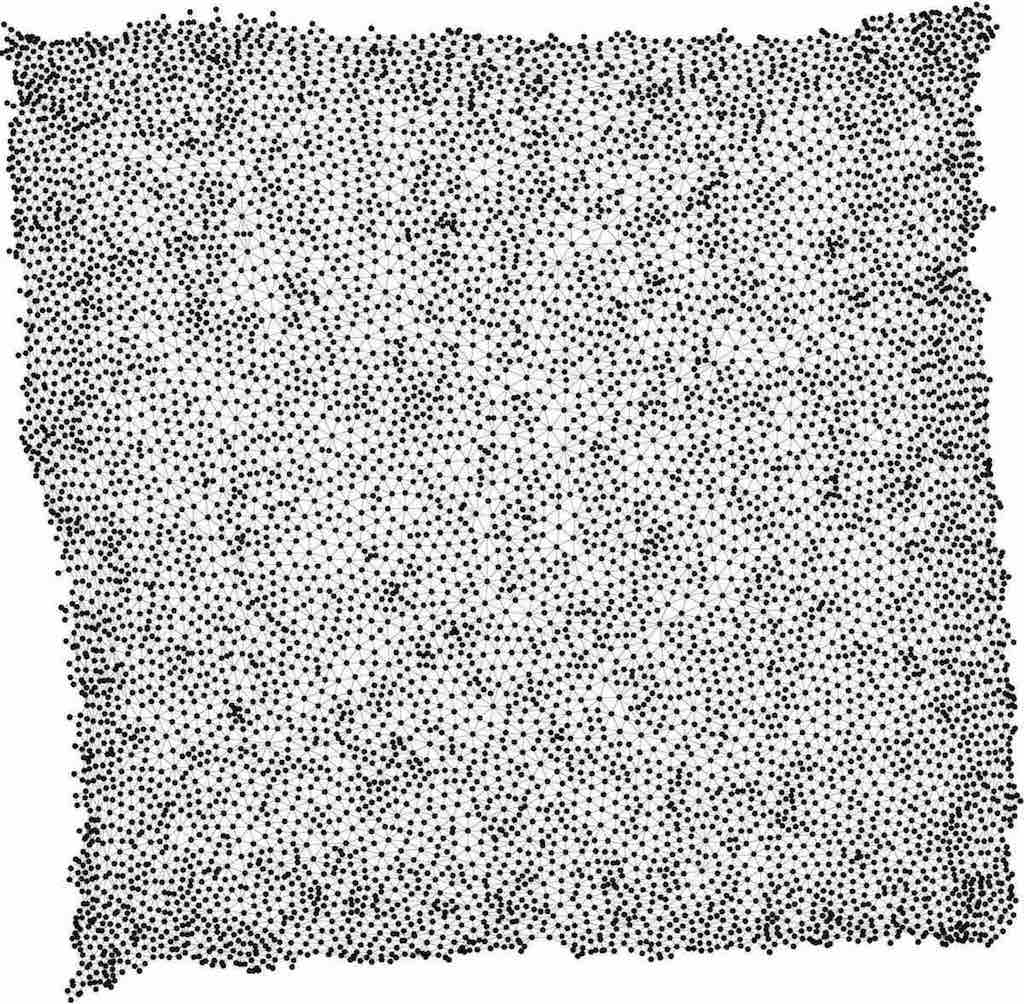}
\hspace{5mm}
\includegraphics[width=0.4\textwidth]{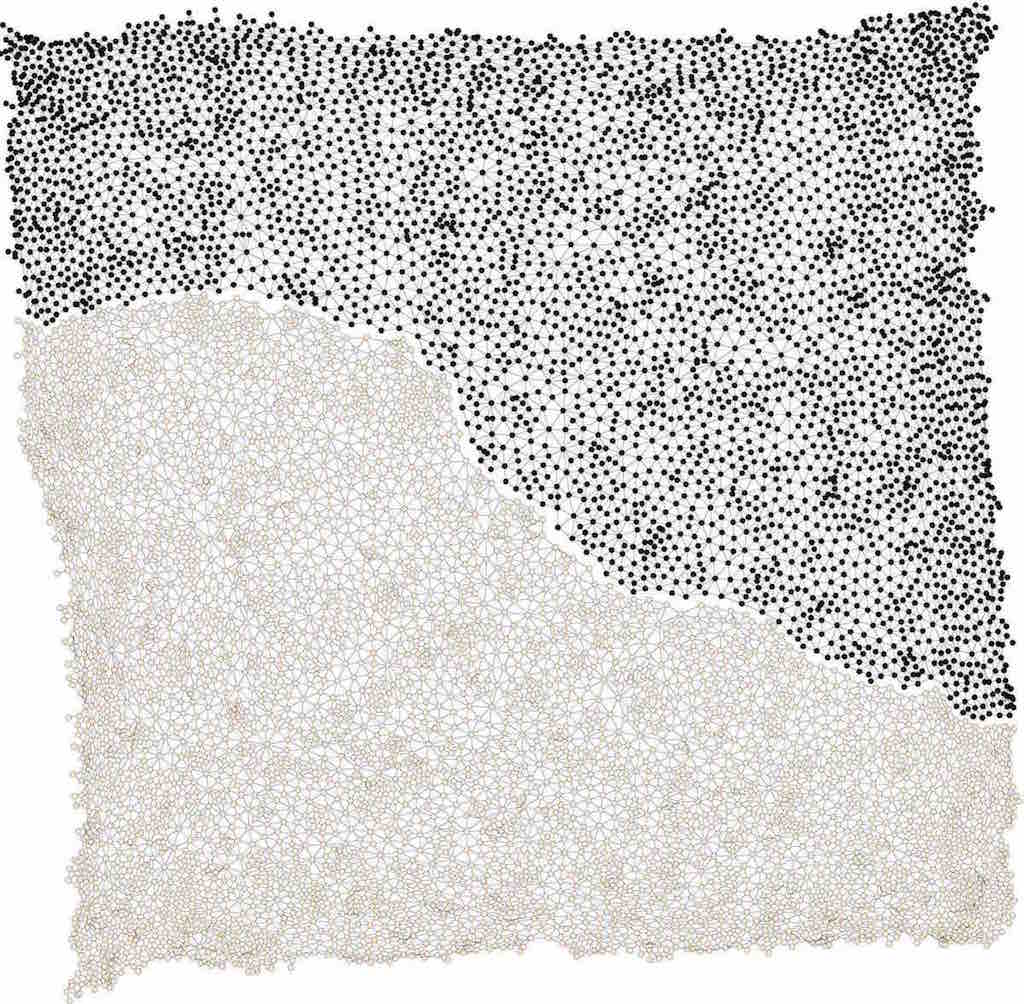}

\includegraphics[width=0.4\textwidth]{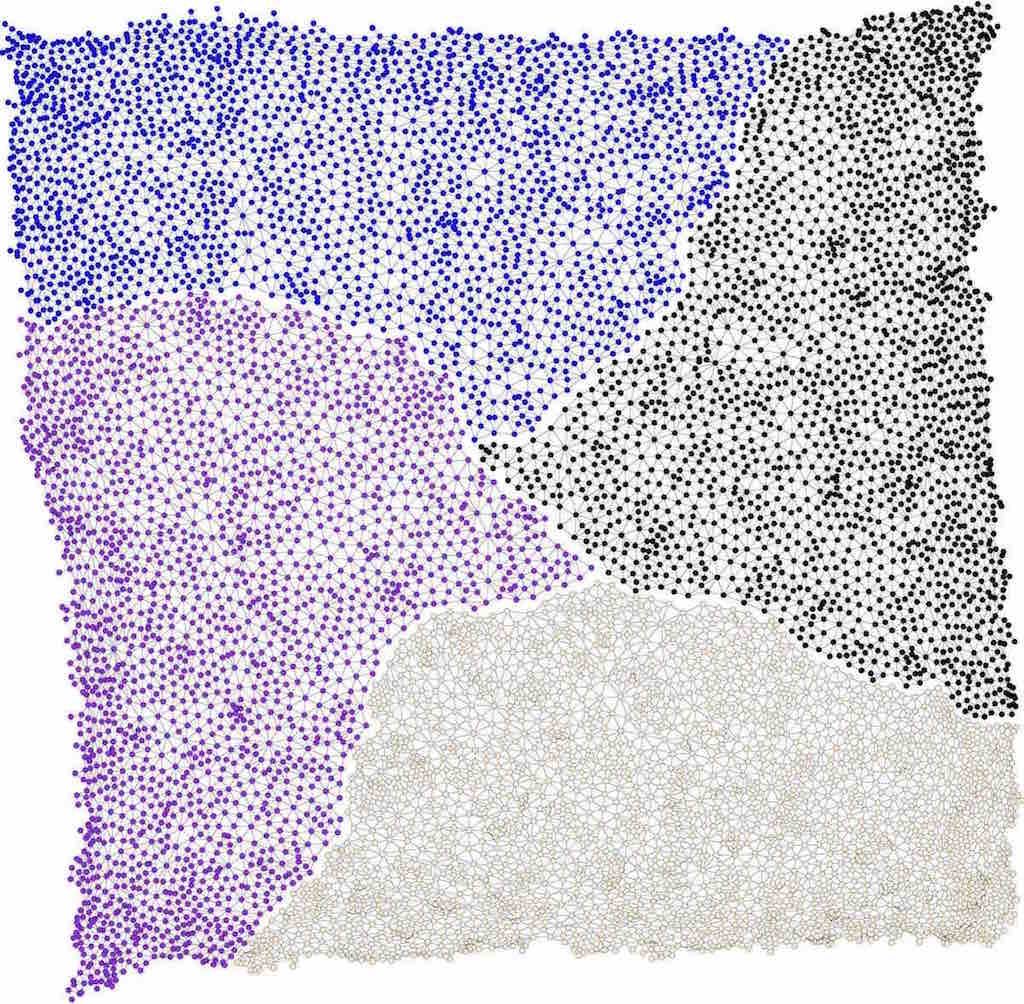}
\hspace{5mm}
\includegraphics[width=0.4\textwidth]{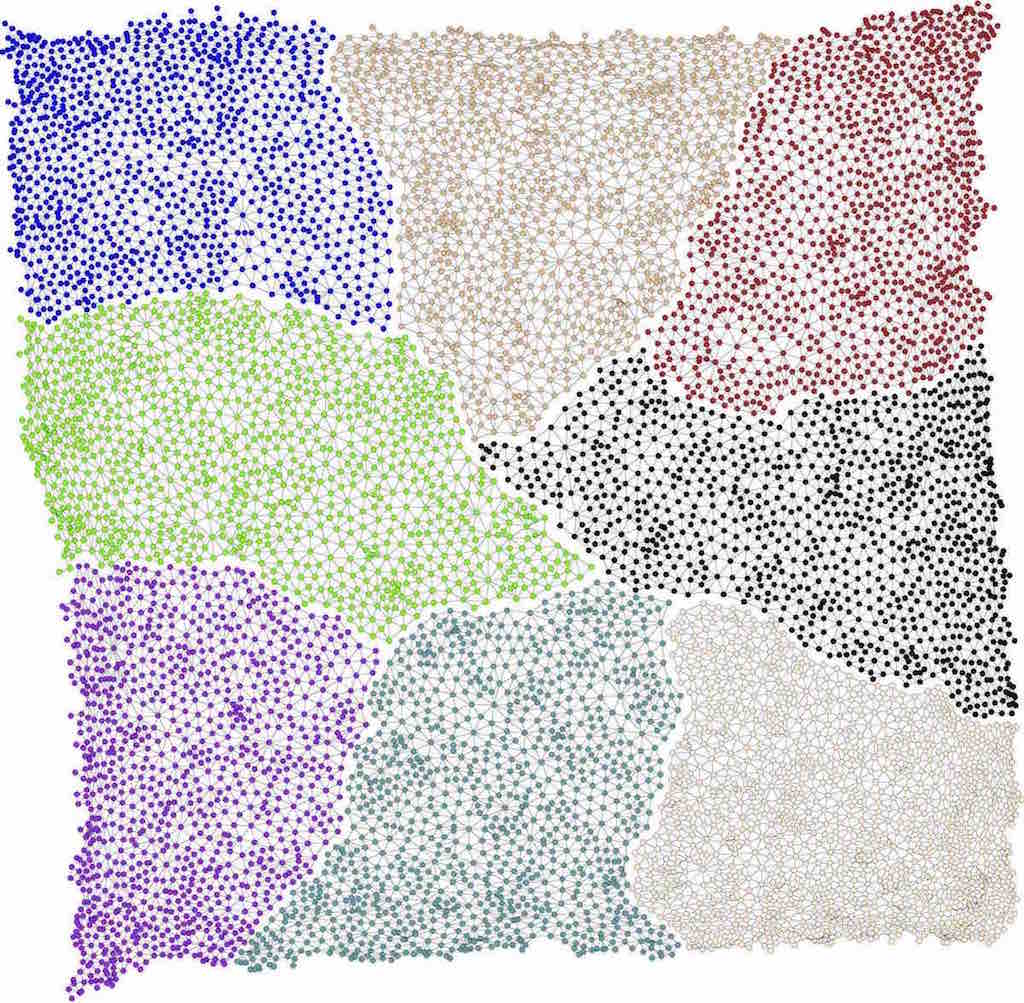}

\includegraphics[width=0.4\textwidth]{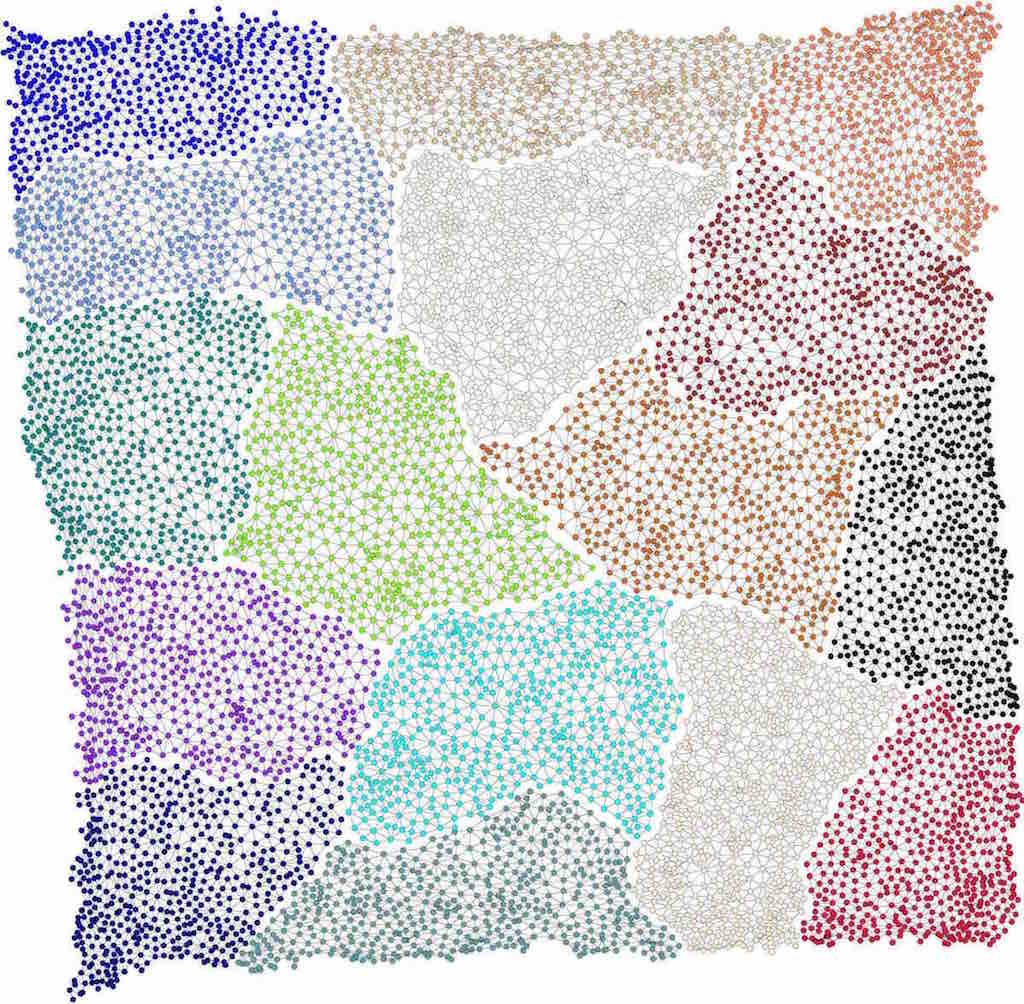}
\hspace{5mm}
\includegraphics[width=0.4\textwidth]{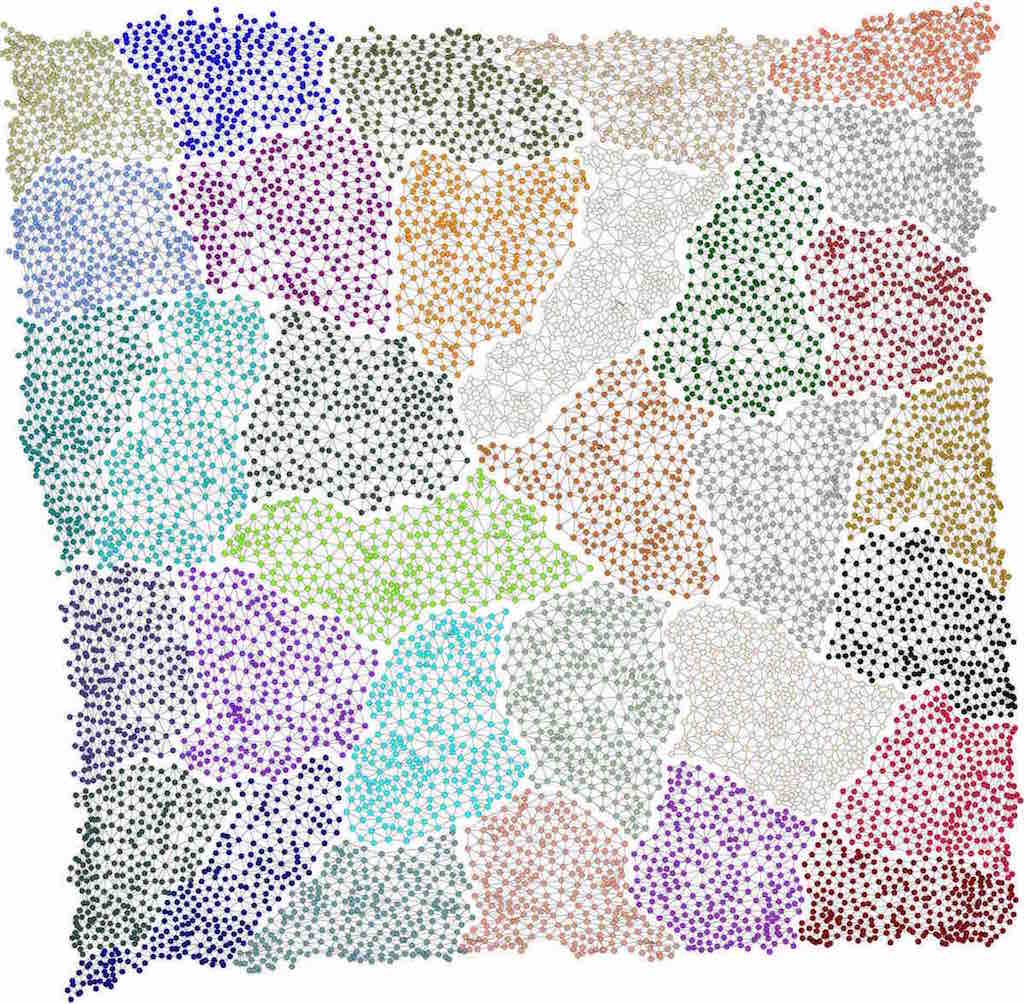}
\caption{An example of 6 levels of a nested partitioning. Clusters are distinguished by different colors. The edges between different clusters are intentionally omitted in this figure for visualization purpose.}
\label{fig:recPart}
\end{figure}
\clearpage

\bibliographystyle{siam}
\bibliography{draft}

\end{document}